\documentclass[12pt,twoside]{amsart}
\usepackage{amssymb}
\usepackage{verbatim}
\usepackage{amsmath}
\usepackage{bm}
\usepackage{a4wide}
\usepackage[latin1]{inputenc}
\usepackage[T1]{fontenc}
\usepackage{times}
\usepackage{amssymb,latexsym}
\usepackage{enumerate}
\usepackage{color}
\usepackage{pict2e}
\usepackage{hyperref}

\makeatletter
\DeclareRobustCommand{\intprod}{%
  \mathbin{\mathpalette\int@prod{(0.1,0)(0.9,0)(0.9,0.8)}}%
}
\DeclareRobustCommand{\intprodr}{%
  \mathbin{\mathpalette\int@prod{(0.1,0.8)(0.1,0)(0.9,0)}}}

\newcommand{\int@prod}[2]{%
  \begingroup
  \sbox\z@{$\m@th#1+$}%
  \setlength\unitlength{\wd\z@}%
  \begin{picture}(1,1)
  \roundcap
  \polyline#2
  \end{picture}%
  \endgroup
}
\makeatother

\makeatletter
\newcommand{\sumprime}{\if@display\sideset{}{'}\sum%
            \else\sum'\fi}
\makeatother

\begin{document}

\numberwithin{equation}{section}

% define theorem environments
\newtheorem{theorem}{Theorem}[section]
\newtheorem{proposition}[theorem]{Proposition}
\newtheorem{conjecture}[theorem]{Conjecture}
\def\theconjecture{\unskip}
\newtheorem{corollary}[theorem]{Corollary}
\newtheorem{lemma}[theorem]{Lemma}
\newtheorem{observation}[theorem]{Observation}
\newtheorem{definition}{Definition}
\numberwithin{definition}{section} %\def\thedefinition{\unskip}
\newtheorem{remark}{Remark}
\def\theremark{\unskip}
\newtheorem{kl}{Key Lemma}
\def\thekl{\unskip}
\newtheorem{question}{Question}
\def\thequestion{\unskip}
\newtheorem{example}{Example}
\def\theexample{\unskip}
\newtheorem{problem}{Problem}

\thanks{The first author is supported by National Natural Science Foundation of China, No. 11771089; the second author is supported by National Natural Science Foundation of China, No. 12071310}

\address{B. Y. Chen: Department of Mathematical Sciences, Fudan University, Shanghai, 200433, China}

\email{boychen@fudan.edu.cn}

\address{L. Zhang: School of Mathematical Sciences, Capital Normal University, Beijing, 100048, China}

\email{zhangly@cnu.edu.cn}

\title[On the $p-$Bergman theory]{On the $p-$Bergman theory}

 \author{Bo-Yong Chen and Liyou Zhang}
 
\date{}

\begin{abstract}
In this paper we attempt to develop a general $p-$Bergman theory on bounded domains in $\mathbb C^n$. To indicate the basic difference between  $L^p$ and  $L^2$ cases, we show that the $p-$Bergman kernel $K_p(z)$ is not real-analytic on some bounded complete Reinhardt domains when $p > 4$ is an even number.  By the calculus of variations we get a fundamental reproducing formula.  This together with certain techniques from nonlinear analysis of the $p-$Laplacian yield a number of results,  e.g.,   the off-diagonal $p-$Bergman kernel $K_p(z,\cdot)$ is H\"older continuous  of order $\frac12$ for $p>1$ and of order $\frac1{2(n+2)}$ for $p=1$.  We also show that the $p-$Bergman metric $B_p(z;X)$ tends to the Carath\'eodory metric $C(z;X)$ as $p\rightarrow \infty$ and the generalized Levi form $i\partial\bar{\partial}\log K_p(z;X)$ is no less than $B_p(z;X)^2$  for $p\ge 2$ and
$ C(z;X)^2$ for $p\le 2.$ Stability of $K_p(z,w)$ or $B_p(z;X)$ as $p$ varies,  boundary behavior of  $K_p(z)$,   as well as basic facts on the $p-$Bergman projection,  are also investigated.
\end{abstract}

\maketitle

\tableofcontents

\section{Introduction}

In the early twentieth  century Stefan Bergman discovered an important link between function theory, geometry and Hilbert space theory, namely the theory of Bergman kernel and Bergman metric.  The Bergman theory makes essential use of complete orthonormal bases in the  Bergman space,  which sheds a particular light on the real difficulty of  extending the Bergman theory to the $L^p$ case.

In this paper we attempt to develop a general $p-$Bergman theory. For a bounded domain $\Omega\subset \mathbb C^n$ we define $A^p(\Omega)$ to be the $p-$Bergman space of $L^p$ holomorphic functions on $\Omega$ (throughout this paper the integrals are with respect to the Lebesgue measure).  We start with a minimizing problem which was also considered by Bergman himself in the  case $p=2$:
\begin{equation}\label{eq:MinProb}
 m_p(z):=\inf\left\{\|f\|_p:f\in A^p(\Omega),f(z)=1\right\}.
 \end{equation}
There exists at least one minimizer for $p>0$ and  exactly one minimizer $m_p(\cdot,z)$ for every $p\ge 1$. We then define the $p-$Bergman kernel by $K_p(z)=m_p(z)^{-p}$ for $p>0$ and the off-diagonal Bergman kernel by $K_p(z,w)=m_p(z,w)K_p(w)$ for $p\ge 1$. Note that $K_2(z)$ and $K_2(z,w)$ are  standard Bergman kernel and off-diagonal Bergman kernel respectively. After the early work of   Narasimhan-Simha \cite{NS} and  Sakai \cite{Sakai}, the study of $K_p(z)$  has attracted much attention in recent years (see e.g., \cite{Siu}, \cite{Tsuji06}, \cite{Tsuji07}, \cite{ChenInvariant}, \cite{BP}, \cite{BPK}, \cite{Yau}, \cite{Tsuji}, \cite{NZZ}, \cite{Taka},  \cite{DWZZ}).
 
 Our first result will indicate the basic difference between $K_p$ and $K_2$ when $p>2$.  
 
 \begin{theorem}\label{th:NRA_0}
 Let $\Omega$ be a bounded complete Reinhardt domain in $\mathbb C^n$. Then the following properties hold:
 \begin{enumerate}
 \item[$(1)$]  If\/ $K_2(z,w)$ is not zero-free, then there exists $k_0\in \mathbb Z^+$ such that $K_{2k}(z)$ is not real-analytic on $\Omega$\/ for any integer $k\ge k_0$. 
 \item[$(2)$] Suppose  there exist $\zeta_0,z_0\in \Omega$ such that $K_2(\zeta_0,z_0)=0$ and ${\rm ord}_{z_0} K_2(\zeta_0,\cdot)=1$. Then $K_{2k}(z)$ is not real-analytic on $\Omega$\/ for any integer $k > 2$. Moreover, either ${\rm Re\,}m_p(\zeta_0,\cdot)$ or ${\rm Im\,}m_p(\zeta_0,\cdot)$  is not real-analytic on $\Omega$\/ for any rational $p>2$.
  \end{enumerate} 
 \end{theorem}
 
 In \cite{Lu}, Lu Qi-Keng asked: for which domains  is $K_2(z,w)$  zero-free? It turns out that for most pseudoconvex domains $K_2(z,w)$ is not zero-free; among them  most explicit examples  are bounded pseudoconvex complete Reinhardt domains (cf. \cite{Boas} and \cite{JP}). We will verify that  the Thullen-type domain $\{(z_1,z_2)\in \mathbb C^2: |z_1|+|z_2|^{2/\alpha}<1\}$ for $\alpha>2$ satisfies the hypothesis in Theorem \ref{th:NRA_0}/(2) by using the calculation in Boas-Fu-Straube \cite{BFS}.
   
 Note that $m_p(\cdot,z)$ and $K_p(\cdot,z)$ are holomorphic on $\Omega$ for fixed $z$. On the other hand, the function theory of $m_p(z,\cdot)$ or $K_p(z,\cdot)$ for fixed $z$ is completely mysterious.   Using the calculus of variations, we get the following fundamental reproducing formula
 \begin{equation}\label{eq:RPF}
 f(z) =  \int_\Omega |m_p(\cdot,z)|^{p-2}\,\overline{K_p(\cdot,z)}\,f,\ \ \ \forall\,f\in A^p(\Omega).
  \end{equation}
 This formula implies a geometric interpretation of $m_p(z)$ that it is precisely  the distance between two hyperplanes $H^c_{p,z}$ and $H^{c+1}_{p,z}$,  where
$$
H^c_{p,z}=\left\{f\in A^p(\Omega): f(z)=c\right\},\ \ \ c\in \mathbb C,
$$  
(see  Proposition \ref{prop:geometric}). 
The nonlinear factor $ |m_p(\cdot,z)|^{p-2}$ in \eqref{eq:RPF} causes the real difficulity for applications. Thus the reproducing formula is of limited use without the help of some techniques from nonlinear analysis of the $p-$Laplacian (cf. \cite{Lind1}). This also indicates the major difference to the Bergman theory.

 It is fairly easy to show that $m_p(z)$ and $K_p(z)$ are locally Lipschitz continuous. However, the regularity problem for $m_p(z,\cdot)$ or $K_p(z,\cdot)$ is more difficult. Regularity in a minimizing problem is classic and goes back to Hilbert's  famous problem-list.
 
 \begin{theorem}\label{th:RegHolder}
\begin{enumerate}
\item[$(1)$]  For any  $p >1 $ and any  open set $U\subset\subset \Omega$,  there exists a constant $C>0$ such that
$$
|m_p(z,w)-m_p(z,w')|\le C|w-w'|^{\frac12},\ \ \ \forall\,z,w,w'\in U.
$$
\item[$(2)$] 
For every open set $U\subset\subset \Omega$,  there exists a constant $C>0$ such that
$$
|m_1(z,w)-m_1(z,w')|\le C|w-w'|^{\frac1{2(n+2)}},\ \ \ \forall\,z,w,w'\in U.
$$
Moreover,  
for every open set $U$ with $w\in U\subset\subset \Omega\backslash S_w$,  where $S_w:=\{m_1(\cdot,w)=0\}$,  there exists a constant $C>0$ such that
$$
|m_1(z,w)-m_1(z,w')|\le C|w-w'|^{\frac12},\ \ \ \forall\,z,w'\in U.
$$
\end{enumerate}
The same conclusions also hold for $K_p$.
\end{theorem}

\begin{remark}
$(1)$ It is interesting to point out that Theorem \ref{th:RegHolder}$/(1)$ plays an essential role in the proof of non real-analyticity of $K_{2k}(z)$ in Theorem \ref{th:NRA_0}.

$(2)$ In a future paper,  Theorem \ref{th:RegHolder}$/(1)$ will be applied to show that $K_p(z)$ is of $C^{1,1/2}$ on $\Omega$ for $p>1$.
\end{remark}

Note that the Cauchy-Schwarz inequality gives  
$$
|K_2(z,w)|\le K_2(z)^{\frac12}\,K_2(w)^{\frac12}\ \ \  \text{and}\ \ \  2{\rm Re\,}K_2(z,w)\le K_2(z)+K_2(w).
$$
 Surprisingly, these inequalities remain valid for general $p\ge 1$.
 
 \begin{theorem}
 \begin{enumerate}
 \item[$(1)$] $
|K_p(z,w)|\le K_p(z)^{\frac1p}\,K_p(w)^{\frac1q},
$
where $1/p+1/q=1$.
\item[$(2)$]  
 $
{\rm Re}\left\{K_p(z,w)+K_p(w,z)\right\}\le K_p(z) + K_p(w).
$
\end{enumerate}
 Each equality holds if and only if $z=w$.
   \end{theorem}
   
   \begin{remark}
   In particular, we have $|K_1(z,w)|\le K_1(z)$, so that $K_1(z,\cdot)$ is a bounded function on $\Omega$ for fixed $z$.
   \end{remark}
   
We also investigate the $p-$Bergman metric given by
$$
B_p(z;X)  :=  {K_p(z)^{-\frac1p}}\cdot {\sup}_f\ |Xf(z)|
$$
where the supremum is taken over all $f\in A^p(\Omega)$ with $f(z)=0$ and $\|f\|_p=1$. It is easy to see that $B_2(z;X)$ is the standard Bergman metric. The $p-$Bergman metric is an invariant (Finsler) metric for\/ {\it simply-connected\/} bounded domains, and is always no less than the Carath\'eodory metric $C(z;X)$ (the case $p=2$ goes back to Lu Qi-Keng \cite{Lu}; see also \cite{Hahn}).  More interestingly, we have 

\begin{proposition}
 $B_p(z;X)\rightarrow C(z;X)$   as $p\rightarrow \infty$. 
 \end{proposition}
 
For a real-valued upper semicontinuous function $u$ defined on a domain $\Omega\subset \mathbb C^n$,  we define the\/ {\it generalized Levi form} of $u$ by
$$
i\partial\bar{\partial} u(z;X):=\liminf_{r\rightarrow 0+}\frac1{r^2}\left\{\frac1{2\pi}\int_0^{2\pi}u(z+re^{i\theta}X)d\theta-u(z)\right\}.
$$
  A natural question is to find the relationship between $i\partial\bar{\partial} \log K_p(z;X)$  and $B_p(z;X)$. Using the  variation method, we are able to verify the following 

\begin{theorem}  
$$
i\partial\bar{\partial} \log K_p(z;X)\ge \left\{
\begin{array}{cl}
B_p(z;X)^2 & \text{for\ \ }p\ge 2\\
C(z;X)^2 &   \text{for\ \ } p\le 2.
 \end{array}
 \right.
 $$
 \end{theorem}
 
 \begin{remark}
 In particular,  $\log K_p(z)$ is a (continuous) strictly psh function. 
 \end{remark}
  
In \cite{Yau}, Yau suggested to  investigate the the relationship between  the $p$-Bergman metrics when $p$ changes. Motivated by the spectrum theory of the $p-$Laplacian (cf. \cite{Lind}), we will show 
 
 \begin{theorem}
\begin{enumerate} 
\item[$(1)$] 
$
\lim_{s\rightarrow p-} m_s(z,w) = m_p(z,w)
$
for $p>1$ and
$
\lim_{s\rightarrow p+} m_s(z,w)
$
exists for $p\ge 1$. Moreover, if $A^{p'}(\Omega)$ lies dense in $A^p(\Omega)$ for some $p'>p$, then
$$
m_p(z,w)=\lim_{s\rightarrow p+} m_s(z,w).
$$
\item[$(2)$]  $\lim_{s\rightarrow p\pm} B_s(z;X)$ exist for $p>0$ and
$
 B_p(z;X) = \lim_{s\rightarrow p-} B_s(z;X).
$
Moreover, if there exists  $p'>p$ such that $A^{p'}(\Omega)$ lies dense in $A^p(\Omega)$, then
$$
B_p(z;X)=\lim_{s\rightarrow p} B_s(z;X).
$$
Conclusion $(1)$ also holds for $K_p$.
\end{enumerate}
\end{theorem}

On the other hand, we have

\begin{proposition}\label{prop:NonCont}
Let $\Omega=D\backslash S$,  where $D$ is a bounded domain in $\mathbb C$ and $S$ is a compact set in $D$ which has positive $2-$capacity but zero $p-$capacity for every $p<2$. Then 
$$
K_2(z)>\lim_{p\rightarrow 2+} K_p(z).
$$
\end{proposition}

Recall that the $p-$capacity of $S$ is given by
$
{\rm Cap}_p(S):=\inf_\phi \int_{\mathbb C} |\nabla \phi|^p,
$
where the infimum is taken over all $\phi\in C_0^\infty(\mathbb C)$ such that $\phi\ge 1$ on $S$. The condition of Proposition \ref{prop:NonCont} is satisfied for instance, if the $h-$Hausdorff measure $\Lambda_h(S)$ of $S$ is positive and finite where $h(t)=(\log1/t)^{-\alpha}$ for some $\alpha>1$. 

Finally,  we shall discuss the boundary behavior of $K_p(z)$.  As a consequence of a classic theorem of Glothendieck (cf. \cite{Grothendieck},  see also \cite{Rudin}),  we have the following general result:
\begin{theorem}\label{p-kernel estimate}
	 Let  $ 1\le \eta \in L^1(\Omega)$.  For every complex line $L$ in $\mathbb C^n$ there exists $z_L\in \partial \pi_L(\Omega) $,  where $\pi_L$ denotes the projection from $\mathbb C^n$ to $L$,   such that for any $w\in \partial \Omega\cap \pi_L^{-1}(z_L)$,
	\begin{equation}\label{eq:upperlimit}
	\limsup_{z\rightarrow w}  \frac{K_p(z)}{\eta(z)}=+\infty.
	\end{equation}
\end{theorem}

\begin{remark}
If $L$ is rotated around the origin,  then we see that there are uncountably infinite boundary points such that \eqref{eq:upperlimit} holds.  It seems interesting to characterize all such points. 
\end{remark}

As  $K_p(z)$ is a continuous plurisubharmonic function on $\Omega$,  we see that  if $K_p(z)$ is an exhaustion function for some $p$,  then $\Omega$ has to be pseudoconvex.   Due to the failure of $L^p-$estimates for  $\bar{\partial}$ on general pseudoconvex domains when $p>2$ (compare \cite{FS}),  it is plausible to study the boundary behavior of $K_p(z)$ through comparison with  $K_2(z)$. With the help of  $L^2$ estimates for $\bar{\partial}$, we are able to show the following

\begin{theorem}
Let $\Omega$ be a bounded pseudoconvex domain with $C^2-$boundary and $\delta$ the boundary distance.  Then the following properties hold:
\begin{enumerate}
\item[$(1)$] There exist  constants $\gamma,C>0$ such that the following estimates hold near $\partial \Omega:$
\begin{eqnarray*}
{K_p(z)^{\frac1p}}/{K_2(z)^{\frac12}} & \le & C\, \delta(z)^{\frac12-\frac1p} |\log \delta(z)|^{\frac{n(p-2)}{2p\gamma}},\ \ \   p\ge 2,\\
{K_p(z)^{\frac1p}}/{K_2(z)^{\frac12}} & \ge & C^{-1}\, \delta(z)^{\frac12-\frac1p} |\log \delta(z)|^{-\frac{(n+\gamma)(p-2)}{2p\gamma}},\ \ \   p\le 2.
\end{eqnarray*}
\item[$(2)$] For every $2\le p<2+\frac2n$  there exists a  constant $C=C_{p,\Omega}>0$ such that the following estimate holds near $\partial \Omega:$
$$
{K_p(z)^{\frac1p}}/{K_2(z)^{\frac12}} \ge C^{-1}\, \delta(z)^{\frac{(n+1)(p-2)}{2p}} |\log \delta(z)|^{-\frac{(n+1)(p-2)}{2p\gamma}}.
$$
\end{enumerate}
\end{theorem}

It is a straightforward consequence of the Ohsawa-Takegoshi extension theorem \cite{OT} that $K_2(z)\gtrsim \delta^{-2}$ holds for bounded pseudoconvex domains with $C^2-$boundary. Thus

\begin{corollary}
If $\Omega$ is a bounded pseudoconvex domain with $C^2-$boundary,     $K_p(z)$ is an exhaustion function       for every $2\le p<2+\frac2n$.
        \end{corollary}
        
        \begin{remark}
    It was shown in  \cite{NZZ} that  $K_p(z)$ is  an exhaustion function for any $0<p<2$ and any bounded pseudoconvex domain. 
        \end{remark}
        
It is reasonable to ask the following
                  
\begin{problem}
  Is  $K_p(z)$  an exhaustion function  for any $p>2$?
\end{problem} 

\begin{remark}
The answer is affirmative when $\Omega$ is a simply-connected uniformly squeenzing domain or smooth strictly pseudoconvex domain (cf. \cite{DWZZ}). 
\end{remark}

We also define the $p-$Bergman projection  to be the metric projection  $P_p:L^p(\Omega)\rightarrow A^p(\Omega)$ for $1<p<\infty$.  It is\/ {\it nonlinear} in general for $p\neq 2$.  We show that  $P_p$ is locally H\"older continuous of order $1/p$ for $2< p<\infty$ and locally H\"older continuous of order $1/2$ for $1<p\le 2$; for any $f\in L^\infty(\Omega)$,  $\{P_p(f)\}_p$ contains a subsequence converging locally uniformly to some $h_\infty\in A^\infty(\Omega)$ which is an element of best approximation of $f$ in $L^\infty(\Omega)$.  Related questions concerning bases in $A^p(\Omega)$ are also investigated.     

 Most results in this paper extend to the $p-$Bergman space related to  Hermitian line bundles over complex manifolds. Nevertheless, we stick to the simplest case of bounded domains with trivial line bundle  in order to make the arguments as transparent as possible.

\section{Definitions and basic properties}

\subsection{The $p-$Bergman space}

For a domain $\Omega\subset\subset \mathbb C^n$ we define the $p-$Bergman space to be
$$
A^p(\Omega):=\left\{f\in \mathcal O(\Omega): \|f\|_p^p:=\int_\Omega |f|^p<\infty\right\}.
$$

\begin{proposition}[Bergman inequality]
For any compact set $S\subset \Omega$ there exists a constant $C_S>0$ such that
\begin{equation}\label{eq:BergIneq}
\sup_{S} |f|^p\le C_S \|f\|_p^p.
\end{equation}
\end{proposition}

\begin{proof}
Set $r=d(S,\partial \Omega)$. For any $z\in S$, we have $P(z,r/n):=\prod_{j=1}^n\Delta(z_j,r/n)\subset \Omega$. It follows directly from the mean-value inequality of plurisubharmonic functions that
\begin{equation}\label{eq:BergIneq_2}
|f(z)|^p \le \frac1{|P(z,r/n)|}\int_{P(z,r/n)} |f|^p \le \frac1{\pi^n(r/n)^{2n}} \|f\|_p^p.
\end{equation}
\end{proof}

\begin{proposition}\label{prop:Banach}
$A^p(\Omega)$ is a Banach space for $p\ge 1$.
\end{proposition}

\begin{proof}
It suffices to verify that $A^p(\Omega)$ is a\/ {\it closed} subspace in $L^p(\Omega)$. Let $\{f_j\}\subset A^p(\Omega)$ satisfy $f_j\rightarrow f_0$ in $L^p(\Omega)$. By (\ref{eq:BergIneq}) we known that $\{f_j\}$ forms a normal family so that there exists a subsequence $f_{j_k}$ converging locally uniformly to some $\widehat{f}_0\in \mathcal O(\Omega)$. Fatou's lemma yields
$$
\|f_{j_k}-\widehat{f}_0\|_p \le \liminf_{m\rightarrow \infty} \|f_{j_k}-f_{j_m}\|_p\le \liminf_{m\rightarrow \infty} \left[ \|f_{j_k}-f_{0}\|_p+\|f_{j_m}-f_{0}\|_p\right]=\|f_{j_k}-f_{0}\|_p,
$$
which implies that $\widehat{f}_0\in A^p(\Omega)$ and $f_{j_k}\rightarrow \widehat{f}_0$ in $L^p(\Omega)$. Thus $f_0=\widehat{f}_0$ holds a.e. on $\Omega$.
\end{proof}

 Analogously, one can show that $A^p(\Omega)$ is a complete metric space for $0<p<1$, where the metric is given by $d(f_1,f_2):=\|f_1-f_2\|_p^p$.  By the Bergman inequality,  we see that for every fixed $z\in \Omega$ the linear functional $A^p(\Omega)\ni f\mapsto f(z)$ is continuous,  while $(L^p(\Omega))^\ast=\{0\}$ for $0<p<1$ (cf.  \cite{Day}).  Thus the Hahn-Banach extension property fails for the pair $A^p(\Omega)\subset L^p(\Omega)$,  where $0<p<1$.  
 
Recall that a Banach space $X$ is said to be \/{\it strictly convex} if for any $x,y\in X$ with $\|x\|=\|y\|=1$ such that $\|x+y\|=2$ we have $x=y$,  and $X$ is said to be\/ {\it uniformly convex} if given $\varepsilon>0$ there exists $\delta>0$ so that if $\|x\|=\|y\|=1$ and $\|x+y\|\ge 2-\delta$ then $\|x-y\|\le \varepsilon$.  Uniformly convex spaces are strictly convex,  but the converse is false.  It is well-known that if $1<p<\infty$ then $L^p(\Omega)$ is uniformly convex (cf.  \cite{Clarkson}),  so is $A^p(\Omega)$ since every subspace of a uniformly convex space must have the same property.   It is also known that $L^1(\Omega)$ is not strictly convex,  and $A^1(\Omega)$ is not uniformly convex in general.  However,  we still have

\begin{proposition}
$A^1(\Omega)$ is strictly convex.
\end{proposition}

 \begin{proof}
 Take $f,g\in A^1(\Omega)$ with $\|f\|_1=\|g\|_1=1$ and $\|f+g\|_1=2$.  We have 
 $$
 \|f+g\|_1=\|f\|_1+\|g\|_1,
 $$
 so that $f=\lambda g$ for some function $\lambda\ge 0$. Thus $\lambda=f/g$ is a meromorphic function on $\Omega$ whose range is contained in $\mathbb R$,  so it has to be a constant.  Since $\|f\|_1=\|g\|_1=1$,  so $\lambda=1$,  i.e.,  $f=g$.
\end{proof}

In \S\,8,  we will investigate $A^p(\Omega)$ in detail for $1<p<\infty$ from viewpoint of the geometry of Banach spaces.

\subsection{The $p-$Bergman kernel} For   a bounded domain in $\Omega\subset \mathbb C^n$, we consider   the following minimizing problem:
 \begin{equation}\label{eq:Min_1}
 m_p(z)=m_{\Omega,p}(z)=\inf\left\{\|f\|_p:f\in A^p(\Omega),f(z)=1\right\}.
 \end{equation}

 \begin{proposition}[Existence]
  There exists at least one minimizer in \eqref{eq:Min_1}.
  \end{proposition}

  \begin{proof}
  Take $\{f_j\}\subset A^p(\Omega)$ such that $f_j(z)=1$ and $\|f_j\|_p\rightarrow m_p(z)$ as $j\rightarrow \infty$. The Bergman inequality implies that $\{f_j\}$ is a normal family so that there exists a subsequence $\{f_{j_k}\}$ which converges locally uniformly to some $f_0\in {\mathcal O}(\Omega)$. By Fatou's lemma, we have
  $$
  \int_\Omega |f_0|^p \le \liminf_{k\rightarrow \infty} \int_\Omega |f_{j_k}|^p=m_p(z)^p.
  $$
  On the other hand, since $f_j(z)=1$, we have $f_0(z)=1$ and
  $
  \|f_0\|_p=m_p(z),
  $
  i.e., $f_0$ is a minimizer.
  \end{proof}

  \begin{proposition}[Uniqueness]\label{prop:uniq}
  For $ p \ge 1$ there is only one minimizer in \eqref{eq:Min_1}.
  \end{proposition}

 \begin{proof}
  Let $f_1,f_2$ be two minimizers of \eqref{eq:Min_1}. We take $h:=\frac{f_1+f_2}2$. Clearly, $h$ belongs to $A^p(\Omega)$ and satisfies $h(z)=1$.  By convexity of the function $\mathbb R^+\ni x\mapsto x^p$ we have
$$
m_p(z)^p\le \int_\Omega |h|^p \le \int_\Omega \frac{|f_1|^p+|f_2|^p}2 = m_p(z)^p,
$$  
so that $\|h\|_p=m_p(z)$.   Set $\widehat{f}_j=f_j/m_p(z)$ for $j=1,2$.  Then we have $\|\widehat{f}_1\|_p=\|\widehat{f}_2\|_p=1$ and $\|\widehat{f}_1+\widehat{f}_2\|_p=2$.  By  strict convexity of $A^p(\Omega)$ we have $\widehat{f}_1=\widehat{f}_2$,  so that $f_1=f_2$. 
      \end{proof}

      \begin{remark}
       It is not known whether the uniqueness result holds for\/ $0<p<1$. 
      \end{remark}

      Let $ m_p(\cdot,z)$ denote a minimizer in (\ref{eq:Min_1})  (one warning: $m_p(z,z)=1\neq m_p(z)$!).
      
      \begin{definition}
      We call $K_p(z):=m_p(z)^{-p}$ the $p-$Bergman kernel for $p>0$ and $K_p(z,w):=m_p(z,w)K_p(w)$ the off-diagonal $p-$Bergman kernel for $p\ge 1$.
      \end{definition}
      
      \begin{proposition}
      For every $p>0$ we have
      \begin{equation}\label{eq:max_1}
K_p(z)=\sup\left\{|f(z)|^p: f\in A^p(\Omega),\|f\|_p=1\right\}.
\end{equation}
      \end{proposition}
      
      \begin{proof}
      Take $f_0\in A^p(\Omega)$ with $f_0(z)=1$ and $\|f_0\|_p=m_p(z)$. We have
  $$
\sup_{f\in A^p(\Omega)} \frac{|f(z)|^p}{\|f\|_p^p}\ge \frac{|f_0(z)|^p}{\|f_0\|_p^p}= m_p(z)^{-p}= K_p(z).
  $$
  On the other hand, for $f\in A^p(\Omega)$ with $\|f\|_p=1$ and $f(z)\neq0$ we see that $\widehat{f}:=f/f(z)\in A^p(\Omega)$
 satisfies $\widehat{f}(z)=1$. It follows that
 $$
 m_p(z) \le \|\widehat{f}\|_p=1/|f(z)|,\ \ \ \text{i.e.},\  |f(z)|^p \le m_p(z)^{-p}=K_p(z).
 $$
If $f(z)=0$, then $|f(z)|^p=0\leq K_p(z)$.
      \end{proof}
      
      In particular, if  $\Omega\subset \Omega'$, then
$$
K_{\Omega,p}(z)\ge K_{\Omega',p}(z)\ \ \ \text{for\ }z\in \Omega.
$$
Moreover, we have
\begin{equation}\label{eq:BergIneq_3}
{|\Omega|}^{-1} \le K_p(z)\le C_n \delta(z)^{-2n}
\end{equation}
in view of (\ref{eq:BergIneq_2}), where $\delta=\delta_\Omega$ denotes the boundary distance.  

Let us present a few elementary properties (some of them are known).
    
        \begin{proposition}[Transformation rule]\label{prop:trans_2}
Let $F:\Omega_1\rightarrow \Omega_2$ be a biholomorphic mapping between bounded simply-connected domains. Let $J_F$ denote the complex Jacobian of $F$. Then 
 \begin{eqnarray}
 m_{\Omega_1,p}(z) & = & m_{\Omega_2,p}(F(z)) |J_F(z)|^{-2/p},\ \ \ p>0, \label{eq:trans_1} \\
 K_{\Omega_1,p}(z) & = & K_{\Omega_2,p}(F(z))|J_F(z)|^2,\ \ \ p>0, \label{eq:trans_2}\\
  m_{\Omega_1,p}(z,w) & = & m_{\Omega_2,p}(F(z),F(w)) J_F(z)^{\frac2p} J_F(w)^{-\frac2p},\ \ \ p\ge 1,\label{eq:trans_3}\\
  K_{\Omega_1,p}(z,w) & = & K_{\Omega_2,p}(F(z),F(w)) J_F(z)^{\frac2p} J_F(w)^{1-\frac2p}\,\overline{J_F(w)},\ \ \ p\ge 1\label{eq:trans_4}.
    \end{eqnarray}
  Moreover,  equalities hold for arbitrary bounded domains when  $2/p\in \mathbb Z^+$.
\end{proposition}
\begin{proof}
Since
$$
\int_{\Omega_2} |f_2|^p = \int_{\Omega_1} |f_2\circ F|^p |J_F|^2,
$$
we conclude that
$
f_2\in A^p(\Omega_2)
$
if and only if $ {f}_1:=f_2 \circ F\cdot J_F^{2/p}\in A^p(\Omega_1)$ (only here we have to use the assumption that $\Omega_1$ is simply-connected). If $f_2(F(z))=1$, then $f_1(z)J_F(z)^{-2/p}=1$, so that
$$
m_{\Omega_1,p}(z)^p \le |J_F(z)|^{-2}\,\int_{\Omega_1} |f_1|^p =  |J_F(z)|^{-2}\,\int_{\Omega_2} |f_2|^p.
$$
Take minimum over $f_2$, we get
$$
m_{\Omega_1,p}(z)^p \le |J_F(z)|^{-2} m_{\Omega_2,p}(F(z))^p.
$$
Consider $F^{-1}$ instead of $F$, we get the inverse inequality in \eqref{eq:trans_1}.

Next, we define for fixed $w\in \Omega_1$, a holomorphic function by 
$$
f_1(z):=m_{\Omega_2,p}(F(z),F(w))  J_F(z)^{\frac2p} J_F(w)^{-\frac2p}.
$$
Clearly, we have $f_1(w)=1$ and
$$
\int_{\Omega_1}|f_1|^p = |J_F(w)|^{-2} \int_{\Omega_2}| m_{\Omega_2,p}(\cdot,F(w)) |^p =  |J_F(w)|^{-2} m_{\Omega_2,p}(F(w))^p=m_{\Omega_1,p}(w)^p
$$
in view of  \eqref{eq:trans_1}.  (\ref{eq:trans_3}) follows immediately from Proposition \ref{prop:uniq}.

The remaining equalities follow from the relations 
$$
K_p(z)=m_p(z)^{-p}\ \ \  \text{and}\ \ \  K_p(z,w)=m_p(z,w)K_p(w).
$$
\end{proof}

 \begin{remark}
  The simply-connected hypothesis can not be removed  (see \cite{NZZ}, Remark 2.3).
 \end{remark}

 \begin{proposition}[Product rule]\label{prop:product_1}
Let $\Omega'$ and $\Omega''$ be  bounded domains in\/ $\mathbb C^{n}$ and\/ $\mathbb C^{m}$ respectively. Set $\Omega=\Omega'\times \Omega''$ and $z=(z',z'')$. Then we have
\begin{eqnarray*}
m_{\Omega,p}(z) & = & m_{\Omega',p}(z')\cdot m_{\Omega'',p}(z''),\ \ \  p>0,\\
m_{\Omega,p}(z,w) & = & m_{\Omega',\,p}(z',w')\cdot m_{\Omega'',\,p}(z'',w''), \ \ \ p\ge 1.
\end{eqnarray*}
The same conclusions also hold for $K_p$.
\end{proposition}

\begin{proof}
For fixed $z'\in \Omega'$ and $z''\in \Omega''$ we take $f_1\in A^p(\Omega')$ and $f_2\in A^p(\Omega'')$ such that $f_1(z')= f_2(z'')=1$ and
$$
m_{\Omega',p}(z')=\|f_1\|_p,\ \ \ m_{\Omega'',p}(z'')=\|f_2\|_p.
$$
Fubini's theorem gives
$$
\int_{\zeta'\in \Omega'}\int_{\zeta''\in \Omega''} |f_1(\zeta')f_2(\zeta'')|^p =\|f_1\|_p^p\cdot \|f_2\|_p^p=m_{\Omega',p}(z')^p\cdot m_{\Omega'',p}(z'')^p. 
$$
Thus
$$
m_{\Omega,p}(z)\le m_{\Omega',p}(z')\cdot m_{\Omega'',p}(z'').
$$
On the other hand, for every $h\in A^p(\Omega)$  we have
\begin{eqnarray*}
 |h(z',z'')|^p & \le & K_{\Omega',p}(z')\cdot \int_{\Omega'} |h(\cdot,z'')|^p\\
 & \le & K_{\Omega',p}(z')\cdot K_{\Omega'',p}(z'')\cdot \int_{\Omega'} \int_{\Omega''} |h|^p,
  \end{eqnarray*}
  so that
  $$
K_{\Omega,p}(z)\le K_{\Omega',p}(z')\cdot K_{\Omega'',p}(z'').
$$
Since $K_p(z)=m_p(z)^{-p}$, the first equality follows immediately.

Next, we note that  for fixed $w\in \Omega$ the function
$$
f_0(z):=m_{\Omega',\,p}(z',w')\cdot m_{\Omega'',\,p}(z'',w'')
$$
is holomorphic on $\Omega$ and satisfies $f_0(w)=1$,
\begin{eqnarray*}
\int_\Omega |f_0|^p & = & \int_{\Omega'} |m_{\Omega',\,p}(\cdot,w')|^p\,\int_{\Omega'} |m_{\Omega'',\,p}(\cdot,w'')|^p\\
& = & m_{\Omega',\,p}(w')^p\,m_{\Omega'',\,p}(w'')^p\\
& = & m_{\Omega,p}(w)^p.
\end{eqnarray*}
 By uniqueness of the minimizer we immediately get the  the second equality.
\end{proof}

\begin{proposition}\label{prop:ball}
For the unit ball\/ $\mathbb B^n\subset \mathbb C^n$ we have
\begin{equation}\label{eq:ball}
K_p(z,w)=K_{\mathbb B^n,\,p}(z,w)=\frac{n!}{\pi^n}\,\frac{(1-|w|^2)^{(n+1)(\frac2p-1)}}{(1-\langle z,w\rangle)^{\frac{2(n+1)}p}},
\end{equation}
where $\langle{z,w}\rangle:=\sum^n_{j=1}z_j\bar{w}_j$.
\end{proposition}

\begin{proof}
For any $f\in A^p(\mathbb B^n)$ with $f(0)=1$ we have
$$
1=|f(0)|^p \le \frac1{|\mathbb B^n|} \int_{\mathbb B^n} |f|^p,
$$
while for $f_0\equiv 1$, 
$$
\int_{\mathbb B^n} |f_0|^p=|\mathbb B^n|\le \int_{\mathbb B^n} |f|^p.
$$
Thus $f_0$ is a minimizer at $0$, so that
$$
m_p(\cdot,0)=f_0(\cdot)\equiv 1,
$$
and
$$
K_p(\cdot,0)=\frac{m_p(\cdot,0)}{m_p(0)^p}=\frac1{|\mathbb B^n|}=\frac{n!}{\pi^n}.
$$
For $a\in \Delta$ we set $w_a:=(a,0')$. Consider the following automorphism of $\mathbb B^n$
$$
F_a:z\mapsto \left(\frac{z_1-a}{1-\bar{a}z_1},\frac{\sqrt{1-|a|^2}\,z'}{1-\bar{a}z_1}\right).
$$
A straightforward calculation shows
$$
J_{F_a}(z)=\frac{(1-|a|^2)^{\frac{n+1}2}}{(1-\bar{a}z_1)^{n+1}},\ \ \ J_{F_a}(w_a)=(1-|a|^2)^{-\frac{n+1}2}.
$$
It follows  that
\begin{eqnarray*}
K_p(z,w_a)  = K_p(z,F^{-1}_a(0)) & = & K_p(F_a(z),0) J_{F_a}(z)^{\frac2p} J_{F_a}(w_a)^{1-\frac2p}\overline{J_{F_a}(w_a)} \\
& = & \frac{n!}{\pi^n} \frac{(1-|a|^2)^{\frac{n+1}p}}{(1-\bar{a}z_1)^{\frac{2(n+1)}p}} (1-|a|^2)^{-(n+1)(1-\frac1p)}\\
& = & \frac{n!}{\pi^n} \frac{(1-|w_a|^2)^{(n+1)(\frac2p-1)}}{(1-\langle z,w_a\rangle)^{\frac{2(n+1)}p}}.
\end{eqnarray*}
For general $w\in \mathbb B^n$ we take a unitary transformation ${\mathcal U}$ with ${\mathcal U}(w)=(|w|,0')$. Thus
\begin{eqnarray*}
K_p(z,w) & = & K_p(\mathcal U(z),\mathcal U(w)) J_{\mathcal U}(z)^{\frac2p} J_{\mathcal U}(w)^{1-\frac2p}\overline{J_{\mathcal U}(w)}\\
& = & \frac{n!}{\pi^n} \frac{(1-|\mathcal U(w)|^2)^{(n+1)(\frac2p-1)}}{(1-\langle \mathcal U(z),\mathcal U(w)\rangle)^{\frac{2(n+1)}p}}\\
& = & \frac{n!}{\pi^n} \frac{(1-|w|^2)^{(n+1)(\frac2p-1)}}{(1-\langle z,w\rangle)^{\frac{2(n+1)}p}}.
\end{eqnarray*}
\end{proof}

This proposition combined with the product rule gives

\begin{proposition}\label{prop:polydisc}
For the unit polydisc $\Delta^n\subset \mathbb C^n$ we have
\begin{equation}\label{eq:polydisc}
K_p(z,w)=K_{\Delta^n,\,p}(z,w)=\frac1{\pi^n} \prod_{j=1}^n \frac{(1-|w_j|^2)^{\frac4p-2}}{(1-\overline{w}_j z_j)^{\frac4p}}.
\end{equation}
\end{proposition}

\begin{proposition}\label{prop:continuity}
\begin{enumerate}
\item[$(1)$] Both $m_p(z)$ and $K_p(z)$ are locally Lipschitz continuous for $p>0$.
\item[$(2)$] Both $m_p(z,w)$  and $K_p(z,w)$ are continuous in $(z,w)$ for $p\ge 1$.
\end{enumerate}
\end{proposition}

\begin{proof}
(1) Let $S$ be a compact set in $\Omega$ and $z\in S$. Take $f\in A^p(\Omega)$ with $\|f\|_p=1$ such that
$|f(z)|^p=K_p(z)$. It follows from Cauchy's estimates  that for any $w\in S$,
$$
K_p(z)^{1/p}=|f(z)|\le |f(w)|+C_S|z-w|\le K_p(w)^{1/p} +C_S|z-w|,
$$
i.e., $K_p(z)^{1/p}$ is locally Lipschitz continuous in $z$, so are $K_p(z)$ and $m_p(z)$.

(2) It suffices to verify continuity of $m_p(z,w)$. Let $z_0,w_0\in \Omega$ be fixed. We first verify that
\begin{equation}\label{eq:converg_1}
\lim_{w\rightarrow w_0} m_p(z_0,w) = m_p(z_0,w_0).
\end{equation}
Let $w_j\rightarrow w_0$. Since
$$
\int_\Omega |m_p(\cdot,w_j)|^p=m_p(w_j)^p=\frac1{K_p(w_j)}\le |\Omega|,
$$
so $\{m_p(\cdot,w_j)\}$ forms a normal family so that there exists a subsequence $\{m_p(\cdot,w_{j_k})\}$ converging locally uniformly to a function $f_0\in \mathcal O(\Omega)$. Fatou's lemma  yields
$$
\int_\Omega |f_0|^p \le \liminf_{k\rightarrow \infty} \int_\Omega |m_p(\cdot,w_{j_k})|^p = \liminf_{k\rightarrow \infty} m_p(w_{j_k})^p=m_p(w_0)^p.
$$
On the other hand, Cauchy's estimates yield
$$
\frac{|m_p(w_0,w_{j_k})-m_p(w_{j_k},w_{j_k})|}{|w_0-w_{j_k}|}\le C  \int_\Omega |m_p(\cdot,w_{j_k})|\le C \|m_p(\cdot,w_{j_k})\|_p |\Omega|^{\frac1q}\le C |\Omega|
$$
where $\frac1p+\frac1q=1$ and $C$ depends only on $w_0$. We then have
$$
f_0(w_0)=\lim_{k\rightarrow \infty}m_p(w_0,w_{j_k}) = \lim_{k\rightarrow \infty} m_p(w_{j_k},w_{j_k})=1,
$$
so that $f_0=m_p(\cdot,w_0)$ by uniqueness of the minimizer.  Consequently,
$$
\lim_{k\rightarrow \infty} m_p(z_0,w_{j_k}) = m_p(z_0,w_0).
$$
Since the sequence $\{w_j\}$ can be chosen arbitrarily, we get (\ref{eq:converg_1}).  Finally,
\begin{eqnarray*}
|m_p(z,w)-m_p(z_0,w_0)| & \le & |m_p(z,w)-m_p(z_0,w)| + |m_p(z_0,w)-m_p(z_0,w_0)| \\
& \le & C_0 |\Omega| |z-z_0| +  |m_p(z_0,w)-m_p(z_0,w_0)| \\
& \rightarrow & 0
\end{eqnarray*}
as $z\rightarrow z_0$ and $w\rightarrow w_0$.
\end{proof}

\subsection{A reproducing formula}
  Throughout  this subsection  we always assume that $p\ge 1$.

      \begin{lemma}\label{lm:var_2}
For any $f\in A^p(\Omega)$ with $f(z)=0$,   we have
 \begin{equation}\label{eq:Var_2}
 \int_\Omega |m_p(\cdot,z)|^{p-2}\,\overline{m_p(\cdot,z)}\,f  = 0.
  \end{equation}
   \end{lemma}

 \begin{proof}
We will use the calculus of variations. For fixed $f$  we consider the family
 $$
 f_t=m_p(\cdot,z)+tf\in A^p(\Omega),\ \ \ t\in \mathbb  C.
 $$
  Since $f_t(z)={1}$, we see that the function $J(t):=\|f_t\|^p_p$ attains the minimum at $t=0$.  Rewrite
 $$
 |f_t|^p=\left(|m_p(\cdot,z)|^2+tf\,\overline{m_p(\cdot,z)}+\overline{tf}\,m_p(\cdot,z)+|t|^2|f|^2\right)^{\frac{p}2}.
 $$
 Since
 $$
 \frac{\partial |f_t|^p}{\partial t}= {\frac{p}2} |f_t|^{{p}-2}\bar{f}_t f
  $$
  holds outside the proper analytic set $f_t^{-1}(0)\subset \Omega$ (whose Lebesgue measure is zero),  
  we have
  $$
  \left|\frac{\partial |f_t|^p}{\partial t}\right|={\frac{p}2} |f_t|^{p-1}|f|\le {\frac{p}2} |f|(|m_p(\cdot,z)|+|f|)^{p-1}=:\phi
  $$
  whenever $|t|\le 1$. Analogously, we may verify that
  $$
  \left|\frac{\partial |f_t|^p}{\partial \bar{t}}\right|\le \phi.
    $$
 Note that  
  $$
  \int_\Omega \phi \le {\frac{p}2} \|f\|_p \||m_p(\cdot,z)|+|f|\|_p^{p-1}<\infty
      $$
      in view of H\"older's inequality when $p>1$. The  inequality for $p=1$ is clearly trivial.   It then follows from the dominated convergence theorem that
  $$
  0=\frac{\partial J}{\partial t}(0)=\int_\Omega \left.\frac{\partial |f_t|^p}{\partial t}\right|_{t=0}={\frac{p}2} \int_\Omega |m_p(\cdot,z)|^{p-2}\,\overline{m_p(\cdot,z)}\,f,
     $$
     i.e., (\ref{eq:Var_2}) holds.
 \end{proof}

Now we reach the following fundamental fact.

  \begin{theorem}[Reproducing formula]\label{th:RP}
For any $f\in A^p(\Omega)$ we have
  \begin{equation}\label{eq:RP}
  f(z)  =  m_p(z)^{-p} \int_\Omega |m_p(\cdot,z)|^{p-2}\,\overline{m_p(\cdot,z)}\,f= \int_\Omega |m_p(\cdot,z)|^{p-2}\,\overline{K_p(\cdot,z)}\,f. 
  \end{equation}
    \end{theorem}

    \begin{proof}
    Let $f\in A^p(\Omega)$. With $f$ replaced by $f-f(z)$ in (\ref{eq:Var_2}), we obtain
      \begin{equation}\label{eq:RP_2}
      \int_\Omega |m_p(\cdot,z)|^{p-2}\,\overline{m_p(\cdot,z)}\,f = f(z)\cdot \int_\Omega |m_p(\cdot,z)|^{p-2}\,\overline{m_p(\cdot,z)}.
      \end{equation}
      Substitute $f=m_p(\cdot,z)$ into (\ref{eq:RP_2}), we obtain
      $$
      m_p(z)^p=\int_\Omega |m_p(\cdot,z)|^{p-2}\,\overline{m_p(\cdot,z)}.
          $$
           This combined with (\ref{eq:RP_2}) yields (\ref{eq:RP}).
         \end{proof}

       In other words,  Theorem \ref{th:RP} states that the continuous linear functional $T:A^p(\Omega)\ni f\mapsto f(z)$  may be represented by integration against 
       $$
       g_T:= \frac{|m_p(\cdot,z)|^{p-2}\,{m_p(\cdot,z)}}{m_p(z)^p}\in L^q(\Omega)
       $$
       with $\|g_T\|_q=m_p(z)^{-1}=\|T\|$,  where $\frac1p+\frac1q=1$.  An important question is when $T$ can be represented by some function in $A^q(\Omega)$.

          We also  have the following geometric interpretation of $m_p(z)$.   
   
     \begin{proposition}\label{prop:geometric}
 Set  
         $
         H^c_{p,z}:=\{f\in A^p(\Omega): f(z)=c\},\ c\in \mathbb C.
         $ 
         Then    $m_p(z)$ is the distance between $H^c_{p,z}$ and $H^{c+1}_{p,z}$.
     \end{proposition} 
\begin{proof}Take $f\in H^c_{p,z}$ and $g\in H^{c+1}_{p,z}$.   By  Theorem 2.13,  we have 
\begin{eqnarray*}
	c+1=g(z) & = & m_p(z)^{-p} \int_\Omega |m_p(\cdot,z)|^{p-2}\,\overline{m_p(\cdot,z)}\,g,\\
		c=f(z) & = & m_p(z)^{-p} \int_\Omega |m_p(\cdot,z)|^{p-2}\,\overline{m_p(\cdot,z)}\,f,
\end{eqnarray*}
so that
\begin{eqnarray*}
	m_p(z)^{p} & = & \int_\Omega |m_p(\cdot,z)|^{p-2}\,\overline{m_p(\cdot,z)}(g-f)\\
	&\leq &  m_p(z)^{p-1}\|g-f\|_{p},
\end{eqnarray*}
i.e.,   $ m_p(z)\leq\|g-f\|_{p}$.   Equality holds when  $g:=f+m_p(\cdot,z)\in H^{c+1}_{p,z}$.
\end{proof}

\begin{proposition}\label{prop:indep}
 Given two distinct points $z,w\in \Omega$, $m_p(\cdot,z)$ and $m_p(\cdot,w)$ are not parallel, i.e., $m_p(\cdot,z)\neq c m_p(\cdot,w)$ for any $c\in \mathbb C$.
\end{proposition}

\begin{proof}
Suppose on the contrary that  $m_p(\cdot,z)=c m_p(\cdot,w)$ for some complex number $c$. It follows from Theorem \ref{th:RP} that for any $f\in A^p(\Omega)$,
\begin{eqnarray*}
f(z) & = & m_p(z)^{-p} \int_\Omega |m_p(\cdot,z)|^{p-2}\, \overline{m_p(\cdot,z)}\,f\\
& = & m_p(z)^{-p} |c|^{p-2}\bar{c} \int_\Omega |m_p(\cdot,w)|^{p-2}\, \overline{m_p(\cdot,w)}\,f\\
& = & \left[m_p(w)/m_p(z)\right]^{p} |c|^{p-2}
\bar{c}\,f(w).
\end{eqnarray*}
On the other hand, 
$$
m_p(z)^p=\int_\Omega |m_p(\cdot,z)|^p = |c|^p \int_\Omega |m_p(\cdot,w)|^p=|c|^p m_p(w)^p.
$$
Thus we have $f(z)=f(w)/c$. But this is absurd since one can choose $f\in A^p(\Omega)$ with $f(z)=0$ and $f(w)\neq 0$.
\end{proof}

\begin{problem}
 Let $w_1,\cdots,w_m$ be different points in $\Omega$. Is it possible to conclude that $m_p(\cdot,w_1)$, $\cdots$, $m_p(\cdot,w_m)$ are linearly independent?
\end{problem}

  \begin{proposition}\label{prop:Triangle}
  We have
  \begin{equation}\label{eq:Triangle}
   |m_p(z,w)|\le m_p(w)/m_p(z)
     \end{equation}
     and equality holds if and only if $z=w$. Equivalently, 
     \begin{equation}\label{eq:Holder_3}
  |K_p(z,w)|\le K_p(z)^{\frac1p}K_p(w)^{\frac1q}
  \end{equation}
  where $\frac1p+\frac1q=1$, and equality holds if and only if $z=w$.  
  \end{proposition}

   \begin{proof}
  Substitute $f=m_p(\cdot,w)$ into (\ref{eq:RP}), we obtain
  \begin{eqnarray*}
  |m_p(z,w)| & = &  m_p(z)^{-p} \left|\int_\Omega |m_p(\cdot,z)|^{p-2}\,\overline{m_p(\cdot,z)}\,m_p(\cdot,w)\right| \\
  & \le & m_p(z)^{-p} m_p(z)^{p-1} m_p(w) \ \ \ \ \ (\text{H\"older's\ inequality})\\
  & = & m_p(w)/m_p(z).
     \end{eqnarray*}
     Clearly, equality holds if $z=w$. On the other hand, if equality in \eqref{eq:Triangle} holds then there exists $r>0$ such that
          \begin{equation}\label{eq:iff}
|m_p(\cdot,w)|^p =r  \left(|m_p(\cdot,z)|^{p-2}|m_p(\cdot,z)|\right)^{\frac{p}{p-1}}=r |m_p(\cdot,z)|^p.
     \end{equation}
     Set $h:=  m_p(\cdot,z)/m_p(\cdot,w)$ and $S_w:=\{m_p(\cdot,w)=0\}$. Since $m_p(w,w)=1$, it follows that $S_w$ is an analytic hypersurface of $\Omega$ and $h$ is holomorphic on $\Omega\backslash S_w$. By \eqref{eq:iff}, we see that $|h|$ is a constant on $\Omega\backslash S_w$. Thus $h$ has to be a constant, i.e., $m_p(\cdot,z)=cm_p(\cdot,w)$ for some complex number $c$ on  $\Omega\backslash S_w$. By continuity, the same equality holds on $\Omega$, so that $z=w$ in view of Proposition \ref{prop:indep}.
       \end{proof}

   \begin{proposition}\label{prop:nonconst}
    For fixed $z\in \Omega$, $m_p(z,\cdot)\neq {\rm const}$.
       \end{proposition}

       \begin{proof}
      Note that
      $$
      |m_p(z,w)|\le m_p(w)/m_p(z) = K_p(w)^{-\frac1p}\,m_p(z)^{-1}.
      $$
      Since $\Omega$ is bounded, there exists a ball $B\supset \Omega$ such that $\partial \Omega\cap \partial B\neq \emptyset$.
      For any $w_0\in \partial \Omega\cap \partial B$, we have
      $$
      K_{\Omega,p}(w)\ge K_{B,p}(w)\rightarrow \infty\ \ \ (w\rightarrow w_0),
      $$
      which implies $m_p(z,w)\rightarrow 0$ as $w\rightarrow w_0$. On the other hand, we have $m_p(z,z)=1$. Thus $m_p(z,\cdot)\neq {\rm const}$.
       \end{proof}

  \begin{problem}
        Is it possible to conclude that $m_p(z,\cdot)$ can not be\/ {\it locally} constant?
       \end{problem}

       \subsection{The  $p-$Bergman metric} For $X=\sum_j X_j \partial/\partial z_j$ we define the  $p-$Bergman metric to be
\begin{equation}\label{eq:p-metric}
B_p(z;X):=  {K_p(z)^{-\frac1p}}\cdot {\sup}_f\ |Xf(z)|
\end{equation}
where the supremum is taken over all $f\in A^p(\Omega)$ with $f(z)=0$ and $\|f\|_p=1$.  Note that $B_2(z;X)$ is the standard Bergman metric. A normal family argument shows that the "$\sup$" in (\ref{eq:p-metric}) can be replaced by "$\max$". For the sake of convenience, we set
\begin{equation}\label{eq:max_3}
\mathcal M_{p}(z;X):= \sup\left\{|Xf(z)|: f\in A^p(\Omega),f(z)=0,\|f\|_p=1\right\}
\end{equation}
 and define  $\mathcal M_p(\cdot,z;X)$ to be the maximizer  of
  (\ref{eq:max_3}), i.e., $\mathcal M_{\Omega,p}(z;X)=X\mathcal M_p(\cdot,z;X)|_z$.

 \begin{proposition}\label{prop:invariant}
Let $F:\Omega_1\rightarrow \Omega_2$ be a biholomorphic mapping between bounded simply-connected domains.   Then
\begin{equation}\label{eq:invariant}
B_{\Omega_1,p}(z;X)=B_{\Omega_2,p}(F(z);F_\ast X).
\end{equation}
Moreover, \eqref{eq:invariant} holds for arbitrary bounded domains whenever $2/p\in \mathbb Z^+$.\end{proposition}

\begin{proof}
By Proposition \ref{prop:trans_2} it suffices to verify
$$
\mathcal M_{\Omega_1,p}(z;X)=\mathcal M_{\Omega_2,p}(F(z);F_\ast X)|J_F(z)|^{2/p}.
$$
If $f_2$ is a test function for $\mathcal M_{\Omega_2,p}(F(z);F_\ast X)$, then $ {f}_1:=f_2 \circ F\cdot J_F^{2/p}$ is a test function for $\mathcal M_{\Omega_1,p}(z;X)$. Then we have
$$
{|X(f_2\circ F)(z) \cdot J_F(z)^{2/p}|} = {|X f_1(z)|} \le \mathcal M_{\Omega_1,p}(z;X).
$$
Take supremum over $f_2\in A^p(\Omega_2)$, we get
$$
\mathcal M_{\Omega_2,p}(F(z);F_\ast X)|J_F(z)|^{2/p}\le \mathcal M_{\Omega_1,p}(z;X).
$$
Consider $F^{-1}$ instead of $F$, we get the inverse inequality.
 \end{proof}

\begin{proposition}
 For the unit ball\/ $\mathbb B^n\subset \mathbb C^n$ we have
% \begin{equation}\label{eq:ball}
% K_p(z)=K_2(z)=\frac{n!}{\pi^n} (1-|z|^2)^{-n-1}
% \end{equation}
 \begin{equation}\label{eq:ballmetric}
 B_p(z;X) =c_{n,p}  \left(\frac{|X|^2}{1-|z|^2}+\frac{|\sum_{j=1}^n z_j X_j|^2}{(1-|z|^2)^2}\right)^{\frac12}
 \end{equation}
 where
 $$
 c_{n,p}= (\pi^n/n!)^{\frac1p}\cdot \sup\left\{\frac{|f(0)|}{\|z_1f\|_p}: f\in A^p(\mathbb B^n) \right\}. $$
 \end{proposition}

 \begin{proof}
 For  $z\in \mathbb B^n$ we take an automorphism $F$ of $\mathbb B^n$ such that $F(z)=0$. %By (\ref{eq:trans}), we obtain
% $$
% \frac{K_p(z)}{K_2(z)}=\frac{K_p(0)}{K_2(0)}.
% $$
% By using the test function $f=1$, we have
% $$
% K_p(0)\ge \frac1{|\mathbb B^n|}.
% $$
% On the other hand, the mean-value inequality implies
% $$
% |f(0)|^p \le \frac1{|\mathbb B^n|} \int_{\mathbb B^n} |f|^p
%  $$
%  for all $f\in A^p(\mathbb B^n)$, so that
%  $$
% K_p(0)\le \frac1{|\mathbb B^n|}.
% $$
% Thus $K_p(0)=\frac1{|\mathbb B^n|}$, from which (\ref{eq:ball}) immediately follows.
 By (\ref{eq:invariant}), we have
 $$
 \frac{B_p(z;X)}{B_2(z;X)}= \frac{B_p(0;F_\ast X)}{B_2(0;F_\ast X)}.
  $$
  It suffices to compute the ratio $B_p(0;X)/B_2(0;X)$. After a unitary transformation, we may assume $X=|X|\partial/\partial z_1$.
 Since every $f\in \mathcal O(\Omega)$ with $f(0)=0$ admits a decomposition $f(z)=\sum_j z_j f_j(z)$  for certain $f_j\in \mathcal O(\Omega)$, we obtain 
  $$
  B_p(0;X)=\frac{|X|}{K_p(0)^{\frac1p}}\cdot\sup\left\{\frac{|f(0)|}{\|z_1f\|_p}: f\in A^p(\mathbb B^n) \right\}.
  $$
  Since $K_p(0)=n!/\pi^n$ and $B_2(0;X)=(n+1)^{\frac12}|X|$, we obtain (\ref{eq:ballmetric}).
 \end{proof}
 
\begin{problem}
What is the product rule for $B_p$?
\end{problem} 

 Recall that the Carath\'eodory metric is defined by
 $$
 C(z;X)=C_\Omega(z;X):=\sup\left\{|Xf(z)|:f\in A^\infty(\Omega), f(z)=0, \|f\|_\infty=1\right\}.
 $$

  \begin{proposition}\label{prop:compare}
   $
  B_p(z;X)\ge C(z;X).
  $
  \end{proposition}

   \begin{proof}
 Take $h\in A^p(\Omega)$ and  $f\in A^\infty(\Omega)$ with $f(z)=0$ and $\|f\|_\infty=1$. Set $g=f\cdot h$. Then we  have $g(z)=0$, $\|g\|_p\le \|h\|_p$ and
  $$
  |Xg(z)|=|Xf(z)|\cdot |h(z)|,
  $$
  so that
  $$
  \mathcal M_p(z;X) \ge \frac{|Xg(z)|}{\|g\|_p}\ge |Xf(z)|\cdot \frac{|h(z)|}{\|h\|_p}.
  $$
  Take supremum over $f$ and $h$, we immediately get the desired conclusion.
    \end{proof}

 \begin{proposition}\label{prop:BergCar}
 \ \ \  $\lim_{p\rightarrow \infty} B_p(z;X)=C(z;X)$.
 \end{proposition}

 \begin{proof}
 Take  a sequence $p_j\rightarrow \infty$ such that
$$
\lim_{j\rightarrow \infty} B_{p_j}(z;X) =\limsup_{p\rightarrow \infty} B_p(z;X).
$$
We also choose $f_j\in A^{p_j}(\Omega)$ with $\|f_j\|_{p_j}=1$, $f_j(z)=0$ and
$$
B_{p_j}(z;X)=|Xf_j(z)|/K_{p_j}(z)^{\frac1{p_j}}
$$
 for every $j$.  Since
\begin{equation}\label{eq:BC_1}
|f_j(\zeta)|^{p_j} \le C_n \delta(\zeta)^{-2n},
\end{equation}
it follows that $\{f_j\}$ forms a normal family, so that there is a subsequence $\{f_{j_k}\}$ converging locally uniformly to some $f_\infty\in \mathcal O(\Omega)$ which satisfies 
$f_\infty (z)=0$ and for any $\zeta\in \Omega$,
$$
|f_\infty(\zeta)|=\lim_{k\rightarrow \infty} |f_{j_k}(\zeta)|\le 1
$$
in view of \eqref{eq:BC_1}. Since $\lim_{p\rightarrow \infty} K_p(z)^{\frac1p}=1$ in view of \eqref{eq:BergIneq_3}, we have 
\begin{eqnarray*}
C(z;X) & \ge & {|Xf_\infty(z)|}
=\lim_{k\rightarrow \infty} {|Xf_{j_k}(z)|}
\cdot \lim_{k\rightarrow \infty} K_{p_{j_k}}(z)^{-\frac1{p_{j_k}}}\\
& = & \lim_{k\rightarrow \infty} B_{p_{j_k}}(z;X)
=\limsup_{p\rightarrow \infty} B_p(z;X).
\end{eqnarray*}
This combined with Proposition \ref{prop:compare} yields the conclusion. 
 \end{proof}

\begin{remark}
We may define the $(p,q)-$Bergman metric by
  $$
 B_{p,q}(z;X):={K_q(z)^{-\frac1p}}\cdot {\sup}_f\, |Xf(z)|
 $$
 where the supremum is taken over all $f\in A^p(\Omega)$ with $f(z)=0$ and $\|f\|_p=1$.  Analogously, we may verify that 
$$
B_{\Omega_1,p,q}(z;X)=B_{\Omega_2,p,q}(F(z);F_\ast X)
$$
for any biholomorphic mapping $F:\Omega_1\rightarrow \Omega_2$  between bounded simply-connected domains. \end{remark} 

The $p-$Bergman kernel can be used to produce various invariant\/  {\it K\"ahler} metrics.   
Following Narasimhan-Simha \cite{NS}, we introduce the following weighted Bergman space
 $$
 A^2_p(\Omega):=\left\{f\in \mathcal O(\Omega):\int_\Omega {|f|^2}/{K_{p}}<\infty\right\}.
 $$
 Let $K_{2,p}(z)$ denote the  Bergman kernel associated to $A^2_p(\Omega)$.
Then  
 $$
 ds^2_{2/m}:=\sum_{j,k=1}^n \frac{\partial^2 \log K_{2,2/m}(z)}{\partial z_j\partial\bar{z}_k} dz_j\otimes d\bar{z}_k
 $$
 gives an invariant K\"ahler metric on $\Omega$ (cf.  \cite{Sakai}; see also \cite{ChenInvariant}). 

There are  two interesting functions related to the limiting case $p=0$. The first one, which is introduced by Tsuji \cite{Tsuji}, is defined to be
$$
K_0(z):=\left(\limsup_{m\rightarrow \infty} K_{2/m}(z)\right)^\ast
$$
where $(\cdot)^\ast$ denotes the upper semicontinuous regularization. The second one, which arises from Siu's work on invariance of plurigenera \cite{Siu}, is defined  by
$$
\widehat{K}_0(z):=\sum_{m=1}^\infty \varepsilon_m K_{2/m}(z)
$$
where $\{\varepsilon_m\}$ is a sequence of positive numbers satisfying $\sum \varepsilon_m<\infty$. By (\ref{eq:BergIneq_3}) we see that both $K_0$ and $\widehat{K}_0$ are well-defined so that $\log K_0$ and $\log \widehat{K}_0$ are psh on $\Omega$; moreover, $K_0$ always dominates $\widehat{K}_0$ while the latter is  continuous.
By Proposition \ref{prop:trans_2} we immediately obtain
 \begin{eqnarray*}
 K_{\Omega_1,0}(z) & = & K_{\Omega_2,0}(F(z))|J_F(z)|^2\\
  \widehat{K}_{\Omega_1,0}(z) & = & \widehat{K}_{\Omega_2,0}(F(z))|J_F(z)|^2
 \end{eqnarray*}
 for any biholomorphic mapping $F:\Omega_1\rightarrow \Omega_2$.

Analogously, we may introduce the following weighted Bergman spaces
  $$
  A^2_0 (\Omega):=\left\{f\in \mathcal O(\Omega):\int_\Omega {|f|^2}/{K_{0}}<\infty\right\}
  $$
  $$
  \widehat{A^2_0 (\Omega)}:=\left\{f\in \mathcal O(\Omega):\int_\Omega {|f|^2}/{\widehat{K}_{0}}<\infty\right\}.
  $$
  Let $K_{2,0}(z)$ (resp. $\widehat{K_{2,0}}(z)$) denote the  Bergman kernel associated to $A^2_0(\Omega)$ (resp. $\widehat{A^2_0 (\Omega)}$). It is not difficult to see that
   $$
 ds^2_0:=\sum_{j,k=1}^n \frac{\partial^2 \log K_{2,0}(z)}{\partial z_j\partial\bar{z}_k} dz_j\otimes d\bar{z}_k
 $$
  $$
 \widehat{ds^2_0}:=\sum_{j,k=1}^n \frac{\partial^2 \log \widehat{K_{2,0}}(z)}{\partial z_j\partial\bar{z}_k} dz_j\otimes d\bar{z}_k
 $$
 are invariant K\"ahler metrics on $\Omega$. 
   
 \begin{problem}
  Is it possible to construct an invariant complete metric on any bounded pseudoconvex domain by using the $p-$Bergman kernel?
 \end{problem}

\section{Zeroes of $K_2(z,w)$ and non real-analyticity of $K_p(z)$}  

  \begin{proposition}\label{prop:Holder_5}
                      If\/ $\frac1p+\frac1q=\frac1r$, then
                       \begin{eqnarray}
                       m_r(z) & \le & m_p(z)\cdot m_q(z) \label{eq:Holder_5}\\
                       K_r(z)^{\frac1r} & \ge & K_p(z)^{\frac1p} \cdot K_q(z)^{\frac1q}. \label{eq:Holder_6}
                       \end{eqnarray}
                 \end{proposition}
                 
                       \begin{proof}
                   It suffices to verify \eqref{eq:Holder_5}.   Take two holomorphic functions $f_p$ and $f_q$ on $\Omega$ with $f_p(z)=f_q(z)=1$ and 
                   $$
                   \|f_p\|_p=m_p(z),\ \ \ \|f_q\|_q=m_q(z).
                   $$
                                     Set $f_r:=f_p f_q$. Then $f_r$ is a holomorphic function on $\Omega$ satisfying $f_r(z)=1$ and H\"older's inequality gives
                                                            $$
                       \|f\|_r\le \|f_p\|_p\cdot \|f_q\|_q=m_p(z)\cdot m_q(z).
                       $$ 
                      By definition of $m_r(z)$ we immediately get \eqref{eq:Holder_5}.   
                                                                \end{proof}
     \begin{proposition}\label{prop:NRA_1}
     Let $p\ge 1$ and $k\in \mathbb Z^+$. We have $K_p(z)=K_{pk}(z)$ if and only if $m_p(\cdot,z)=m_{pk}(\cdot,z)^k$.
     \end{proposition}  
     
     \begin{proof}
  Suppose $K_p(z)=K_{pk}(z)$.  Since $f_k:=m_{pk}(\cdot,z)^k$ is a holomorphic function satisfying $f_k(z)=1$ and 
    $$
    \int_\Omega |f_k|^p =\int_\Omega |m_{pk}(\cdot,z)|^{pk}=m_{pk}(z)^{pk}=m_p(z)^p,
    $$ 
    it follows from  uniqueness of the minimizer that $f_k=m_p(\cdot,z)$. 
    
    The other direction follows from
    $$
    m_p(z)^p = \int_\Omega |m_p(\cdot,z)|^p=\int_\Omega |m_{pk}(\cdot,z)|^{pk}=m_{pk}(z)^{pk}.
    $$
     \end{proof}    
     
     \begin{proposition}\label{prop:NRA_2}
   Suppose that $\Omega$ is a bounded simply-connected domain in $\mathbb C^n$ and $m_p(\cdot,z)$ is zero-free for some $p\ge 1$ and $z\in \Omega$. Then  
   \begin{enumerate}
   \item[$(1)$] \ \ \ $K_s(z) = K_p(z)$\/ for any $s\ge p$.
   \item[$(2)$] \ \ \ $m_s(\cdot,z) = m_p(\cdot,z)^{p/s}$\/ for any $s\ge p$.
   \end{enumerate}
     \end{proposition}  
     
     \begin{proof}
    (1) By the hypothesis we may define $f_{p,s}:=m_p(\cdot,z)^{\frac{p}s}\in \mathcal O(\Omega)$ with $f_{p,s}(z)=1$. Since
    $$
    \int_\Omega |f_{p,s}|^s = \int_\Omega |m_p(\cdot,z)|^p=m_p(z)^p,
    $$
    we have 
    \begin{equation}\label{eq:NRA_1}
    K_s(z) \ge \frac{|f_{p,s}(z)|^s}{\|f_{p,s}\|_s^s} = \frac1{m_p(z)^p}=K_p(z).
    \end{equation}
    On the other hand, Proposition \ref{prop:Holder_5} yields that if $\frac1s+\frac1t=\frac1p$ then
    \begin{eqnarray*}
    K_p(z)^{\frac1p} & \ge & K_s(z)^{\frac1s}\cdot K_t(z)^{\frac1t}\\
    & \ge & K_s(z)^{\frac1s}\cdot K_p(z)^{\frac1t} 
       \end{eqnarray*}
       in view of \eqref{eq:NRA_1} since $t\ge p$. Thus we get $K_p(z)\ge K_s(z)$.
       
       (2) Note that $f_{p,s}\in \mathcal O(\Omega)$ satisfies $f_{p,s}(z)=1$ and 
       $$
       \int_\Omega |f_{p,s}|^s=m_p(z)^p=m_s(z)^s.
       $$
       Uniqueness of the minimizer gives $f_{p,s}=m_s(\cdot,z)$.
       
     \end{proof}  
     
     An immediate consequence is
     
     \begin{corollary}
     Let $\Omega$ be a bounded simply-connected domain. If $K_2(\cdot,w)$ is not zero-free for some $w\in \Omega$,  then $K_p(\cdot,w)$ is also not zero-free for all $1\le p\le 2$.   
     \end{corollary}     
     
     For a bounded domain $\Omega\subset \mathbb C^n$ we set 
     $$
     \mathcal F(\Omega):=\left\{z\in \Omega: K_2(\cdot,z) \ \text{is\ zero-free} \right\}\ \ \ \text{and}\ \ \ \mathcal N(\Omega):=\Omega\backslash \mathcal F(\Omega).
     $$ 
  Since $K_2(\zeta,z)=\overline{K_2(z,\zeta)}$, we conclude that if $K_2(\zeta,z)=0$ then $\zeta,z\in \mathcal N(\Omega)$. 
     
     \begin{proposition}
     $\mathcal F(\Omega)$ is a closed subset and $\mathcal N(\Omega)$ is an open subset.
     \end{proposition}
     
     \begin{proof}
     Suppose on the contrary that $\mathcal N(\Omega)$ is not open, i.e., there exist $z_0\in \mathcal N(\Omega)$ and a sequence of points $\{z_j\}\subset \mathcal F(\Omega)$ such that $z_j\rightarrow z_0$ as $j\rightarrow \infty$. Since $K_2(\cdot,z_j)\rightarrow K_2(\cdot,z_0)$, it follows from Hurwitz's theorem that either $z_0\in \mathcal F(\Omega)$ or $K_2(\cdot,z_0)\equiv 0$. The latter can not happen since $K_2(z_0,z_0)>0$. Thus we get a contradiction. 
     \end{proof}
     
     For a set $E$ we denote by $E^\circ$  the set of inner points of $E$. Then we have
     
      \begin{theorem}\label{th:NRA_2}
     Let $\Omega$ be a bounded simply-connected domain in $\mathbb C^n$ such that both $\mathcal N(\Omega)$ and $\mathcal F(\Omega)^\circ$ are nonempty. Then the following properties hold:
     \begin{enumerate}
     \item[$(1)$] There exists  $k_0\in \mathbb Z^+$ such that $K_{2k}(z)$ is not real-analytic in $\Omega$ for any integer $k\ge k_0$. 
     \item[$(2)$] Suppose furthermore that there exists $\zeta_0,z_0\in \mathcal N(\Omega)$ with ${\rm ord}_{z_0}K_2(\zeta_0,\cdot)=1$. Then  $K_{2k}(z)$ is not real-analytic in $\Omega$ for any integer $k > 2$. Moreover,  either ${\rm Re}\,m_{p}(\zeta_0,\cdot)$ or ${\rm Im}\,m_{p}(\zeta_0,\cdot)$ cannot be real-analytic in $\Omega$ for any rational $p>2$.
          \end{enumerate}
     \end{theorem}
     
     \begin{proof}
     (1) By Proposition \ref{prop:NRA_2}, we have $K_{2k}(z)=K_2(z)$ for any $z\in \mathcal F(\Omega)$.  Suppose on the contrary that $K_{2k}(z)$ is real-analytic in $\Omega$. Then $K_{2k}(z)=K_2(z)$ for any $z\in \Omega$ by the uniqueness theorem for  real-analytic functions since $\mathcal F(\Omega)^\circ\neq \emptyset$. It then follows from Proposition \ref{prop:NRA_1} that 
     \begin{equation}\label{eq:NRA_4}
     m_2(\cdot,z)=m_{2k}(\cdot,z)^k,\ \ \ \forall\,z\in \Omega.
     \end{equation}
      Take $\zeta_0,z_0\in \mathcal N(\Omega)$ so that $K_2(\zeta_0,z_0)=0$. Then $m_2(\zeta_0,z_0)=0$, which implies $m_{2k}(\zeta_0,z_0)=0$. Set $k_0':={\rm ord}_{z_0}K_2(\zeta_0,\cdot)={\rm ord}_{z_0}m_2(\zeta_0,\cdot)$. Since $\overline{K_2(\zeta_0,\cdot)}$ is holomorphic and not identically zero, we conclude that $k_0'<\infty$.  Since $m_{2k}(\zeta_0,\cdot)$ is locally $\frac12-$H\"older continuous in view of Theorem \ref{th:reg_1} in the next section, we see that  for $k\ge k_0:=2k_0'+2$,
     $$
     {\rm ord}_{z_0}m_2(\zeta_0,\cdot)<{\rm ord}_{z_0}m_{2k}(\zeta_0,\cdot)^k,
     $$
    which is contradictory to \eqref{eq:NRA_4}.
    
    (2) For any $k > 2$, we have
    $$
    {\rm ord}_{z_0}m_2(\zeta_0,\cdot)=1< {\rm ord}_{z_0}m_{2k}(\zeta_0,\cdot)^k,
            $$
            so that the first assertion follows analogously as (1). For the second assertion we write $p=k/l$ for $k,l\in \mathbb Z^+$. By Proposition \ref{prop:NRA_2}, we have $K_{p}(z)=K_2(z)$ and $m_p(\zeta_0,z)^k=m_2(\zeta_0,z)^{2l}$ for any $z\in \mathcal F(\Omega)$. Suppose on the other contrary that both ${\rm Re}\,m_{p}(\zeta_0,\cdot)$ and ${\rm Im}\,m_{p}(\zeta_0,\cdot)$ are real-analytic in $\Omega$. Since $\mathcal F(\Omega)^\circ\neq \emptyset$, it follows from the  uniqueness theorem for real-analytic functions  that
            $$
            m_p(\zeta_0,z)^k=m_2(\zeta_0,z)^{2l},\ \ \ \forall\,z\in \Omega.
                        $$    
                        On the other hand, we have
                        $$
                        {\rm ord}_{z_0}m_2(\zeta_0,\cdot)^{2l} = 2l <k \le {\rm ord}_{z_0}m_{2k}(\zeta_0,\cdot)^k 
                        $$
                        since $m_{2k}(\zeta_0,\cdot)$ is real-analytic,
                      which is  a contradiction.
                            \end{proof}

          \begin{remark}
        Actually we have proved a stronger conclusion that $K_{2k}(z)-K_2(z)$ does not enjoy the unique continuation property, i.e., it vanishes identically if it vanishes on a nonempty open subset.   There are plenty of non real-analytic functions which still verify the unique continuation property. 
                  \end{remark}

\begin{proposition}\label{prop:Reinhardt}
Let $\Omega$ be a bounded complete Reinhardt domain in $\mathbb C^n$. Then there exists $\varepsilon>0$ such that 
$$
B_\varepsilon (0):=\left\{z\in \mathbb C^n: |z|<\varepsilon \right\}\subset \mathcal F(\Omega).
$$
\end{proposition}

\begin{proof}
Note that 
$$
K_2(z,w) =\sum_\alpha a_\alpha z^\alpha \bar{w}^\alpha,\ \ \ a_\alpha=1/\int_\Omega |z^\alpha|^{2}.
$$
Take $r\ll  1$ so that $B_r(0)\subset \Omega$. The series expansion above implies that 
$$
K_2(z,w) = K_2(rz,w/r), \ \ \ \forall\,z\in \Omega,\ w\in B_{r^2}(0)
$$
(this observation is essentially due to S. R. Bell). Thus for $r\ll 1$ there exists a constant $C_r>0$ such that 
$$
|K_2(z,w)|\le C_r,\ \ \ \forall\,z\in \Omega,\ w\in B_{r^2}(0),
$$
and Cauchy's estimates gives 
$$
|K_2(z,0)-K_2(z,w)|\le C_r |w|, \ \ \ \forall\,z\in \Omega,\ w\in B_{r^2/2}(0).
$$
Since $K_2(z,0)=1/|\Omega|$, it follows that if $\varepsilon$ is sufficiently small then $K_2(z,w)\neq 0$ for all $z\in \Omega$ and $w\in B_\varepsilon(0)$, i.e., 
$
B_\varepsilon (0)\subset \mathcal F(\Omega).
$
\end{proof}

\begin{remark}
Since complete Reinhardt domains are always simply-connected, we see that Theorem \ref{th:NRA_0} follows from Theorem \ref{th:NRA_2} and Proposition \ref{prop:Reinhardt}.
\end{remark}

Next we will show that  the following Thullen-type domain
$$
\Omega=\left\{(z_1,z_2): |z_1|+|z_2|^{2/\alpha}<1 \right\}
$$
  verifies the hypothesis of Theorem \ref{th:NRA_2}/(2) for every $\alpha>2$. It is known from (9) in \cite{BFS} that 
  \begin{eqnarray*}
  K_2((z_1^2,0),(w_1^2,0)) & = & \frac1{4\alpha \pi^2 x}\cdot \frac{\partial}{\partial x^2}\left[\frac1{(1-x)^\alpha}-\frac1{(1+x)^\alpha}\right]\\
  & = & \frac{\alpha(\alpha+1)}{4\alpha \pi^2 x}\cdot \left[\frac1{(1-x)^{\alpha+2}}-\frac1{(1+x)^{\alpha+2}}\right]  \end{eqnarray*}
  when $x:=z_1 \bar{w}_1\neq 0$. It is easy to see that if $\alpha>2$ then equation
  $$
  \frac1{(1-x)^{\alpha+2}}=\frac1{(1+x)^{\alpha+2}} 
   $$
   has a solution with $0<|x|<1$ if 
   $$
   \frac1{1-x}=\frac{e^{2\pi i/(\alpha+2)}}{1+x},\ \ \ \text{i.e.,}\ \ \ x=\frac{e^{2\pi i/(\alpha+2)}-1}{e^{2\pi i/(\alpha+2)}+1}=:x_\alpha.
   $$
   Note that the function $\eta(x):=(1-x)^{-\alpha-2}-(1+x)^{-\alpha-2}$ satisfies
   \begin{eqnarray*}
   \eta'(x_\alpha) & = & (\alpha+2)\left[\frac1{(1-x_\alpha)^{\alpha+3}}+\frac1{(1+x_\alpha)^{\alpha+3}}\right]\\
   & = & \frac{\alpha+2}{(1+x_\alpha)^{\alpha+3}}\left[e^{2\pi i\cdot \frac{\alpha+3}{\alpha+2}}+1\right]\\
   & \neq & 0
   \end{eqnarray*}
   for $\alpha>2$. Thus for $z_{1,\alpha}=\bar{w}_{1,\alpha}=\sqrt{x_\alpha}$ the order of $K_2((x_\alpha,0),\cdot)$ at the point 
     $(\bar{x}_\alpha,0)$ equals $1$.

We conclude this section by a remark as follows.  Following a suggestion of the referee, it is reasonable to study the relationship between the set
\[
\mathcal{A}_p(\Omega):=\left\{z\in\Omega:K_p(\cdot)\text{ is not real-analytic at }z\right\}
\]
and $\mathcal F(\Omega)$ or $\mathcal N(\Omega)$.  We claim that if $\Omega$ is a bounded simply-connected  domain such that both $\mathcal N(\Omega)$ and $\mathcal F(\Omega)^\circ$ are nonempty then
\[
\partial\mathcal{F}(\Omega)^\circ\cap\Omega\ (=\partial\overline{\mathcal{N}(\Omega)}\cap\Omega)\subset\mathcal{A}_{2k}(\Omega)
\]
holds for any integer  $k>2.$ 
To see this,  suppose on the contrary that there exists a point $w_0\in (\partial\mathcal{F}(\Omega)^\circ\cap\Omega) \backslash \mathcal{A}_{2k}(\Omega)$; then there is a  constant $\varepsilon>0$ such that 
$$
B(w_0,\varepsilon)\cap\mathcal{F}(\Omega)^\circ\neq\emptyset,\ \ \ B(w_0,\varepsilon)\cap\mathcal{N}(\Omega)\neq\emptyset,
$$
and $K_{2k}(\cdot)$ is real-analytic on $B(w_0,\varepsilon)$. Since $K_{2k}(z)=K_2(z)$ for $z\in{B(w_0,\varepsilon)\cap\mathcal{F}(\Omega)^\circ}$ in view of Proposition \ref{prop:NRA_2}, it follows from the uniqueness theorem for real-analytic functions that $K_{2k}(z)=K_2(z)$ for all $z\in B(w_0,\varepsilon)$.  Take $z_0\in{B(w_0,\varepsilon)\cap\mathcal{N}(\Omega)}$ and $\zeta_0\in\Omega$ such that $K_2(\zeta_0,z_0)=0$.  Since $K_2(\zeta_0,\zeta_0)\neq0$,  it follows that $S:=\{K_2(\zeta_0,\cdot)=0\}$ is a complex analytic hypersurface which contains $z_0$.  Hence there exists $z_0'\in{S}\cap{B(w_0,\varepsilon)}$ such that $\mathrm{ord}_{z_0'}K_2(\zeta_0,\cdot)=1$.  By an argument analogous to the proof of Theorem \ref{th:NRA_2},  we get a contradiction.

\section{H\"older continuity of $m_p(z,\cdot)$ and $K_p(z,\cdot)$}

Throughout this section we always assume that $p\ge 1$.
Let us
introduce a useful auxiliary function as follows
$$
H_p(z,w):=K_p(z) + K_p(w) -{\rm Re}\left\{K_p(z,w)+K_p(w,z)\right\}.
$$
Clearly, 
$
H_p(z,w)=H_p(w,z)$ {and} $H_p(z,z)=0$.
Moreover, we have

\begin{proposition}\label{prop:H-Lip}
For every compact set $S\subset \Omega$, there exists a constant $C>0$ such that
$$
|H_p(z,w)|\le C|z-w|,\ \ \ \forall\,z,w\in S.
$$
\end{proposition}

\begin{proof}
It suffices to verify
$$
|K_p(z,w)-K_p(w)|\le C|z-w|,\ \ \ \forall\,z,w\in S.
$$
This follows from Cauchy's estimates since $K_p(\cdot,w)$ is holomorphic and
 uniformly bounded in a small neighborhood of $S$ in view of \eqref{eq:Holder_3}.
\end{proof}

A less obvious observation is

\begin{theorem}\label{th:H-positive}
We have $H_p(z,w)\ge 0$ and equality holds if and only if $z=w$.
\end{theorem}

For the proof of Theorem \ref{th:H-positive} and the regularity results in the sequel, the following "elementary" inequalities play a key role.

\begin{proposition}\label{prop:eleIneqs}
Let $a,b\in \mathbb C$.  The following inequalities hold.
\begin{enumerate}
\item[$(1)$] For $p\ge 2$ we have
\begin{eqnarray}\label{eq:eleIneq_1}
&& {\rm Re}\left\{(|b|^{p-2}\bar{b}-|a|^{p-2}\bar{a})(b-a)\right\}\\
 & \ge & \frac12 (|b|^{p-2}+|a|^{p-2})|b-a|^2\nonumber\\
 & \ge &  2^{1-p} |b-a|^p;\nonumber
\end{eqnarray}
\item[$(2)$] For $1\le p\le 2$ we have
\begin{eqnarray}\label{eq:eleIneq_2}
&& {\rm Re}\left\{(|b|^{p-2}\bar{b}-|a|^{p-2}\bar{a})(b-a)\right\} \\
& \ge & (p-1) |b-a|^2(|a|+|b|)^{p-2}\nonumber \\
&& + (2-p) |{\rm Im}(a\bar{b})|^2(|a|+|b|)^{p-4};\nonumber
\end{eqnarray}
\item[$(3)$] For $p>2$ we have
\begin{eqnarray}\label{eq:eleIneq_3}
|b|^p  & \ge & |a|^p+p{\rm Re}\left\{|a|^{p-2}\bar{a}(b-a)\right\}
+\frac1{4^{p+3}}|b-a|^p;
\end{eqnarray}
\item[$(4)$] For $1<p\le 2$ we have
\begin{eqnarray}\label{eq:eleIneq_4}
|b|^p  & \ge & |a|^p+p{\rm Re}\left\{|a|^{p-2}\bar{a}(b-a)\right\}
 +A_p |b-a|^2(|a|+|b|)^{p-2}
\end{eqnarray}
where $A_p=\frac{p}2\min\{1,p-1\};$
\item[$(5)$] For $p=1$ we have
\begin{eqnarray}\label{eq:eleIneq_5}
|b|  & \ge & |a|+{\rm Re}\left\{|a|^{-1}\bar{a}(b-a)\right\}
 +A_1 |{\rm Im}(\bar{a}b)|^2(|a|+|b|)^{-3}
\end{eqnarray}
where $A_1>0$ is some numerical constant.
\end{enumerate}
\end{proposition}

These inequalities have their origins in  nonlinear analysis of the $p-$Laplacian (see e.g., \cite{Lind1}). We will provide the proof in Appendix.

\begin{lemma}\label{lm:H-positive_1}
For $p\ge 2$ we have
\begin{equation}\label{eq:H-positive_1}
|m_p(z,w)-m_p(z,w')|^p \le \frac{2^{p-1} K_p(z)}{K_p(w)K_p(w')}\cdot H_p(w,w').
\end{equation}
\end{lemma}

\begin{proof}
Substitute at first $a=m_p(\cdot,w)$ and $b=m_p(\cdot,w')$ into (\ref{eq:eleIneq_1}) then take integration over $\Omega$, we obtain
\begin{eqnarray*}
 && 2^{1-p} \int_\Omega  |m_p(\cdot,w')-m_p(\cdot,w)|^p\\
& \le & \int_\Omega | m_p(\cdot,w') |^p -{\rm Re} \int_\Omega |m_p(\cdot,w')|^{p-2}\overline{m_p(\cdot,w')}\,m_p(\cdot,w)\\
&& \\
&& +\int_\Omega | m_p(\cdot,w) |^p-{\rm Re} \int_\Omega |m_p(\cdot,w)|^{p-2}\overline{m_p(\cdot,w)}\,m_p(\cdot,w')\\
& = &  m_p(w')^p+m_p(w)^p-{\rm Re}\left\{m_p(w')^p m_p(w',w)+m_p(w)^p m_p(w,w')\right\}\\
& = & \frac{H_p(w',w)}{K_p(w)K_p(w')},
\end{eqnarray*}
where the first equality follows from the reproducing formula.
Since
\begin{equation}\label{eq:meanvalue}
|m_p(z,w')-m_p(z,w)|^p\le K_p(z)\int_\Omega  |m_p(\cdot,w')-m_p(\cdot,w)|^p,
\end{equation}
we immediately get (\ref{eq:H-positive_1}).
\end{proof}

\begin{lemma}\label{lm:H-positive_2}
For $1<p\le 2$ we have
\begin{equation}\label{eq:H-positive_2}
|m_p(z,w)-m_p(z,w')|^p \le \frac{C_p K_p(z)}{K_p(w')K_p(w)}\left[K_p(w')+K_p(w)\right]^{1-\frac{p}2} H_p(w,w')^{{\frac{p}2}}
\end{equation}
where $C_p=2^{\frac{(p-1)(2-p)}2}/(p-1)^{\frac{p}2}$.
\end{lemma}

\begin{proof}
Let $f_1,f_2\in A^p(\Omega)$.  H\"older's inequality yields
\begin{eqnarray*}
\int_\Omega |f_2-f_1|^p & = & \int_\Omega |f_2-f_1|^p (|f_2|+|f_1|)^{p(p-2)/2}(|f_2|+|f_1|)^{p(2-p)/2}\\
& \le & \left\{\int_\Omega |f_2-f_1|^2 (|f_2|+|f_1|)^{p-2} \right\}^{\frac{p}2} \left\{\int_\Omega (|f_2|+|f_1|)^{p} \right\}^{1-\frac{p}2}\\
& \le &  \left\{\frac1{p-1}\int_\Omega {\rm Re}\left[ (|f_2|^{p-2}\bar{f}_2-|f_1|^{p-2}\bar{f}_1)(f_2-f_1)\right] \right\}^{\frac{p}2} \left\{\int_\Omega (|f_2|+|f_1|)^{p} \right\}^{1-\frac{p}2}
\end{eqnarray*}
in view of (\ref{eq:eleIneq_2}). Take $f_2=m_p(\cdot,w')$ and $f_1=m_p(\cdot,w)$, we obtain
\begin{eqnarray*}
 \int_\Omega |m_p(\cdot,w')-m_p(\cdot,w)|^p
& \le & \left\{\frac1{p-1}  \frac{H_p(w',w)}{K_p(w')K_p(w)}\right\}^{\frac{p}2}
 \cdot \left\{ 2^{p-1} \left[ m_p(w')^p+m_p(w)^p\right] \right\}^{1-\frac{p}2}\\
& = &  \frac{C_p}{K_p(w')K_p(w)} \left[ K_p(w')+K_p(w)\right]^{1-\frac{p}2}
 H_p(w',w)^{\frac{p}2}.
\end{eqnarray*}
This combined with (\ref{eq:meanvalue}) gives (\ref{eq:H-positive_2}).
\end{proof}

\begin{lemma}\label{lm:H-positive_3}
\begin{equation}\label{eq:H-positive_3}
\int_\Omega \frac{|{\rm Im} \{ m_1(\cdot,w')\overline{m_1(\cdot,w)} \}|^2}{(|m_1(\cdot,w'))|+|m_1(\cdot,w)|)^3}\le \frac{H_1(w,w')}{K_1(w')K_1(w)}.
\end{equation}
\end{lemma}

\begin{proof}
It suffices to take $b=m_p(\cdot,w)$, $a=m_p(\cdot,w')$ in (\ref{eq:eleIneq_2}) with $p=1$,  then take integration over $\Omega$.
\end{proof}

\begin{proof}[Proof of Theorem \ref{th:H-positive}]
Lemma \ref{lm:H-positive_1}$\sim$\ref{lm:H-positive_3} give $H_p(z,w)\ge 0$. Now suppose $H_p(z,w)=0$. It follows from Lemma \ref{lm:H-positive_1} and Lemma \ref{lm:H-positive_2} that $m_p(\cdot,z)=m_p(\cdot,w)$ whenever $p>1$, so that $z=w$ in view of Proposition \ref{prop:indep}.  It remains to deal with the case $p=1$. Consider the following proper analytic subset in $\Omega$
$$
S_{z,w}:=\{m_p(\cdot,z)=0\}\cup \{m_p(\cdot,w)=0\}.
$$
Let $V\subset\subset U\subset\subset \Omega\backslash S_{z,w}$ be two open sets. We take 
$$
h:=m_1(\cdot,z)/m_1(\cdot,w)\in \mathcal O(U).
$$
Since $|{\rm Im\,}h|^2$ is subharmonic on $U$, it follows from the mean-value inequality and Lemma \ref{lm:H-positive_3} that for every $\zeta\in V$
\begin{eqnarray*}
|{\rm Im\,}h(\zeta)|^2  \lesssim \int_U |{\rm Im\,}h|^2 & = & \int_U |{\rm Im} \{ m_1(\cdot,z)\overline{m_1(\cdot,w)} \}|^2 |m_1(\cdot,w)|^{-4}\\
& \lesssim &   \int_U \frac{|{\rm Im} \{ m_1(\cdot,z)\overline{m_1(\cdot,w)} \}|^2}{(|m_1(\cdot,z))|+|m_1(\cdot,w)|)^3}\\
& \lesssim & \frac{H_1(z,w)}{K_1(z)K_1(w)}=0.
\end{eqnarray*}
Since $V$ and $U$  can be arbitrarily chosen, we conclude that ${\rm Im\,}h= 0$ on the domain $\Omega\backslash S_{z,w}$, so that $h={\rm const.}$, i.e., $m_p(\cdot,z)=cm_p(\cdot,w)$ holds on $\Omega\backslash S_{z,w}$ for some complex number $c$. By continuity, the same equality remains valid on $\Omega$, so that $z=w$ in view of Proposition \ref{prop:indep}.
\end{proof}

\begin{theorem}\label{th:reg_1}
 For any  $p>1$ and any  compact set $S\subset \Omega$, there exists a constant $C>0$ such that
\begin{equation}\label{eq:reg_1}
|m_p(z,w)-m_p(z,w')|\le C|w-w'|^{\frac12}, \ \ \ \forall\,z,w,w'\in S.
\end{equation}
The same conclusion also holds for $K_p$.
\end{theorem}

\begin{proof}
It suffices to verify the conclusion for $m_p(z,w)$ since $K_p(z,w)=m_p(z,w)K_p(w)$ and $K_p(z)$ is locally Lipschitz continuous.  Lemma \ref{lm:H-positive_2} together with Proposition \ref{prop:H-Lip} immediately yield \eqref{eq:reg_1} for $1<p\le 2$.
Analogously, Lemma \ref{lm:H-positive_1} combined with  Proposition \ref{prop:H-Lip}  gives
$$
|m_p(z,w)-m_p(z,w')| \lesssim |w-w'|^{\frac1p}
$$
for $p> 2$, which is weaker than \eqref{eq:reg_1}, however. 
We have to adopt another approach.
Substitute at first $a=m_p(\cdot,w)$ and $b=m_p(\cdot,w')$ into (\ref{eq:eleIneq_1}) then take integration over $\Omega$, we obtain
$$
  \int_\Omega (|m_p(\cdot,w')|^{p-2}+|m_p(\cdot,w)|^{p-2}) |m_p(\cdot,w')-m_p(\cdot,w)|^2
 \le  \frac{2H_p(w,w')}{K_p(w')K_p(w)} \lesssim |w-w'|,
  $$
  which implies
  $$
   \int_\Omega |m_p(\cdot,w)|^{p-2} |m_p(\cdot,w')-m_p(\cdot,w)|^2  \lesssim |w-w'|.
     $$
  Now fix an open set $U$ with 
  $
  S\subset U\subset\subset \Omega.
  $  
  Since $\{m_p(\cdot,w):w\in S\}$ is a continuous family of holomorphic functions on $\Omega$ (in view of Proposition \ref{prop:continuity}),  it follows from the celebrated theorem of Demailly-Koll\'ar \cite{DK} on semi-continuity of the complex singularity exponent  that there exist positive constants $c=c(S)$ and $M=M(S)$ such that
$$
\int_U |m_p(\cdot,w)|^{-c}\le M,\ \ \ \forall\,w\in S.
$$
Fix
$
\alpha:= \frac{2c}{p-2+c}<2.
$
 By H\"older's inequality, we have
\begin{eqnarray*}
\int_U |m_p(\cdot,w')-m_p(\cdot,w)|^\alpha & \le & \left\{\int_U |m_p(\cdot,w)|^{p-2} |m_p(\cdot,w')-m_p(\cdot,w)|^2 \right\}^{\alpha/2}\\
&& \cdot \left\{\int_U |m_p(\cdot,w)|^{-\frac{p-2}2\alpha\cdot\frac2{2-\alpha}}\right\}^{1-\alpha/2}\\
& \le &  \left\{\int_\Omega |m_p(\cdot,w)|^{p-2} |m_p(\cdot,w')-m_p(\cdot,w)|^2 \right\}^{\alpha/2}\\
&& \cdot  \left\{\int_U |m_p(\cdot,w)|^{-c}\right\}^{1-\alpha/2}\\
& \lesssim &  |w-w'|^{\alpha/2}.
\end{eqnarray*}
This combined with the mean-value inequality  gives (\ref{eq:reg_1}). 
\end{proof}

\begin{theorem}\label{th:reg_2}
\begin{enumerate}
\item[$(1)$]
Let $S_w:=\{m_1(\cdot,w)=0\}$. For every open set $U$ with $w\in U\subset\subset \Omega\backslash S_w$,  there exists a constant $C>0$ such that
\begin{equation}\label{eq:reg_3}
|m_1(z,w)-m_1(z,w')|\le C|w-w'|^{\frac12},\ \ \ \forall\,z,w'\in U.
\end{equation}
\item[$(2)$] For every open set $U\subset\subset \Omega$,  there exists a constant $C>0$ such that
\begin{equation}\label{eq:reg_4}
|m_1(z,w)-m_1(z,w')|\le C|w-w'|^{\frac1{2(n+2)}},\ \ \ \forall\,z,w,w'\in U.
\end{equation}
\end{enumerate}
The same conclusions also hold for $K_1$.
\end{theorem}

\begin{proof}
$(1)$ First of all, (\ref{eq:H-positive_3}) implies
$$
\int_\Omega \frac{|{\rm Im} \{ m_1(\cdot,w')\overline{m_1(\cdot,w)} \}|^2}{(|m_1(\cdot,w'))|+|m_1(\cdot,w)|)^3}\lesssim {H_1(w',w)}\lesssim |w-w'|.
$$
Take open sets $U',U''$ with $U\subset\subset U'\subset\subset U''\subset\subset \Omega\backslash S_w$.  It follows from the mean-value inequality that
\begin{eqnarray*}
\sup_{U'} \left|{\rm Im}\left\{\frac{m_1(\cdot,w')}{m_1(\cdot,w)}\right\}\right|^2 &\lesssim& \int_{U''} \left|{\rm Im}\left\{\frac{m_1(\cdot,w')}{m_1(\cdot,w)}\right\}\right|^2\\
&\lesssim & \int_\Omega \frac{|{\rm Im} \{ m_1(\cdot,w')\overline{m_1(\cdot,w)} \}|^2}{(|m_1(\cdot,w'))|+|m_1(\cdot,w)|)^3}\\
& \lesssim & |w-w'|.
 \end{eqnarray*}
 Set $h:= \frac{m_1(\cdot,w')}{m_1(\cdot,w)}-1$. Then we have
 $$
 \sup_{U'} |{\rm Im\,}h|\lesssim |w-w'|^{\frac12}.
 $$
 To proceed the proof, we need the following cerebrated Borel-Carath\'eodory inequality
 \begin{equation}\label{eq:BC}
 \sup_{\Delta_r} |f|\le \frac{2r}{R-r}\, \sup_{\Delta_r} {\rm Im\,}f +\frac{R+r}{R-r}\, |f(0)|
 \end{equation}
 where $\Delta_r=\{z\in \mathbb C: |z|<r\}$ and $f\in \mathcal O(\Delta_R)$ for some $R>r$. Take a ball $B_R(w)\subset\subset U'$. Apply (\ref{eq:BC}) to every complex line through $w$, we obtain
 $$
 \sup_{B_{\frac{R}2}(w)} |h| \le 2\sup_{B_{{R}}(w)} {\rm Im\,}h +3 |h(w)|.
  $$
  Note that
  $$
  |h(w)|=|m_1(w,w') -1|=|m_1(w,w')-m_1(w',w')|\lesssim |w-w'|.
  $$
  Thus we obtain
  $$
   \sup_{B_{\frac{R}2}(w)} |h|\lesssim |w-w'|^{\frac12}.
     $$
     Take a chain of balls connecting $w$ and $z$ and apply \eqref{eq:BC} analogously on each ball, we obtain
     $$
     |h(z)|\lesssim |w-w'|^{\frac12},
     $$
     from which (\ref{eq:reg_3}) immediately follows.
     
 $(2)$ Fix $\alpha>0$ for a moment.  Take an open set $U'$ with $U\subset\subset U'\subset\subset \Omega$.  Given   two distinct points $w,w'$ in $U$,  set   $\varepsilon:=|w-w'|^{\alpha}$.   If $\max\{|m_1(z,w)|,|m_1(z,w')|\}\le \varepsilon$,  then
   $$
   |m_1(z,w)-m_1(z,w')|\le 2\varepsilon.
   $$ 
   Now consider the case $\max\{|m_1(z,w)|,|m_1(z,w')|\}> \varepsilon$.  Without loss of generality,  we assume that $|m_1(z,w)|>\varepsilon$.  Since $|m_1(z,w)-m_1(z',w)|\le C|z-z'|$ for $z,z'\in U$,  we have 
   $$
   |m_1(z',w)|>\varepsilon/2,\ \ \ \forall\,z': |z'-z|\le \tilde{\varepsilon}:=\varepsilon/(2C).
   $$
   Here and in what follows  $C$ denotes a generic constant depending only on $U,U',\Omega$.  Since
   $$
\int_\Omega \frac{|{\rm Im} \{ m_1(\cdot,w')\overline{m_1(\cdot,w)} \}|^2}{(|m_1(\cdot,w'))|+|m_1(\cdot,w)|)^3} \le C |w-w'|,
$$
it follows from the mean-value inequality that
\begin{eqnarray*}
\sup_{B_{\tilde{\varepsilon}/2}(z)} \left|{\rm Im}\left\{\frac{m_1(\cdot,w')}{m_1(\cdot,w)}\right\}\right|^2 &\le& C \varepsilon^{-2n} \int_{B_{\tilde{\varepsilon}}(z)} \left|{\rm Im}\left\{\frac{m_1(\cdot,w')}{m_1(\cdot,w)}\right\}\right|^2\\
& = &  C \varepsilon^{-2n} \int_{B_{\tilde{\varepsilon}}(z)} \frac{|{\rm Im} \{ m_1(\cdot,w')\overline{m_1(\cdot,w)} \}|^2}{ |m_1(\cdot,w)|^2}\\
& \le &  C \varepsilon^{-2n-2} \int_{B_{\tilde{\varepsilon}}(z)} |{\rm Im} \{ m_1(\cdot,w')\overline{m_1(\cdot,w)} \}|^2\\
&\le & C \varepsilon^{-2n-2} \int_\Omega \frac{|{\rm Im} \{ m_1(\cdot,w')\overline{m_1(\cdot,w)} \}|^2}{(|m_1(\cdot,w'))|+|m_1(\cdot,w)|)^3}\\
& \le & C |w-w'|^{1-2(n+1)\alpha}.
 \end{eqnarray*}
 Similar as above,  we may deduce from  the Borel-Carath\'eodory inequality that
 $$
\sup_{B_{\tilde{\varepsilon}/4}(z)} |h|\le C|w-w'|^{\frac12-(n+1)\alpha}.
 $$
 Take $\alpha=\frac1{2(n+2)}$,  we obtain 
 $$
 |m_1(z,w')-m_1(z,w)|=|h(z)||m_1(z,w)|\le C|w-w'|^{\frac1{2(n+2)}}.
 $$
\end{proof}

\begin{problem}
 Let $z\in \Omega$ be fixed. Are $m_1(z,\cdot)$ and $K_1(z,\cdot)$ locally $\frac12-$H\"older continuous on $\Omega$?
\end{problem}

\begin{problem}
What about the metric structure or analytic structure of the level set $\{m_p(z,\cdot)=c\}$ where $c\in \mathbb C$?
\end{problem}

\section{ $B_p(z;X)$ and $i\partial\bar{\partial} \log K_p(z;X)$}

For a real-valued upper semicontinuous function $u$ defined on a domain $\Omega\subset \mathbb C^n$,  we define the generalized Levi form  of $u$ by
$$
i\partial\bar{\partial} u(z;X):=\liminf_{r\rightarrow 0+}\frac1{r^2}\left\{\frac1{2\pi}\int_0^{2\pi}u(z+re^{i\theta}X)d\theta-u(z)\right\}
$$
where we identify $X=\sum_j X_j\partial/\partial z_j$ with  $(X_1,\cdots,X_n)$ for the sake of simplicity.   Note that if $u$ is $C^2$ then $i\partial\bar{\partial} u(z;X)$ is the standard Levi form of $u$.

\begin{theorem}\label{th:Levi_1}
For every $p\le 2$ we have
\begin{equation}\label{eq:Levi_1}
i\partial\bar{\partial} \log K_p(z;X) \ge  C(z;X)^2.
\end{equation}
\end{theorem}

We need the following simple fact.

 \begin{lemma}\label{lm:var_1}
 For every $p>0$ we have
 \begin{equation}\label{eq:Var_1}
 \int_\Omega |m_p(\cdot,z)|^{p}f  = 0
  \end{equation}
  for all $f\in A^\infty(\Omega)$ with $f(z)=0$ and $\|f\|_\infty=1$.
 \end{lemma}
 \begin{proof}
 Given $f\in A^\infty(\Omega)$ with $f(z)=0$ and $\|f\|_\infty=1$, we define
 $$
 f_t={m_p(\cdot,z)} (1+tf),\ \ \ t\in \mathbb C.
 $$
 Clearly,  $f_t$ belongs to $A^p(\Omega)$ and satisfies $f_t(z)=1$. We then have
 \begin{eqnarray*}
 m_p(z)^p & \le & \int_\Omega |f_t|^p = \int_\Omega |m_p(\cdot,z)|^p\left(1+2{\rm Re}\{tf\}+|t|^2|f|^2\right)^{\frac{p}2}\\
 & = &  \int_\Omega |m_p(\cdot,z)|^p\left(1+p{\rm Re}\{t f\}+O(|t|^2)\right)\\
 & = &  m_p(z)^p+p{\rm Re}\left\{ t \int_\Omega |m_p(\cdot,z)|^{p}f \right\}+O(|t|^2) \end{eqnarray*}
as $t\rightarrow 0$. Take $t=se^{-i\arg \int_\Omega |m_p(\cdot,z)|^{p}f}$ with $s\in \mathbb R$, we immediately get (\ref{eq:Var_1}).
\end{proof}

\begin{proof}[Proof of Theorem \ref{th:Levi_1}]Fix $r>0$ and $\theta\in \mathbb R$ for a moment.
For $t\in \mathbb C$ with $|t|=O(r)$ and $f\in A^\infty(\Omega)$ with $f(z)=0$ and $\|f\|_\infty=1$ we define
 $
 f_t={m_p(\cdot,z)} (1+tf)
 $
 as above. We have
 \begin{eqnarray*}
 |f_t(z+re^{i\theta}X)|^p & = & |m_p(z+re^{i\theta}X,z)|^p |1+tf(z+re^{i\theta}X)|^p\\
 & = & |m_p(z+re^{i\theta}X,z)|^p |1+t re^{i\theta} Xf(z)+O(r^3)|^p\\
 & = &  |m_p(z+re^{i\theta}X,z)|^p \left(1+p{\rm Re}\left\{t re^{i\theta} Xf(z)\right\}+O(r^3)\right)
 \end{eqnarray*}
 and
 \begin{eqnarray*}
 \|f_t\|_p^p & = & \int_\Omega |m_p(\cdot,z)|^p \left(1+2{\rm Re}\{tf\}+|tf|^2\right)^{\frac{p}2}\\
 & = & \int_\Omega |m_p(\cdot,z)|^p \left(1+p{\rm Re}\{tf\} +\frac{p}2|tf|^2+\frac{p(p-2)}2({\rm Re}\{tf\})^2+O(r^3)\right)\\
 & = & \int_\Omega |m_p(\cdot,z)|^p \left(1+p{\rm Re}\{tf\} +\frac{p^2}4 |tf|^2+\frac{p(p-2)}8\left((tf)^2+(\overline{tf})^2\right)+O(r^3)\right)\\
 &\le & m_p(z)^p\left(1+\frac{p^2|t|^2}4 + O(r^3)\right)
  \end{eqnarray*}
  in view of Lemma \ref{lm:var_1}. 
  
  Take $t=\varepsilon r e^{-i\theta}\overline{Xf(z)}$ where $\varepsilon>0$ is a constant to be determined later. Then we have
  $$
  |f_t(z+re^{i\theta}X)|^p=  |m_p(z+re^{i\theta}X,z)|^p \left(1+p\varepsilon r^2 |Xf(z)|^2+O(r^3)\right)
  $$
  and
  $$
  \|f_t\|_p^p\le m_p(z)^p\left(1+\frac{p^2}4 \varepsilon^2 r^2 |Xf(z)|^2 + O(r^3)\right),
    $$
    so that
    $$
    K_p(z+re^{i\theta}X)\ge \frac{ |f_t(z+re^{i\theta}X)|^p}{\|f_t\|_p^p}\ge \frac{ |m_p(z+re^{i\theta}X,z)|^p }{m_p(z)^p}\cdot\frac{1+p\varepsilon r^2 |Xf(z)|^2+O(r^3)}{1+\frac{p^2}4 \varepsilon^2 r^2 |Xf(z)|^2 + O(r^3)}.
    $$
    Thus we have
    \begin{eqnarray*}
    i\partial\bar{\partial} \log K_p(z;X) & \ge & \liminf_{r\rightarrow 0+} \frac1{r^2}\left\{\frac1{2\pi}\int_0^{2\pi} \log  \frac{ |m_p(z+re^{i\theta}X,z)|^p }{m_p(z)^p}\,d\theta-\log K_p(z) \right\}\\
    && + p\varepsilon|Xf(z)|^2-\frac{p^2\varepsilon^2}4 |Xf(z)|^2\\
    &\ge & p\varepsilon(1-p\varepsilon/4) |Xf(z)|^2\\
    & = & |Xf(z)|^2
    \end{eqnarray*}
    when $\varepsilon=2/p$.   Here the second inequality follows by applying the mean-value inequality to the subharmonic function 
    $$
    u(t):=\log |m_p(z+tX,z)/m_p(z)|^p
    $$ 
    with $u(0)=\log 1/m_p(z)^p=\log K_p(z)$. Take supremum over $f$, we obtain (\ref{eq:Levi_1}).
 \end{proof}

 \begin{theorem}\label{th:Levi_2}
For every $p\ge 2$ we have
\begin{equation}\label{eq:Levi_2}
i\partial\bar{\partial} \log K_p(z;X)\ge B_p(z;X)^2.
\end{equation}
\end{theorem}

\begin{proof}
Fix $r,\theta$ for a moment. Take $f\in A^p(\Omega)$ with $f(z)=0$ and $\|f\|_p=1$. For $t\in \mathbb C$ with $|t|=O(r)$   we define
$$
f_t:=m_p(\cdot,z)+tf.
$$
Analogously we have
\begin{eqnarray*}
|f_t(z+re^{i\theta}X)|^p & = & |m_p(z+re^{i\theta}X,z)+tf(z+re^{i\theta}X)|^p\\
 & = & |m_p(z+re^{i\theta}X,z)+t re^{i\theta} Xf(z)+O(r^3)|^p\\
 & = &  |m_p(z+re^{i\theta}X,z)|^p \left(1+p{\rm Re}\left\{t re^{i\theta} Xf(z)\right\}+O(r^3)\right)
\end{eqnarray*}
since $m_p(z+re^{i\theta}X,z)=1+O(r)$. A straightforward calculation yields
\begin{eqnarray*}
\frac{\partial |f_t|^p}{\partial t} & = & \frac{p}2 |f_t|^{p-2} \overline{f_t}f\\
\frac{\partial^2 |f_t|^p}{\partial t^2} & = & \frac{p(p-2)}4 |f_t|^{p-4}\left( \overline{f_t}f\right)^2\\
\frac{\partial^2 |f_t|^p}{\partial t\partial \bar{t}} & = & \frac{p^2}4 |f_t|^{p-2} |f|^2.
\end{eqnarray*}
Set $J(t):=\|f_t\|_p^p$. From the proof of Lemma \ref{lm:var_2} we have already known that
$$
\frac{\partial J}{\partial t}(0) = \frac{\partial J}{\partial \bar{t}}(0)=0.
$$
Note that for $|t|\le 1$ we have
$$
\left|\frac{\partial^2 |f_t|^p}{\partial t^2}\right| \le \frac{p(p-2)}4 |f_t|^{p-2} |f|^2\le \frac{p(p-2)}4 (|m_p(\cdot,z)|+|f| )^{p-2} |f|^2=: \frac{p(p-2)}4\cdot g
$$
$$
\left|\frac{\partial^2 |f_t|^p}{\partial t \partial \bar{t}}\right| \le \frac{p^2}4\cdot g,
$$
while H\"older's inequality gives
$$
\int_\Omega g\le \|f\|_p^2 \left\{\int_\Omega (|m_p(\cdot,z)|+|f| )^p\right\}^{1-\frac2p}<\infty.
$$
It follows from the dominated convergence theorem that
$$
\frac{\partial^2 J}{\partial t^2}(0) = \int_\Omega \frac{\partial^2 |f_t|^p}{\partial t^2}(0)=\frac{p(p-2)}4 \int_\Omega |m_p(\cdot,z)|^{p-4}\left(\overline{m_p(\cdot,z)}\,f\right)^2
$$
$$
\frac{\partial^2 J}{\partial t\partial \bar{t}}(0) = \int_\Omega \frac{\partial^2 |f_t|^p}{\partial t\partial\bar{t}}(0)=\frac{p^2}4 \int_\Omega |m_p(\cdot,z)|^{p-2}|f|^2.
$$
Thus we have
\begin{eqnarray*}
J(t) & = & J(0) + {\rm Re}\left\{\frac{\partial^2 J}{\partial t^2}(0) t^2\right\}
+\frac{\partial^2 J}{\partial t\partial \bar{t}}(0) |t|^2 +o(|t|^2)\\
& = & m_p(z)^p + {\rm Re}\left\{\frac{\partial^2 J}{\partial t^2}(0) t^2\right\} + \frac{p^2}4 |t|^2 \int_\Omega |m_p(\cdot,z)|^{p-2}|f|^2 +o(|t|^2)\\
& \le & m_p(z)^p + {\rm Re}\left\{\frac{\partial^2 J}{\partial t^2}(0) t^2\right\} + \frac{p^2}4 m_p(z)^{p-2} |t|^2  + o(|t|^2)
\end{eqnarray*}
in view of H\"older's inequality. Take $t=\varepsilon re^{-i\theta} \overline{Xf(z)}$ as above, we obtain
\begin{eqnarray*}
 &&   K_p(z+re^{i\theta}X)\\
  & \ge & \frac{ |m_p(z+re^{i\theta}X,z)|^p }{m_p(z)^p}\\
    && \cdot\frac{1  +p\varepsilon r^2 |Xf(z)|^2+O(r^3)}{1 + {\rm Re}\left\{\frac{\partial^2 J}{\partial t^2}(0) (\varepsilon re^{-i\theta} \overline{Xf(z)})^2/m_p(z)^p\right\} +\frac{p^2}4 \varepsilon^2 r^2 |Xf(z)|^2/m_p(z)^2 + o(r^2)}.
    \end{eqnarray*}
    Since 
    $$
    \int_0^{2\pi}  {\rm Re}\left\{\frac{\partial^2 J}{\partial t^2}(0) (\varepsilon re^{-i\theta} \overline{Xf(z)})^2/m_p(z)^p\right\} d\theta=0,
    $$
    it follows that
    \begin{eqnarray*}
     i\partial\bar{\partial} \log K_p(z;X) & \ge & p\varepsilon |Xf(z)|^2-\frac{p^2}4 \frac{\varepsilon^2 |Xf(z)|^2}{m_p(z)^2}\\
     & = & p\varepsilon |Xf(z)|^2 \left(1-\frac{p}4 \frac{\varepsilon}{m_p(z)^2}\right) \\
     & = &  |Xf(z)|^2 m_p(z)^2
        \end{eqnarray*}
    when $\varepsilon=\frac2p\,m_p(z)^2$. Take supremum over $f$, we obtain (\ref{eq:Levi_2}).
    \end{proof}
 
 \begin{problem}
 Is it possible to conclude that $i\partial\bar{\partial} \log K_p(z;X) = B_p(z;X)^2$ for $2\le p<\infty$?
 \end{problem}

 By Theorem \ref{th:Levi_1} and Theorem \ref{th:Levi_2},  $\log K_p$ is a continuous strictly psh function on $\Omega$. In particular, we have
 
 \begin{corollary}
 The minimal set of $K_p(z)$ defined by 
 $$
 {\rm Min}_p(\Omega):= \{z\in \Omega: K_p(z)=\inf_{\zeta\in \Omega}\, K_p(\zeta)\}
 $$
  is either empty or a totally real subset of\/ $\Omega$.
 \end{corollary}
 
 \begin{problem}
 What about the metric structure of\/  ${\rm Min}_p(\Omega)$?
 \end{problem}

\section{Stability of $m_p$, $K_p$ and $B_p$ as $p$ varies}

We first prove the following 

 \begin{proposition}\label{th:stab_1}
 \begin{enumerate}
 \item[$(1)$]
 $\lim_{s\rightarrow p\pm} K_s(z)$ exists for $p>0$ and
\begin{equation}\label{eq:stab_1}
\lim_{s\rightarrow p+} K_s(z)\le K_p(z) = \lim_{s\rightarrow p-} K_s(z).
\end{equation}
\item[$(2)$] 
 If there exists  $p'>p$ such that $A^{p'}(\Omega)$ lies dense in $A^p(\Omega)$, then
\begin{equation}\label{eq:stab_3}
K_p(z)=\lim_{s\rightarrow p} K_s(z).
\end{equation}
\end{enumerate}
 \end{proposition}

 \begin{proof}
 (1) Let $0<p<s<\infty$ and $f\in A^s(\Omega)$. By H\"older's inequality, we have
  $$
  \|f\|_p \le \|f\|_s \cdot |\Omega|^{\frac1p-\frac1s},
  $$
  so that
  $$
  K_p(z)^{\frac1p}\ge \frac{|f(z)|}{\|f\|_p} \ge \frac{|f(z)|}{\|f\|_s\cdot |\Omega|^{\frac1p-\frac1s}}.
  $$
  Take supremum over $f\in A^s(\Omega)$, we obtain
  \begin{equation}\label{eq:decreasing}
  |\Omega|^{\frac1p}\cdot K_p(z)^{\frac1p}\ge    |\Omega|^{\frac1s}\cdot K_s(z)^{\frac1s}.
    \end{equation}
  It follows that both $\lim_{s\rightarrow p-}K_s(z)$ and $\lim_{s\rightarrow p+} K_s(z)$ exist and
  $$
  K_p(z)\le \lim_{s\rightarrow p-}K_s(z);\ \ \ K_p(z)\ge \lim_{s\rightarrow p+}K_s(z).
      $$
   To achieve the inequality
  $$
  K_p(z)\ge \lim_{s\rightarrow p-}K_s(z),
    $$
  we first take $f_s\in A^s(\Omega)$ with $\|f_s\|_s=1$ and $|f_s(z)|=K_s(z)^{\frac1s}$. Clearly $\{f_s\}$ forms a normal family so that there exists a sequence $s_j\uparrow p$ with $f_{s_j}$ converging locally uniformly to some ${f}_p\in \mathcal O(\Omega)$. By Fatou's lemma, we have
  $$
  \int_\Omega |{f}_p|^p \le \liminf_{j\rightarrow \infty} \int_\Omega |f_{s_j}|^p
  \le \liminf_{j\rightarrow \infty} \left[\int_\Omega |f_{s_j}|^{s_j}\right]^{\frac{p}{s_j}}\cdot |\Omega|^{1-\frac{p}{s_j}}=1.
    $$
  It follows that
  $$
  K_p(z)\ge \frac{|{f}_p(z)|^p}{\|{f}_p\|_p^p}\ge |{f}_p(z)|^p=\lim_{j\rightarrow\infty} |f_{s_j}(z)|^p = \lim_{j\rightarrow\infty} K_{s_j}(z)^{\frac{p}{s_j}} = \lim_{j\rightarrow\infty} K_{s_j}(z)=\lim_{s\rightarrow p-}K_s(z).
  $$

  (2) Let $\mathcal M_p(\cdot,z)$ be the maximizer in \eqref{eq:max_1}. Suppose there exists a sequence $\{f_j\}\subset A^{p'}(\Omega)$ for some $p'>p$ such that
  $$
  \int_\Omega |f_j-\mathcal M_p(\cdot,z)|^p\rightarrow 0
  $$
  as $j\rightarrow \infty$. It  follows  that for every $0<\varepsilon<1$ there exists $j_\varepsilon\in \mathbb Z^+$ such that
  $$
  \|f_{j_\varepsilon}\|_p\le 1+\varepsilon
  $$
  and
  $$
  |f_{j_\varepsilon}(z)|\ge (1-\varepsilon) K_p(z)^{\frac1p}
    $$
    in view of the mean-value inequality.
  Since
  $$
  |f_{j_\varepsilon}|^s\le 1+|f_{j_\varepsilon}|^{p'}\in L^1(\Omega)
  $$
  for every $s\le p'$, we have
  $$
  \int_\Omega |f_{j_\varepsilon}|^p=\lim_{s\rightarrow p+} \int_\Omega |f_{j_\varepsilon}|^s
    $$
    in view of the dominated convergence theorem.
  Thus
  $$
  \lim_{s\rightarrow p+} K_s(z)\ge \lim_{s\rightarrow p+} \frac{|f_{j_\varepsilon}(z)|^s}{\int_\Omega |f_{j_\varepsilon}|^s}=\frac{|f_{j_\varepsilon}(z)|^p}{\int_\Omega |f_{j_\varepsilon}|^p} \ge  \left(\frac{1-\varepsilon}{1+\varepsilon}\right)^p K_p(z).
   $$
   Since $\varepsilon$ can be arbitrarily small, we obtain $ \lim_{s\rightarrow p+} K_s(z)=K_p(z)$ so that (\ref{eq:stab_3}) holds.
    \end{proof}

    A bounded domain $\Omega$ in $\mathbb C^n$ is said to have positive hyperconvexity index if there exists a negative continuous psh function $\rho$ on $\Omega$ satisfying $-\rho\lesssim \delta^\alpha$ for some $\alpha>0$ (cf. \cite{ChenH-index}). It follows from Proposition 1.4 in \cite{ChenH-index}  that if $\Omega$ has positive hyperconvexity index then $A^p(\Omega)$ lies dense in $A^2(\Omega)$ for some $p>2$. Thus we have

\begin{corollary}\label{cor:stab_2}
If\/ $\Omega$ has positive hyperconvexity index, then
$$
K_2(z)=\lim_{p\rightarrow 2} K_p(z).
$$
 \end{corollary}

 Let $E$ be a compact set in $\mathbb C$. Let $\mathcal O({E})$ be the set of  functions which are holomorphic in a neighborhood of $E$ and $L^p_a(E)$ the set of functions in $L^p(E)$ which are holomorphic in $E^\circ$.  It is well-known \cite{HedbergApprox} that $\mathcal O({E})$ lies dense in $L^p_a(E)$ for $p\in [1,2)$. Since $\mathcal O({E})\subset A^{s}(E^\circ)$ for every $s>0$, we have

 \begin{corollary}\label{cor:Stab_3}
 If\/ $\Omega$ is a bounded fat domain in $\mathbb C$, i.e., $\overline{\Omega}^\circ=\Omega$, then  \eqref{eq:stab_3}  holds for $p\in [1,2)$.
 \end{corollary}

The following proposition can be used to produce plenty of examples with
\begin{equation}\label{eq:NC}
K_2(z)>\lim_{p\rightarrow 2+} K_p(z).
\end{equation}

\begin{proposition}\label{prop:NC}
Let $\Omega=D\backslash S$ where $D$ is a bounded domain in $\mathbb C$ and $S$ is a compact set in $D$ which has positive $2-$capacity but zero $p-$capacity for every $p<2$. Then \eqref{eq:NC} holds.
\end{proposition}

Recall that the $p-$capacity of $S$ is given by
$$
{\rm Cap}_p(S):=\inf_\phi \int_{\mathbb C} |\nabla \phi|^p
$$
where the infimum is taken over all $\phi\in C_0^\infty(\mathbb C)$ such that $\phi\ge 1$ on $S$. The condition of Proposition \ref{prop:NC} is satisfied for instance, if the $h-$Hausdorff measure $\Lambda_h(S)$ of $S$ is positive and finite where $h(t)=(\log1/t)^{-\alpha}$ for some $\alpha>1$.

 \begin{proof}[Proof of  Proposition \ref{prop:NC}]
   It is a classical result  that $A^p(\Omega)=A^p(D)$ if and only if ${\rm Cap}_q(S)=0$ where $\frac1p+\frac1q=1$ (cf. \cite{Carleson} and \cite{Hedberg}). Hence
   $$
   K_{\Omega,p}(z)=K_{D,p}(z)\le C_n d(S,\partial D)^{-2n}
   $$
   for all $z\in \Omega$ with $d(z,S)\le d(S,\partial D)/2$. On the other hand, since ${\rm Cap}_2(S)>0$, i.e., $S$ is non-polar, so there exists a regular point $a\in S$ for the Dirichlet problem on $\Omega$. Then we have
   $$
   \lim_{z\rightarrow a} K_{\Omega,2}(z)=\infty
   $$
   (see \cite{Ohsawa}). Thus we obtain
   $$
   K_{\Omega,2}(z)> \lim_{p\rightarrow 2+} K_{\Omega,p}(z)
      $$
      whenever $z$ is sufficiently close to $a$.
   \end{proof}

\begin{theorem}\label{th:Stab_3}
\begin{enumerate}
\item[$(1)$]
 For $p>1$ we have
\begin{equation}\label{eq:Stab_3}
\lim_{s\rightarrow p-} m_s(z,w) = m_p(z,w).
\end{equation}
\item[$(2)$]
$
\lim_{s\rightarrow p+} m_s(z,w)
$
exists. Moreover, if $A^{p'}(\Omega)$ lies dense in $A^p(\Omega)$ for some $p'>p$, then
\begin{equation}\label{eq:r-continuity_1}
m_p(z,w)=\lim_{s\rightarrow p+} m_s(z,w).
\end{equation}
\end{enumerate}
The same conclusions also hold for $K_p$.
\end{theorem}

\begin{proof}
$(1)$ We first consider the case $p>2$. Fix $2< s<p$. Apply  (\ref{eq:eleIneq_3}) with $a=m_s(\cdot,w)$, $b=m_p(\cdot,w)$, and $p$ replaced by $s$, we obtain
\begin{eqnarray*}
 && C_s \int_\Omega  |m_p(\cdot,w)-m_s(\cdot,w)|^s\\
& \le & \int_\Omega | m_p(\cdot,w) |^s - \int_\Omega | m_s(\cdot,w) |^s\\
&& -p{\rm Re} \int_\Omega |m_s(\cdot,w)|^{s-2}\,\overline{m_s(\cdot,w)} \left[ m_p(\cdot,w)-m_s(\cdot,w) \right]\\
& \le & |\Omega|^{1-\frac{s}p} \left[\int_\Omega | m_p(\cdot,w) |^p\right]^{\frac{s}p}-m_s(w)^s\\
& = & |\Omega|^{1-\frac{s}p} m_p(w)^s-m_s(w)^s.
\end{eqnarray*}
Thus we have
\begin{eqnarray*}
|m_p(z,w)-m_s(z,w)|^s & \le & K_s(z)\int_\Omega  |m_p(\cdot,w)-m_s(\cdot,w)|^s\\
& \le & C_s^{-1} K_s(z)\left\{ |\Omega|^{1-\frac{s}p} m_p(w)^s-m_s(w)^s\right\}.
\end{eqnarray*}
This combined with  (\ref{eq:stab_1}) gives (\ref{eq:Stab_3}).

Now suppose $1<p\le 2$. Fix $1<s<p$. Let $f_1,f_2\in A^s(\Omega)$. H\"older's inequality yields
\begin{eqnarray}\label{eq:off_Ineq}
\int_\Omega |f_2-f_1|^s & = & \int_\Omega |f_2-f_1|^s (|f_2|+|f_1|)^{s(s-2)/2}(|f_2|+|f_1|)^{s(2-s)/2}\nonumber\\
& \le & \left\{\int_\Omega |f_2-f_1|^2 (|f_2|+|f_1|)^{s-2} \right\}^{\frac{s}2} \left\{\int_\Omega (|f_2|+|f_1|)^{s} \right\}^{1-\frac{s}2}\nonumber\\
& \le &  \left\{A_s^{-1}\int_\Omega \left[ |f_2|^s-|f_1|^s -s{\rm Re}\left( |f_1|^{s-2}\bar{f}_1 (f_2-f_1)\right)\right] \right\}^{\frac{s}2} \nonumber\\
&& \cdot \left\{\int_\Omega (|f_2|+|f_1|)^{s} \right\}^{1-\frac{s}2}
\end{eqnarray}
in view of (\ref{eq:eleIneq_4}). Take $f_2=m_p(\cdot,w)$ and $f_1=m_s(\cdot,w)$, we obtain
\begin{eqnarray*}
&& |m_p(z,w)-m_s(z,w)|^s\\
 & \le & K_s(z) \int_\Omega |m_p(\cdot,w)-m_s(\cdot,w)|^s \\
& \le & K_s(z) \left\{A_s^{-1} \left[|\Omega|^{1-\frac{s}p} m_p(w)^s-m_s(w)^s\right] \right\}^{\frac{s}2}\\
&& \cdot \left\{ 2^{s-1} \left [ |\Omega|^{1-\frac{s}p} m_p(w)^s+m_s(w)^s\right] \right\}^{1-\frac{s}2},
\end{eqnarray*}
which gives (\ref{eq:Stab_3}).

$(2)$  We first consider the case $p\ge 2$. Let $p\le s<r$. Apply  (\ref{eq:eleIneq_3}) with $a=m_s(\cdot,w)$, $b=m_r(\cdot,w)$, and $p$ replaced by $s$, we obtain
\begin{eqnarray*}
 && C_s \int_\Omega  |m_r(\cdot,w)-m_s(\cdot,w)|^s\\
& \le & \int_\Omega | m_r(\cdot,w) |^s - \int_\Omega | m_s(\cdot,w) |^s\\
&& -p{\rm Re} \int_\Omega |m_s(\cdot,w)|^{s-2}\overline{m_s(\cdot,w)}\,(m_r(\cdot,w)-m_s(\cdot,w))\\
& \le & |\Omega|^{1-\frac{s}r} \left\{\int_\Omega | m_r(\cdot,w) |^r\right\}^{\frac{s}r}-m_s(w)^s\\
& = & |\Omega|^{1-\frac{s}r} m_r(w)^s-m_s(w)^s.
\end{eqnarray*}
Thus we have
\begin{eqnarray}\label{eq:r-continuity_3}
|m_r(z,w)-m_s(z,w)|^s & \le & K_s(z)\int_\Omega  |m_r(\cdot,w)-m_s(\cdot,w)|^s\nonumber\\
& \le & C_s^{-1} K_s(z) \left\{ |\Omega|^{1-\frac{s}r} m_r(w)^s-m_s(w)^s\right\}.
\end{eqnarray}
Since $\lim_{s\rightarrow p+} m_s(w)$ exists, it follows that $\{m_s(z,w)\}$ forms a Cauchy family as $s\rightarrow p+$, so that
$
\lim_{s\rightarrow p+} m_s(z,w)
$
also exists.

Next suppose $1<p< 2$. Let $p\le s<r<2$. Take $f_2=m_r(\cdot,w)$ and $f_1=m_s(\cdot,w)$ in (\ref{eq:off_Ineq}), we obtain
\begin{eqnarray}\label{eq:r-continuity_4}
 |m_r(z,w)-m_s(z,w)|^s
 & \le & K_s(z) \int_\Omega |m_r(\cdot,w)-m_s(\cdot,w)|^s\nonumber \\
& \le & K_s(z) \left\{A_s^{-1} \left[ |\Omega|^{1-\frac{s}r} m_r(w)^s-m_s(w)^s\right] \right\}^{\frac{s}2}\nonumber\\
&& \cdot \left\{ 2^{s-1} \left[ |\Omega|^{1-\frac{s}r} m_r(w)^s+m_s(w)^s \right] \right\}^{1-\frac{s}2},
\end{eqnarray}
from which the assertion immediately follows.

Finally, we deal with the case $p=1$. Let $1\le s<r$. By a similar argument as the case $p\ge 2$ with (\ref{eq:eleIneq_3}) replaced by   (\ref{eq:eleIneq_5}), we obtain
\begin{equation}\label{eq:off-param_4}
  A_1 \int_\Omega  \frac{|{\rm Im}\{m_r(\cdot,w)\overline{m_s(\cdot,w)}\}|^2}{(|m_r(\cdot,w)|+|m_s(\cdot,w)|)^3}
 \le  |\Omega|^{1-\frac{s}r} m_r(w)^s-m_s(w)^s.
\end{equation}
Since $\{m_r(\cdot,w)\}_r$ and $\{m_s(\cdot,w)\}_s$ are normal families, there exist subsequences $\{m_{r_j}(\cdot,w)\}$ and $\{m_{s_j}(\cdot,w)\}$ which converge locally uniformly  to certain holomorphic functions $h_1$ and $h_2$ respectively.  Clearly, we have $h_1(w)=h_2(w)=1$.

For fixed $w$ we set $S_w:=h_1^{-1}(0)\cup h_2^{-1}(0)$. For any open sets $V\subset\subset U \subset\subset \Omega\backslash S_w$ there exists a constant $C\gg 1$ such that
 $$
 C^{-1} \le \min\{m_{r_j}(z,w),m_{s_j}(z,w)\}\le \max\{m_{r_j}(z,w),m_{s_j}(z,w)\}\le C,\ \ \ \forall\,z\in U
 $$
 for all $j\gg 1$. Thus (\ref{eq:off-param_4}) combined with the mean-value inequality gives
 $$
 \left|{\rm Im}\,\frac{m_{r_j}(z,w)}{m_{s_j}(z,w)}\right|^2 \lesssim |\Omega|^{1-\frac{s_j}{r_j}} m_{r_j}(w)^{s_j}-m_{s_j}(w)^{s_j} \rightarrow 0\ \ \ (j\rightarrow \infty),\ \ \ \forall z\in V,
  $$
  which implies that ${\rm Im\,}\frac{h_1}{h_2}=0$ on  $V$, hence on $\Omega\backslash S_w$. It follows that $h_1/h_2={\rm const.}$ on $\Omega\backslash S_w$, hence on $\Omega$. As $h_1(w)=h_2(w)=1$,  we obtain $h_1=h_2$.

  In general, we consider two arbitrary convergent subsequences $\{m_{r^1_j}(\cdot,w)\}$ and $\{m_{r^2_j}(\cdot,w)\}$ of
  $\{m_r(\cdot,w)\}_r$.  Let $\{s_j\}$ be a subsequence of $\{r^2_j\}$ with $s_j<r^1_j$. By the above argument we know that the sequences $\{m_{r^1_j}(\cdot,w)\}$ and $\{m_{s_j}(\cdot,w)\}$ has the same limit, which is also the limit of $\{m_{r^2_j}(\cdot,w)\}$. Thus $\lim_{p\rightarrow 1+} m_p(z,w)$ exists.

 If $A^{p'}(\Omega)$ lies dense in $A^p(\Omega)$ for some $p'>p$, then we have $m_p(w)=\lim_{s\rightarrow p}m_p(w)$. Thus (\ref{eq:r-continuity_3}) together with (\ref{eq:r-continuity_4}) yield (\ref{eq:r-continuity_1}) for $p>1$. Apply (\ref{eq:off-param_4}) with $s=1$ in a similar way as above, we obtain (\ref{eq:r-continuity_1}) for $p=1$. 
 \end{proof}

\begin{corollary}\label{cor:zero}
Let $\Omega$ be a bounded domain with positive hyperconvexity index. If $K_2(\cdot,w)$ is not zero-free for some $w\in \Omega$, then there exists a number $\varepsilon=\varepsilon(w)>0$ such that $K_p(\cdot,w)$ is also not zero-free for $p\in (2-\varepsilon,2+\varepsilon)$.
\end{corollary}
\begin{proof}
Suppose on the contrary that there exists a sequence $p_j\rightarrow 2$ such that $K_{p_j}(\cdot,w)$ is zero-free for all $j$. It follows from Theorem \ref{th:Stab_3} and Hurwitz's theorem that $K_2(\cdot,w)\equiv 0$, which is clearly impossible.
\end{proof}

 \begin{proposition}\label{prop:stab_5}
 \begin{enumerate}
 \item[$(1)$]
 $\lim_{s\rightarrow p\pm} B_s(z;X)$ exists for $p>0$ and
\begin{equation}\label{eq:stab_5}
 B_p(z;X) = \lim_{s\rightarrow p-} B_s(z;X).
\end{equation}
\item[$(2)$] If there exists  $p'>p$ such that $A^{p'}(\Omega)$ lies dense in $A^p(\Omega)$, then
\begin{equation}\label{eq:stab_6}
B_p(z;X)=\lim_{s\rightarrow p} B_s(z;X).
\end{equation}
\end{enumerate}
 \end{proposition}

 \begin{proof}
 $(1)$ Similar as the proof of Proposition \ref{th:stab_1},  we have
    $$
  |\Omega|^{\frac1p}\cdot |X \mathcal M_p(\cdot,z;X)(z)|\ge    |\Omega|^{\frac1s}\cdot |X \mathcal M_s(\cdot,z;X)(z)|,
      $$
  so that $\lim_{s\rightarrow p\pm} |X \mathcal M_s(\cdot,z;X)(z)|$ exist and
  $$
 |X \mathcal M_p(\cdot,z;X)(z)| \le \lim_{s\rightarrow p-} |X \mathcal M_s(\cdot,z;X)(z)|
 $$
 $$
 |X \mathcal M_p(\cdot,z;X)(z)|\ge \lim_{s\rightarrow p+} |X \mathcal M_s(\cdot,z;X)(z)|.
      $$
      A  normal family argument  yields
      $$
      |X \mathcal M_p(\cdot,z;X)(z)| = \lim_{s\rightarrow p-} |X \mathcal M_s(\cdot,z;X)(z)|.
            $$
            This combined with Proposition \ref{th:stab_1} gives
          $$
          B_p(z;X)=\lim_{s\rightarrow p-} B_s(z;X).
                    $$

$(2)$ Suppose that
  $A^{p'}(\Omega)$ lies dense in $A^p(\Omega)$ for some $p'>p$. We may choose a sequence $f_j$ of functions in $A^{p'}(\Omega)$ such that
  $$
  \int_\Omega |f_j-\mathcal M_p(\cdot,z;X)|^p\rightarrow 0
  $$
  as $j\rightarrow \infty$. It follows that for every $0<\varepsilon<1$ there exists $j_\varepsilon\in \mathbb Z^+$ such that
  $$
  |f_{j_\varepsilon}(z)|\le \varepsilon, \ \ \ |Xf_{j_\varepsilon}(z)| \ge (1-\varepsilon) |X \mathcal M_p(\cdot,z;X)(z)|,
    $$
  and
  $$
\|f_{j_\varepsilon}\|_p\le 1+\varepsilon.
  $$
  Again we have
  $$
  \int_\Omega |f_{j_\varepsilon}|^p=\lim_{s\rightarrow p+} \int_\Omega |f_{j_\varepsilon}|^s.
    $$
    Since
    $$
  \|f_{j_\varepsilon}\|_p \le \|f_{j_\varepsilon}\|_s \cdot |\Omega|^{\frac1p-\frac1s}\ \ \ \text{and}\ \ \ \|f_{j_\varepsilon}\|_s \le \|f_{j_\varepsilon}\|_{p'} \cdot |\Omega|^{\frac1s-\frac1{p'}},
    $$
  it follows that
    $$
   \|f_{j_\varepsilon}\|_s - \|f_{j_\varepsilon}\|_p   \le    \|f_{j_\varepsilon}\|_s-\|f_{j_\varepsilon}\|_s^{\frac{s}p} + \|f_{j_\varepsilon}\|_s^{\frac{s}p} - \|f_{j_\varepsilon}\|_p  \rightarrow 0
   $$
   as $s\rightarrow p+$.
       By using the test function $f_{j_\varepsilon}-f_{j_\varepsilon}(z)$, we obtain
  \begin{eqnarray*}
  \lim_{s\rightarrow p+} \mathcal |X \mathcal M_s(\cdot,z;X)(z)| & \ge & \liminf_{s\rightarrow p+} \frac{|X f_{j_\varepsilon}(z)|}{ \|f_{j_\varepsilon}-f_{j_\varepsilon}(z)\|_s}\\
  & \ge & \liminf_{s\rightarrow p+} \frac{|X f_{j_\varepsilon}(z)|}{ \|f_{j_\varepsilon}\|_s+ |f_{j_\varepsilon}(z)| |\Omega|^{\frac1s}}\\
  &\ge &   \frac{1-\varepsilon}{1+\varepsilon+\varepsilon |\Omega|^{\frac1p}}\cdot |X \mathcal M_p(\cdot,z;X)(z)|.
   \end{eqnarray*}
  Thus (\ref{eq:stab_6}) holds.
 \end{proof}

\section{Boundary behavior of $K_p(z)$}
\subsection{An application of Grothendieck's theorem}

Let us first recall the following result of Grothendieck:
\begin{theorem}[cf.  \cite{Grothendieck} for $p\ge 1$; cf.  \cite{Rudin} for $p>0$]\label{Grothendieck}
	Suppose $0<p<\infty$, and 
	\begin{enumerate}
		\item [$(a)$] $(X, \mu)$ is a finite measure space, i.e. $\mu(X)<\infty$;
		\item [$(b)$] $E$ is a closed subspace of $L^p(X,\mu)$;
		\item [$(c)$] $E\subset L^{\infty}(X,\mu)$.
	\end{enumerate}
Then $E$ is finite-dimensional.
	\end{theorem}
	
\begin{corollary}[cf.  \cite{Grothendieck} for $p\ge 1$]
	Let $(X, \mu)$ be a measure space and $0< p<\infty$.  Suppose $\eta\in L^p(X,\mu)$ and $\eta>0$. If \begin{enumerate}
		\item [$(i)$] $E$ is a closed subspace of $L^p(X,\mu)$;
		\item [$(ii)$] $\frac{1}{\eta}\cdot E:=\{\frac1\eta\cdot x:x\in E\}\subset L^{\infty}(X,\mu)$,
	\end{enumerate}
then $E$ is finite-dimensional.
\end{corollary}
\begin{proof} Set $d\nu:=\eta^p d\mu$. Then $(X,\nu)$ is a finite measure space since $\eta\in L^p(X,\mu)$,  and 
$$
\frac{1}{\eta}\cdot E\subset L^{\infty}(X,\mu)\subset L^{\infty}(X,\nu).
$$
 On the other hand,  it is easy to verify that $\frac{1}{\eta}\cdot E$ is a closed subspace of $L^p(X,\nu)$.  Thus Theorem \ref{Grothendieck} implies that $\frac{1}{\eta}\cdot E$ is finite-dimensional,  so is $E$.
\end{proof}

\begin{proof}[Proof of Theorem \ref{p-kernel estimate}]
We may assume $L=\{(z_1,0,\cdots,0): z_1\in \mathbb C\}$.  It is easy to verify that $\pi_L(\Omega)$ is a bounded domain in $L$.  By the definition of $K_p(z)$ and the fact that $\eta\ge 1$,  it suffices to find $f\in A^p(\Omega)$  depending only on $z_1$ such that  $|f|^p/\eta$ is unbounded on $\Omega$.  
Let $l^p$ be the Banach space of sequences of complex numbers $\{c_j\}_{j=1}^\infty$ with 
$\sum_j |c_j|^p<\infty$.  We set $\phi_j(z)=z_1^j/\|z_1^j\|_{L^p(\Omega)}$ for $j\in \mathbb Z^+$ and
$$
E_p:=\left\{f={\sum}_j c_j \phi_j :  \{c_j\}_{j=1}^\infty\in l^1\ \text{for\ }p\ge 1;  \{c_j\}_{j=1}^\infty\in l^p\ \text{for\ } 0<p < 1\right\}.
$$
Set $f_N=\sum_{j=1}^N c_j\phi_j$.  
Since 
$$
\|f_{N+l}-f_N\|_{L^p(\Omega)}\le \sum_{j=1}^{l} \|c_{N+j} \phi_{N+j}\|_{L^p(\Omega)}=\sum_{j=1}^l  |c_{N+j}| \ \ \text{for\ }p\ge 1
$$
and 
$$
\|f_{N+l}-f_N\|_{L^p(\Omega)}^p \le \sum_{j=1}^l \|c_{N+j} \phi_{N+j}\|_{L^p(\Omega)}^p=\sum_{j=1}^l  |c_{N+j}|^p  \ \ \text{for\ }0<p< 1,
$$
it follows that $E_p$ is a well-defined linear subspace of $A^p(\Omega)$, which is clearly of infinite dimension.  Let $\overline{E}_p$ be the closure of $E_p$ in $A^p(\Omega)$.  For every $f\in \overline{E}_p$, there is a sequence $\{\widehat{f}_k\}\subset E_p$ such that $\widehat{f}_k\rightarrow f$ in $A^p(\Omega)$.  By the Bergman inequality we see that $\widehat{f}_k$ converges locally uniformly to $f$,  so that $f$ depends only on $z_1$.   Now suppose 
	$$
	\frac{1}{\eta^{1/p}} \cdot \overline{E}_p\subset L^{\infty}(\Omega).
	$$
	Since $\overline{E}_p$ is a closed subspace of $L^p(\Omega)$ and $\eta^{1/p}\in L^p(\Omega)$,  we have $\mathrm{dim}\, \overline{E}_p<\infty$ in view of Corollary \ref{Grothendieck},  which   is absurd  since $E_p$ is of infinite dimension.   Thus there exists $f\in \overline{E}_p$ such that $|f|^p/\eta$ is unbounded on $\Omega$. 
	\end{proof}
	
	Recall that the one-sided upper Minkowski dimension of $\partial \Omega$ is given by 
	$$
		\overline{\mathrm{dim}}_{\mathcal{M}}(\partial \Omega):=2n-\varliminf\limits_{\epsilon\to 0}\frac{\log|\Omega \backslash\Omega_{\epsilon}|}{\log \epsilon},
	$$
	where $\Omega_\epsilon:=\{z\in\Omega: \delta(z) >\epsilon\}$ and $|\cdot|$ denotes the volume.   
\begin{corollary}\label{cor:Minkowski}
 If\/ $\alpha< 2n-\overline{\mathrm{dim}}_{\mathcal{M}}(\partial \Omega)$,  then for every complex line $L$ in $\mathbb C^n$ there exists $z_L\in \partial \pi_L(\Omega) $ such that for any $w\in \partial \Omega\cap \pi_L^{-1}(z_L)$,
 $$
  \limsup_{z\rightarrow w} \delta(z)^{\alpha} K_p(z)=+\infty.
 $$
\end{corollary}
\begin{proof} It suffices to apply Theorem \ref{p-kernel estimate} with $\eta=C\delta^{-\alpha}$ for $C\gg 1$,  in view of the following elementary lemma.  
\end{proof}

\begin{lemma}\label{integrablity of distance function}
	If\/ $\alpha< 2n-\overline{\mathrm{dim}}_{\mathcal{M}}(\partial \Omega)$,  then $\delta^{-\alpha}\in L^{1}(\Omega)$.
\end{lemma} 
\begin{proof} Take $\alpha<\beta<2n-\overline{\mathrm{dim}}_{\mathcal{M}}(\partial \Omega)$.  Then we have 
$$
\varliminf\limits_{\epsilon\to 0}\frac{\log|\Omega \backslash\Omega_{\epsilon}|}{\log \epsilon}>\beta,
$$
 so that $|\Omega \backslash\Omega_{\epsilon}|\le C \epsilon^{\beta}$.  Thus
	\begin{align*}
		\int_{\Omega}\delta^{-\alpha}&=\int_{\Omega_1}\delta^{-\alpha}+\sum^{\infty}_{j=0}\int_{\Omega_{2^{-j-1}}\backslash\Omega_{2^{-j}}}\delta^{-\alpha}\\
		&\leq  C'+ \sum^{\infty}_{j=0} 2^{\alpha(j+1)}|\Omega_{2^{-j-1}}\backslash\Omega_{2^{-j}}|\\
		&\leq  C'+ 2^{\alpha} C \sum^{\infty}_{j=0} 2^{-(\beta-\alpha)j} \\
		& \leq \infty.
	\end{align*}
	\end{proof}
	
	\subsection{An application of $K_p$   to properly discontinuous groups}
For a domain $\Omega\subset\subset \mathbb C^n$ we denote by ${\rm Aut}(\Omega)$ the group of holomorphic automorphisms of $\Omega$. A subgroup $G$ of ${\rm Aut}(\Omega)$ is said to be properly discontinuous if for every compact set $S\subset \Omega$ there are only a finite number of elements $g\in G$ with $S\cap g(S)\neq \emptyset$. A well-known result of H. Cartan states that $\Omega/G$ is a complex space if $G$ is properly discontinuous.

Let $L(G_z)$ denote the set of limit points of $G_z:=\{g(z):g\in G\}$.
Set
$$
L(G):=\bigcup_{z\in G} L(G_z).
$$
Since $G$ is properly discontinuous, we have $L(G)\subset \partial \Omega$.

\begin{proposition}\label{th:p-holo}
Let $\Omega\subset \mathbb C^n$ be a bounded simply-connected domain and $G\subset {\rm Aut}(\Omega)$ a properly discontinuous group. For any $0<p<\infty$ and any $w\in L(G)$, there exists $f\in A^p(\Omega)$ such that
$$
\limsup_{z\rightarrow w} |f(z)|=\infty.
$$
\end{proposition}

\begin{proof}
By a classical result of Poincar\'e-Siegel (cf. \cite{Siegel}) we know that
\begin{equation}\label{eq:Siegel}
\sum_{g\in G} |J_g(z)|^2<\infty,\ \ \ \forall\,z\in \Omega.
\end{equation}
Take $z_0\in \Omega$ and $\{g_j\}\subset \Gamma$ such that $z_j:=g_j(z_0)\rightarrow w$. By Proposition \ref{prop:trans_2} we have
$$
K_p(z_0)=K_p(z_j) |J_{g_j}(z_0)|^2,
$$
so that
$$
\sum_{j=1}^\infty \frac{K_p(z_0)}{K_p(z_j)} =\sum_{j=1}^\infty |J_{g_j}(z)|^2<\infty
$$
in view of (\ref{eq:Siegel}), which implies
$$
\lim_{j\rightarrow \infty} K_p(z_j)=\infty.
$$

Suppose on the contrary that
$$
{\sup}_j\ |f(z_j)|<\infty,\ \ \ \forall\,f\in A^p(\Omega).
$$
By the Bergman inequality, we see that the continuous linear functional
$$
F_j: f\in A^p(\Omega)\mapsto f(z_j)
$$
 satisfies
$
\sup_j |F_j(f)|<\infty,
$
so that $\sup_j \|F_j\|<\infty$ in view of the Banach-Steinhaus theorem. Since
$$
\|F_j\|=\sup_{f\in A^p(\Omega)} \frac{|f(z_j)|}{\|f\|_p}=K_p(z_j)^{1/p},
$$
we obtain $\sup_j K_p(z_j)<\infty$, a contradiction!
\end{proof}

\begin{corollary}
  For any neighborhood $U$ of $w\in L(G)$,  the Hausdorff dimension of $\partial \Omega\cap U$ is no less than $2n-1$.
 \end{corollary}

 \begin{proof}
 Suppose on the contrary that there exist $\alpha<2n-1$ and a neighborhood $U$ of $w$ such that
 $$
 \Lambda_\alpha(\partial\Omega\cap U)=0,
 $$
 where $\Lambda_\alpha$ means the $\alpha-$dimensional Hausdorff measure. 
 It follows that $\partial\Omega\cap U$ is a polar set, so that $D:=U \backslash \partial\Omega$ is connected and $D\subset \Omega$. Since $K_{D,p}(z)\ge K_{\Omega,p}(z)$ for $z\in D$, we infer from the proof of Proposition \ref{th:p-holo} that
 $$
 \limsup_{z\rightarrow w} K_{D,p}(z)=\infty,\ \ \ \forall\, 0<p<\infty.
 $$
 On the other hand, thanks to a theorem on removable singularity due to Harvey-Polking \cite{HP}, we have $A^p(D)=A^p(U)$ for
  $
 p=\frac{2n-\alpha}{2n-\alpha-1},
 $
 so that
 $$
  \limsup_{z\rightarrow w} K_{D,p}(z)<\infty,
   $$
   a contradiction!
 \end{proof}

Let  $w\in L(G)$. We would like to ask  the following

 \begin{problem}
  Does there exist $f\in A^\infty(\Omega)$ which can not be extended holomorphically across $w$?
 \end{problem}

 \begin{problem}
 Is it possible to conclude that $\Lambda_{2n-1}(\partial \Omega\cap U)>0$ for any neighborhood $U$ of $w$?
 \end{problem}

\subsection{Comparison of $K_p(z)$ and $K_2(z)$}

\begin{theorem}\label{th:Comp_1}
Let $\Omega$ be a bounded pseudoconvex domain with $C^2-$boundary. Then there exist  constants $\gamma,C>0$ such that the following estimates hold near $\partial \Omega:$
\begin{eqnarray}
\frac{K_p(z)^{\frac1p}}{K_2(z)^{\frac12}} & \le & C\, \delta(z)^{\frac12-\frac1p} |\log \delta(z)|^{\frac{n(p-2)}{2p\gamma}},\ \ \   p>2,\label{eq:comp_1}\\
\frac{K_p(z)^{\frac1p}}{K_2(z)^{\frac12}} & \ge & C^{-1}\, \delta(z)^{\frac12-\frac1p} |\log \delta(z)|^{-\frac{(n+\gamma)(p-2)}{2p\gamma}},\ \ \   p<2.\label{eq:comp_2}
\end{eqnarray} 
\end{theorem}

\begin{proof}
The argument is essentially similar as \cite{ChenFu11} (see also \cite{ChenH-index}). Recall that there exists a smooth negative psh function $\rho$ on $\Omega$ such that $-\rho\asymp \delta^\gamma$ for some $\gamma>0$ (cf. \cite{DiederichFornaess77}). It then follows from a very useful estimate of Blocki \cite{BlockiGreen} for the pluricomplex Green function that there is a constant $C>1$ such that for any $z$ sufficiently close to $\partial \Omega$,
\begin{equation}\label{eq:Blocki}
\{g_\Omega(\cdot,z)\le -1\}\subset \left\{\zeta\in\Omega:C^{-1} \delta(z)|\log \delta(z)|^{-\frac1\gamma}\le \delta(\zeta) \le C\delta(z)|\log\delta(z)|^{\frac{n}\gamma}\right\},
\end{equation}
where $g_\Omega(\zeta,z)$ is the pluricomplex Green function defined by 
$$
  g_\Omega(\zeta,z)=\sup\left\{u(\zeta):u\in PSH^-(\Omega)\ \text{and}\  u(\zeta)\le \log |\zeta-z|+O(1)\ {\rm near\ }z\right\}.
    $$
      Note that $g_\Omega(\cdot,z)$ is a continuous negative psh function on $\Omega\backslash \{z\}$ which satisfies
     \begin{equation}\label{eq:DF}
     -i\partial\bar{\partial} \log (-g_\Omega(\cdot,z))\ge i\partial \log (-g_\Omega(\cdot,z))\wedge \bar{\partial} \log (-g_\Omega(\cdot,z))
     \end{equation}
     as currents. For $p>2$ and $z\in \Omega$ we take $f\in A^p(\Omega)$ with $\|f\|_p=1$ and $f(z)=K_p(z)^{\frac1p}$.  Let $\chi:\mathbb R\rightarrow [0,1]$ be a cut-off function such that $\chi|_{(-\infty,-\log 2]}=1$ and $\chi|_{[0,\infty)}=0$. Set 
     \begin{equation}\label{eq:v}
     v:=f\bar{\partial}\chi(-\log(-g_\Omega(\cdot,z))).
          \end{equation}
       By \eqref{eq:DF} we have
       $$
    i\bar{v}\wedge v \le |f|^2 |\chi'(-\log(-g_\Omega(\cdot,z)))|^2 \cdot \left[     -i\partial\bar{\partial} \log (-g_\Omega(\cdot,z))\right].
               $$   
     The Donnelly-Fefferman estimate (cf. \cite{DonnellyFefferman}, see also \cite{BerndtssonCharpentier}, \cite{BlockiDF}) then yields  a solution of the equation
  $
 \bar{\partial} u=v
 $ (in the sense of distributions)
 such that
 \begin{eqnarray*}
 \int_\Omega |u|^2 e^{-2n g_\Omega(\cdot,z)} & \le & C_0 \int_\Omega  |f|^2 |\chi'(-\log(-g_\Omega(\cdot,z)))|^2 e^{-2ng_\Omega(\cdot,z)}\\
 & \le & C_n \int_{\{\zeta\in\Omega:\delta(\zeta) \le C\delta(z)|\log\delta(z)|^{\frac{n}\gamma}\} } |f|^2\ \ \ ({\rm by\ }(\ref{eq:Blocki}))\\
 & \le & C_n |\{\zeta\in\Omega:\delta(\zeta) \le C\delta(z)|\log\delta(z)|^{\frac{n}\gamma}\}|^{1-\frac2p}\|f\|_p^2\\
 & \le & C \delta(z)^{1-\frac2p} |\log\delta(z)|^{\frac{n(p-2)}{p\gamma}}
 \end{eqnarray*}
 where the third inequality follows from H\"older's inequality. Here for the sake of simplicity we use the symbol $C$ to denote any large constant depending only on $\Omega$.
 Set
 $$
 F:=f\chi(-\log(-g_\Omega(\cdot,z)))-u.
 $$
 Clearly, we have $F\in {\mathcal O}(\Omega)$. Since $g_\Omega(\zeta,z)=\log |\zeta-z|+O(1)$ as $\zeta\rightarrow z$ and $u$ is holomorphic in a neighborhood of $z$, it follows that $u(z)=0$, i.e., $F(z)=f(z)=K_p(z)^{\frac1p}$. Moreover, we have
 \begin{eqnarray*}
 \int_\Omega |F|^2 & \le & 2\int_\Omega |f \chi(-\log(-g_\Omega(\cdot,z)))|^2+2\int_\Omega |u|^2\\
 & \le & C \delta(z)^{1-\frac2p} |\log\delta(z)|^{\frac{n(p-2)}{p\gamma}}
  \end{eqnarray*}
 since $g_\Omega(\cdot,z)<0$.
 Thus we get
 $$
 K_2(z)^{\frac12}\ge \frac{|F(z)|}{\|F\|_2}\ge C^{-1}\,K_p(z)^{\frac1p}  \delta(z)^{\frac1p-\frac12} |\log\delta(z)|^{-\frac{n(p-2)}{2p\gamma}},
 $$
i.e., (\ref{eq:comp_1}) holds.

 Next we define 
$$
A^2_{\alpha,\varepsilon}(\Omega):=\left\{f\in \mathcal O(\Omega): \|f\|_{\alpha,\varepsilon}^2:=\int_\Omega |f|^2 (-\rho+\varepsilon)^\alpha<\infty\right\},\ \ \ \alpha,\varepsilon>0.
$$
Let $K_{\alpha,\varepsilon}$ denote the Bergman kernel associated to the Hilbert space $A^2_{\alpha,\varepsilon}(\Omega)$. We first compare $K_p(z)$ with $K_{\alpha,\varepsilon}(z)$ for $\alpha=\frac1\gamma(\frac2p-1)$ and $\varepsilon=\delta(z)^\gamma$ when $p<2$. Take $f\in A^2_{\alpha,\varepsilon}(\Omega)$ such that $\|f\|_{\alpha,\varepsilon}=1$ and $f(z)=K_{\alpha,\varepsilon}(z)^{\frac12}$. By H\"older's inequality we have
\begin{eqnarray*}
\int_\Omega |f|^p & = & \int_\Omega |f|^p (-\rho+\varepsilon)^{\frac1\gamma(1-\frac{p}2)} (-\rho+\varepsilon)^{-\frac1\gamma(1-\frac{p}2)}\\
& \le & \left[\int_\Omega |f|^2 (-\rho+\varepsilon)^\alpha\right]^{\frac{p}2}\left[\int_\Omega (-\rho+\varepsilon)^{-\frac1\gamma}\right]^{1-\frac{p}2}\\
& \le & C \left[\int_{\zeta\in\Omega} (\delta(\zeta)+\delta(z))^{-1}\right]^{1-\frac{p}2}\\
&\le & C |\log\delta(z)|^{1-\frac{p}2}.
\end{eqnarray*}
It follows that 
\begin{equation}\label{eq:comp_3}
K_p(z)^{\frac1p} \ge \frac{|f(z)|}{\|f\|_p}\ge C^{-1} K_{\alpha,\varepsilon}(z)^{\frac12} |\log\delta(z)|^{\frac12-\frac1p}.
\end{equation}
Now we compare $K_{\alpha,\varepsilon}(z)$ with $K_2(z)$. Set $\psi:=-\alpha\log(-\rho+\varepsilon)$. Clearly, $\psi$ is psh on $\Omega$. Similar as above, we first take $f\in A^2(\Omega)$ with $\|f\|_2=1$ and $f(z)=K_2(z)^{\frac12}$ then  solve the equation $\bar{\partial}u=v$ (where $v$ is given by \eqref{eq:v}) with the following estimate
 \begin{eqnarray*}
 \int_\Omega |u|^2 e^{-\psi-2n g_\Omega(\cdot,z)} & \le & C_0 \int_\Omega  |f|^2 |\chi'(-\log(-g_\Omega(\cdot,z)))|^2  e^{-\psi-2ng_\Omega(\cdot,z)}\\
 & \le & C_n \int_{S_z} |f|^2(-\rho+\varepsilon)^\alpha\\
 & \le & \sup_{S_z}\left\{(-\rho+\varepsilon)^\alpha\right\} \cdot \|f\|_2^2\\
 & \le & C \delta(z)^{\frac2p-1} |\log\delta(z)|^{\frac{n(2-p)}{p\gamma}},
 \end{eqnarray*}
where $S_z:=\{\zeta\in\Omega:\delta(\zeta) \le C\delta(z)|\log\delta(z)|^{\frac{n}\gamma}\}$. Thus $ F:=f\chi(-\log(-g_\Omega(\cdot,z)))-u$ is a holomorphic function on $\Omega$ which satisfies $F(z)=f(z)=K_2(z)^{\frac12}$ and
 $$
 \|F\|_{\alpha,\varepsilon}\le C \delta(z)^{\frac1p-\frac12} |\log\delta(z)|^{\frac{n(2-p)}{2p\gamma}}. 
 $$
 Thus
 $$
 K_{\alpha,\varepsilon}(z)^{\frac12}\ge \frac{|F(z)|}{\|F\|_{\alpha,\varepsilon}}\ge C^{-1} K_2(z)^{\frac12}  \delta(z)^{\frac12-\frac1p} |\log\delta(z)|^{-\frac{n(2-p)}{2p\gamma}}.  
 $$
 This together with (\ref{eq:comp_3}) yield (\ref{eq:comp_2}).
    \end{proof}

    \begin{theorem}\label{th:Comp_2}
Let $\Omega$ be a bounded pseudoconvex domain with $C^2-$boundary. For every $2\le p<2+\frac2n$  there exists a  constant $C=C_{p,\Omega}>0$ such that the following estimate holds near $\partial \Omega:$
\begin{equation}\label{eq:comp_6}
\frac{K_p(z)^{\frac1p}}{K_2(z)^{\frac12}} \ge C^{-1}\, \delta(z)^{\frac{(n+1)(p-2)}{2p}} |\log \delta(z)|^{-\frac{(n+1)(p-2)}{2p\gamma}}.
\end{equation} 
Here $\gamma$ is the same as above.
\end{theorem}     \begin{proof}
     For $0\le \alpha<1$ we define
     $$
     A^2_{-\alpha}(\Omega):=\left\{f\in \mathcal O(\Omega): \|f\|_{-\alpha}^2:=\int_\Omega|f|^2 \delta^{-\alpha}<\infty\right\}.
     $$
     Let $K_{-\alpha}$ denote the Bergman kernel associated to $A^2_{-\alpha}(\Omega)$. We first compare $K_p(z)$ with $K_{-\alpha}(z)$. Let $f\in A^2_{-\alpha}(\Omega)$. Since $\log |f|^2-\alpha\log \delta$ is psh on $\Omega$, so is $|f|^2\delta^{-\alpha}$. Apply the mean-value inequality to the psh function $|f|^2\delta^{-\alpha}$ on certain polydisc with center $z$ and volume $\asymp \delta(z)^{n+1}$, we have 
           \begin{equation}\label{eq:comp_4}
        |f(z)|^2\delta(z)^{-\alpha}\le C_n  \delta(z)^{-n-1} \|f\|_{-\alpha}^2. 
        \end{equation}
       It follows that 
      \begin{eqnarray}\label{eq:comp_4'}
       I_p(\varepsilon,f) & := &  \int_{\{\delta=\varepsilon\}} |f|^p dS\le \sup_{\{\delta=\varepsilon\}} |f|^{p-2}\cdot I_2(\varepsilon,f)
       \nonumber \\
        & \le & C_n \varepsilon^{-\frac{(n+1-\alpha)(p-2)}2}\|f\|_{-\alpha}^{p-2}\cdot I_2(\varepsilon,f).
             \end{eqnarray}
             Here $dS$ denotes the surface element. 
             Note that 
             $$
             \alpha=\frac{(n+1-\alpha)(p-2)}2 \iff \alpha=\frac{(n+1)(p-2)}p
             $$
             and 
             $$
             \alpha<1 \iff p<2+\frac2n.
             $$
             We fix such $\alpha$ and take $f\in A^2_{-\alpha}(\Omega)$ with $f(z)=K_{-\alpha}(z)^{\frac12}$ and $\|f\|_{-\alpha}=1$.    
             For certain $\varepsilon_0\ll1$, 
      \begin{eqnarray*}
      \int_\Omega |f|^p & = & \int_{\{\delta\ge \varepsilon_0\}} |f|^p +\int_0^{\varepsilon_0} I_p(\varepsilon,f)d\varepsilon\\
      & \le & C+ C_n\int_0^{\varepsilon_0}  \varepsilon^{-\alpha}I_2(\varepsilon,f) d\varepsilon \ \ \ \  \ (\text{by}\ \eqref{eq:comp_4'})\\
      & = & C+C_n \|f\|_{-\alpha}^2=C+C_n
            \end{eqnarray*}
          where $C=C_{p,\Omega}$. Thus
          \begin{equation}\label{eq:comp_5}
          K_p(z)^{\frac1p}\ge \frac{|f(z)|}{\|f\|_p}\ge C^{-1} K_{-\alpha}(z)^{\frac12}.
                    \end{equation}
           Next we compare $K_{-\alpha}(z)$  with $K_2(z)$.   Put
$$
\varphi=2ng_\Omega(\cdot,z)-\log(-g_\Omega(\cdot,z)+1).
$$
Take $f\in A^2(\Omega)$ with $f(z)=K_2(z)^{\frac12}$ and $\|f\|_2=1$. If $v$ is given as above, then 
$$
i\bar{v}\wedge v \le |f|^2 |\chi'(-\log(-g_\Omega(\cdot,z)))|^2\, \frac{(-g_\Omega(\cdot,z)+1)^2}{g_\Omega(\cdot,z)^2}\cdot {i\partial \bar\partial{\varphi}}.
$$
 By Theorem 1.1 in \cite{Chen14} we may solve $\bar{\partial}u=v$ with the following estimate
\begin{eqnarray*}
\int_\Omega |u|^2e^{-\varphi}\delta^{-\alpha} & \le & C_{\alpha,\Omega} \int_\Omega |f|^2 |\chi'(-\log(-g_\Omega(\cdot,z)))|^2 \, \frac{(-g_\Omega(\cdot,z)+1)^2}{g_\Omega(\cdot,z)^2}\, e^{-\varphi}\delta^{-\alpha}\\
& \le & C_{\alpha,\Omega} \int_{\{\zeta\in\Omega:\delta(\zeta) \ge C^{-1}\delta(z)|\log\delta(z)|^{-\frac{1}\gamma}\}} |f|^2\delta^{-\alpha}\ \ \ ({\rm by\ }\eqref{eq:Blocki})\\
& \le & C_{\alpha,\Omega}\, \delta(z)^{-\alpha} |\log\delta(z)|^{\frac{\alpha}\gamma}.
\end{eqnarray*}
It follows that 
$$
F:=\chi(-\log(-g_\Omega(\cdot,z)))f-u
$$
is a holomorphic function on $\Omega$ satisfying $F(z)=f(z)=K_2(z)^{\frac12}$ and 
$$
\int_\Omega |F|^2 \delta^{-\alpha}\le C_{\alpha,\Omega}\, \delta(z)^{-\alpha} |\log\delta(z)|^{\frac{\alpha}\gamma}.
$$
Thus 
$$
K_{-\alpha}(z)\ge \frac{|F(z)|^2}{\|F\|_{-\alpha}^2} \ge C_{\alpha,\Omega}^{-1}\,K_2(z) \delta(z)^{\alpha} |\log\delta(z)|^{-\frac{\alpha}\gamma}.
$$
This together with (\ref{eq:comp_5}) yield (\ref{eq:comp_6}).
        \end{proof}
        
\section{$A^p(\Omega)$ as a linear space}
\subsection{The $p-$Bergman projection}
  Let $(X,d)$ be a metric space,   $Y$ a set in $X$ and $x\in X$.  Following \cite{Singer},  we say that $y_0\in Y$ is an element of\/ {\it best approximation} of $x$ if 
$$
d(x,y_0)=\inf_{y\in Y} d(x,y).
$$
Let $\mathcal P_Y(x)$ be the set of all such elements,  which defines  a map $\mathcal P_Y:X\rightarrow Y$ called the\/ {\it metric projection} from $X$ onto $Y$.     $Y$ is called\/ {\it proximinal} (respectively\/ {\it Chebyshev}) if $\mathcal P_Y(x)\neq \emptyset$ (respectively $\mathcal P_Y(x)$ is a singleton) for each $x\in X$.  

\begin{proposition}\label{prop:p-projection}
\begin{enumerate}
\item[$(1)$] $A^p(\Omega)$ is proximinal in $L^p(\Omega)$ for $0<p\le \infty$.
\item[$(2)$] If\/ $h_0\in \mathcal P_{A^p(\Omega)}(f)$ with $1<p<\infty$,  then  
\begin{equation}\label{eq:p-projection}
\int_\Omega |f-h_0|^{p-2} \overline{f-h_0}\cdot h=0,\ \ \ \forall\,h\in A^p(\Omega).
\end{equation} 
Here for a function $g\in L^p(\Omega)$ one defines $|g|^{p-2}\bar{g}=0$ on $g^{-1}(0)$.  
\item[$(3)$]  
Let $1<p<\infty$.  For any $f\in L^p(\Omega)$, 
 $h_0\in \mathcal P_{A^p(\Omega)}(f)$ and $h\in A^p(\Omega)$,  we have 
 \begin{equation}\label{eq:p-projection_2}
 \frac1{4^{p+3}}\|h-h_0\|_p^p \le \|f-h\|_p^p-\|f-h_0\|_p^p,\ \ \  2\le p<\infty,
 \end{equation}
  \begin{eqnarray}\label{eq:p-projection_3}
 C_p^{-1}\|h-h_0\|_p^p & \le & \left(\|f-h\|_p^p-\|f-h_0\|_p^p\right)^{\frac{p}2} \\
 && \times \left(\|f-h\|_p^p+\|f-h_0\|_p^p\right)^{1-\frac{p}2},\ \ \ 1<p\le 2\nonumber
 \end{eqnarray}
 where $C_p$ depends only on $p$.
In particular,  $A^p(\Omega)$ is a Chebyshev set  of $L^p(\Omega)$ for $1<p<\infty$.
\item[$(4)$] If\/ $h_0\in A^p(\Omega)$ satisfies \eqref{eq:p-projection} for some $1<p<\infty$,  then $h_0\in \mathcal P_{A^p(\Omega)}(f)$.  
\end{enumerate}
\end{proposition}

\begin{proof}
$(1)$ Take a sequence $\{h_j\}\subset A^p(\Omega)$ such that $\|f-h_j\|_p\rightarrow d:=\inf_{h\in A^p(\Omega)}\|f-h\|_p$.  Since 
$$
\|h_j\|_p\le \|f\|_p+\|f-h_j\|_p\le C<\infty,  
$$
it follows from the Bergman inequality that there is a subsequence $\{h_{j_k}\}$ converges locally uniformly to some $h_0\in \mathcal O(\Omega)$,  which satisfies $\|h_0\|_p\le C$.  Thus
$$
d\le \|f-h_0\|_p\le \liminf_{k\rightarrow \infty} \|f-h_{j_k}\|_p=d,
$$
so that $h_0\in \mathcal P_{A^p(\Omega)}(f)$.

$(2)$  For every $h\in A^p(\Omega)$ the function $J(t):=\|f-h_0+th\|_p^p$ attains the  minimum at $t=0$.  Thus
$$
0=\frac{\partial J}{\partial t}(0)=\frac{p}2 \int_\Omega |f-h_0|^{p-2} \overline{f-h_0}\cdot h.
$$

$(3)$ We first consider the case $2\le p<\infty$.  
Substitute  $a=f-h_0$,   $b=f-h$ into \eqref{eq:eleIneq_3}  and integration over $\Omega$,  we obtain
\begin{eqnarray*}
\frac1{4^{p+3}}\int_\Omega |h-h_0|^p & \le & \|f-h\|_p^p-\|f-h_0\|_p^p\\
&& -p\mathrm{Re}\int_\Omega |f-h_0|^{p-2}\overline{f-h_0}(h_0-h)\\
& = & \|f-h\|_p^p-\|f-h_0\|_p^p
\end{eqnarray*} 
in view of \eqref{eq:p-projection}.  

Next,  suppose  $1<p\le 2$.  
 Similar as above,  we use  \eqref{eq:eleIneq_4} instead of \eqref{eq:eleIneq_3} to get 
$$
 A_p \int_\Omega |h-h_0|^2(|f-h_0|+| f-h |)^{p-2}\\
 \le  \|f-h\|_p^p-\|f-h_0\|_p^p.
$$
On the other hand,  H\"older's inequality gives
\begin{eqnarray*}
 \int_\Omega |h-h_0|^p
& \le & \left\{\int_\Omega |h-h_0|^2(|f-h_0|+| f-h |)^{p-2} \right\}^{\frac{p}2} \left\{\int_\Omega ( |f-h_0|+|f-h| )^{p} \right\}^{1-\frac{p}2}\\
& \le & C_p \left(\|f-h\|_p^p-\|f-h_0\|_p^p\right)^{\frac{p}2} \left(\|f-h\|_p^p+\|f-h_0\|_p^p\right)^{1-\frac{p}2}.
\end{eqnarray*}

$(4)$
Take $h_0'\in \mathcal P_{A^p(\Omega)}(f)$ so that
$$
\int_\Omega |f-h_0'|^{p-2} \overline{f-h_0'}\cdot h=0,\ \ \ \forall\,h\in A^p(\Omega).
$$
Since $h_0-h_0'\in A^p(\Omega)$,  we have
\begin{eqnarray*}
\|f-h_0\|_p^p & = & \int_\Omega |f-h_0|^{p-2}\overline{f-h_0}\cdot (f-h_0)\\
&  = & \int_\Omega |f-h_0|^{p-2}\overline{f-h_0}\cdot (f-h_0')\\
& \le & \|f-h_0\|_p^{p-1} \|f-h_0'\|_p,
\end{eqnarray*} 
so that $\|f-h_0\|_p\le \|f-h_0'\|_p$,  which in turn implies $h_0=h_0'$ by uniqueness of the best approximation.
\end{proof}

\begin{remark}
For $p=2$,  the Pythagorean theorem yields
$$
\|h-h_0\|_2^2=\|f-h\|_2^2-\|f-h_0\|_2^2,
$$
which implies that \eqref{eq:p-projection_2} and \eqref{eq:p-projection_3} are optimal.
\end{remark}

  For $f\in L^p(\Omega)\ (1<p<\infty)$  let $P_p(f)\in A^p(\Omega)$ be the unique element of  best approximation  of $f$.   It is reasonable to call the  metric projection
$P_p:f\mapsto P_p(f)$ the\/ {\it $p-$Bergman projection} on $\Omega$.  Clearly,  the restriction of $P_p$ to $A^p(\Omega)$ is the identity map.  Note that $P_2$ is the standard Bergman projection.  

\begin{proposition}
\begin{enumerate}
\item[$(1)$] If $f\in L^p(\Omega)$ is a function such that $\bar{\partial}f$ is well-defined,  then $Q_p(f):=f-P_p(f)$ is the $L^p-$minimal solution of $\bar{\partial}u=\bar{\partial} f.$
\item[$(2)$]  For any $f\in L^p(\Omega)$ and $F\in \mathrm{Aut}(\Omega)$,  we have 
$$
P_p(f)\circ F\cdot J_F^{2/p}=P_p\left(f\circ F\cdot J_F^{2/p}\right)
$$
whenever $\Omega$ is simply-connected; in particular,  if $f\in P_p^{-1}(0)$,  then $f\circ F\cdot J_F^{2/p}\in P_p^{-1}(0)$.
\item[$(3)$] There exists a constant $C_p>0$ such that for all $f,g\in L^p(\Omega)$ 
$$
\|P_p(f)-P_p(g)\|_p \le
\left\{
\begin{array}{ll}
 C_p (\|f\|_p+\|f-g\|_p)^{1-1/p}\|f-g\|_p^{1/p} & \text{for\ } 2< p<\infty;\\
 C_p (\|f\|_p+\|f-g\|_p)^{1/2}\|f-g\|_p^{1/2}  & \text{for\ } 1 < p \le 2.
\end{array}
\right.
$$
\item[$(4)$] For any $f\in L^\infty(\Omega)$,  there exists a sequence $p_j\uparrow \infty$ such that $P_{p_j}(f)$ converges locally uniformly to some $h_\infty\in A^\infty(\Omega)$ which is an element of best approximation of $f$ in $L^\infty(\Omega)$.
\item[$(5)$] Let $d_p(f)$ denote the distance of  $f\in L^p(\Omega)$ to $A^p(\Omega)$ for $0<p\le \infty$.  Then we have 
$$
d_\infty(f)=\lim_{p\rightarrow \infty} d_p(f),\ \ \ \forall\,f\in L^\infty(\Omega).
$$
\end{enumerate}

\end{proposition}

\begin{proof}

 $(1)$  If $u'$ is another $L^p-$solution,  then $h:=Q_p(f)-u'\in A^p(\Omega)$ and we have 
$$
\|u'\|_p=\|Q_p(f)-h\|_p=\|f-(P_p(f)+h)\|_p\ge \|f-P_p(f)\|_p = \|Q_p(f)\|_p.
$$

$(2)$ For any $h\in A^p(\Omega)$,  we have
$$
\int_\Omega |f-h|^p\ge \int_\Omega |f-P_p(f)|^p,
$$
so that
$$
\int_\Omega |f\circ F-h\circ F|^p |J_F|^2 \ge \int_\Omega |f\circ F-P_p(f)\circ F|^p|J_F|^2,
$$
that is
$$
\int_\Omega \left|f\circ F\cdot J_F^{2/p}-h\circ F \cdot J_F^{2/p} \right|^p 
 \ge \int_\Omega \left|f\circ F \cdot J_F^{2/p} -P_p(f)\circ F \cdot J_F^{2/p} \right|^p.
$$
Note that 
$f\in L^p(\Omega)$ (respectively $h\in A^p(\Omega)$) if and only if $f\circ F\cdot J_F^{2/p}\in L^p(\Omega)$ (respectively $h\circ F \cdot J_F^{2/p}\in A^p(\Omega)$),   thus
$$
P_p(f)\circ F\cdot J_F^{2/p}=P_p\left(f\circ F\cdot J_F^{2/p}\right).
$$

$(3)$ Since $\|Q_p(f)\|_p$ is the distance from $f$ to $A^p(\Omega)$,  it is easy to see that
$$
\left| \|Q_p(f)\|_p-\|Q_p(g)\|_p\right|\le \|f-g\|_p,\ \ \ \forall\,f,g\in L^p(\Omega).
$$ 
We first consider the case $2\le p<\infty$.  
By \eqref{eq:p-projection_2},  we have
\begin{eqnarray*}
\frac1{4^{p+3}}\int_\Omega |P_p(f)-P_p(g)|^p & \le & \|f-P_p(g)\|_p^p-\|f-P_p(f)\|_p^p\\
& \le & p\left(\|f-P_p(g)\|_p+\|f-P_p(f)\|_p\right)^{p-1}\\
&& \times \left( \|f-P_p(g)\|_p- \|f-P_p(f)\|_p\right)\\
& \le & p\left(\|f-g\|_p+\|Q_p(g)\|_p+\|Q_p(f)\|_p\right)^{p-1}\\
&& \times \left( \|f-g\|_p+\|Q_p(g)\|_p- \|Q_p(f)\|_p\right)\\
& \le & C_p' (\|f\|_p+\|f-g\|_p)^{p-1}\|f-g\|_p.
\end{eqnarray*}
The case $1<p\le 2$ can be verified analogously by using \eqref{eq:p-projection_3} instead of \eqref{eq:p-projection_2}.

$(4)$ Set $h_p:=P_p(f)$.  Since
$$
\|h_p\|_p=\|f-Q_p(f)\|_p\le \|f\|_p+\|Q_p(f)\|_p\le 2\|f\|_p\le 2\|f\|_\infty |\Omega|^{\frac1p},
$$
it follows from the Bergman inequality that for any $z\in \Omega$
$$
|h_p(z)|\le \left\{C_n\delta(z)^{-2n}\right\}^{\frac1p}\|h_p\|_p\le 2 \left\{C_n|\Omega|\delta(z)^{-2n}\right\}^{\frac1p}\|f\|_\infty,
$$
so that $\{h_p\}$ forms a normal family and there exists a sequence $h_{p_j}$ converges locally uniformly to certain $h_\infty\in \mathcal O(\Omega)$.  Moreover,  
$$
|h_\infty(z)|=\lim_{j\rightarrow \infty} |h_{p_j}(z)|\le 2\|f\|_\infty,\ \ \ \forall\,z\in \Omega.
$$
On the other hand,  we have
\begin{equation}\label{eq:proj-infin_1}
\|f-h_{p_j}\|_{p_j}\le \|f-h\|_{p_j},\ \ \ \forall\,h\in A^\infty(\Omega).
\end{equation}
For any open set $U\subset\subset \Omega$,  we have
\begin{equation}\label{eq:proj-infin_2}
\|f-h_{p_j}\|_{L^{p_j}(\Omega)} \ge \|f-h_{p_j}\|_{L^{p_j}(U)}\ge \|f-h_{\infty}\|_{L^{p_j}(U)}-\|h_{p_j}-h_\infty\|_{L^{p_j}(U)}.
\end{equation}
Note that
\begin{equation}\label{eq:proj-infin_3}
\|h_{p_j}-h_\infty\|_{L^{p_j}(U)} \le \|h_{p_j}-h_\infty\|_{L^{\infty}(U)} |U|^{\frac1{p_j}}\rightarrow 0\ \ \ (j\rightarrow \infty).
\end{equation}
Use $\eqref{eq:proj-infin_1}\sim \eqref{eq:proj-infin_3}$ with $p_j\rightarrow \infty$,  we obtain
$$
\|f-h_\infty\|_{L^\infty(U)}\le \|f-h\|_{L^\infty(\Omega)},\ \ \ \forall\,h\in A^\infty(\Omega).
$$ 
Since $U$ can be arbitrarily  chosen,  we have
$$
\|f-h_\infty\|_\infty\le \|f-h\|_\infty,\ \ \ \forall\,h\in A^\infty(\Omega),
$$
i.e.,  $h_\infty$ is an element of best approximation of $f$ in $L^\infty(\Omega)$.

\item[$(5)$] We first show that $c=\lim_{p\rightarrow \infty} d_p(f)$ exists.  Note that for $p'>p>1$
$$
d_p(f)=\|Q_p(f)\|_p\le \|Q_{p'}(f)\|_p\le \|Q_{p'}(f)\|_{p'}|\Omega|^{\frac1p-\frac1{p'}}=d_{p'}(f)|\Omega|^{\frac1p-\frac1{p'}}.
$$
This means that $|\Omega|^{-\frac1p}d_p(f)$ is nondecreasing in $p$,  so that  $c=\lim_{p\rightarrow \infty} d_p(f)$ exists. 

Next,  since 
$$
d_p(f) \le \|f-h\|_p,\ \ \ \forall\,h\in A^\infty(\Omega),  
$$
so by letting $p\rightarrow \infty$ we obtain $c\le \|f-h\|_\infty$ for all $h\in A^\infty(\Omega)$,  i.e.,  $c\le d_\infty(f)$. On the other hand,  it follows from \eqref{eq:proj-infin_2} and \eqref{eq:proj-infin_3} that
$$
c \ge \|f-h_\infty\|_{L^\infty(U)}
$$ 
for any  open set $U\subset\subset \Omega$.
Thus $c\ge \|f-h_\infty\|_{L^\infty(\Omega)}=d_\infty(f)$.
\end{proof}

\begin{problem}
Is it possible to conclude that $h_p\rightarrow h_\infty$  locally uniformly  as $p\rightarrow \infty$?
\end{problem}

The following elementary properties of $P_p$  are known even in more general cases (compare \cite{Singer}).  

\begin{proposition}\label{prop:proj-properties}
\begin{enumerate}
\item[$(1)$]  $L^p(\Omega)$ admits a decomposition as follows
$$
L^p(\Omega)=A^p(\Omega)\oplus P_p^{-1}(0),\ \ \ f=P_p(f)+Q_p(f).
$$
\item[$(2)$] $P_p^{-1}(0)$ is  $\mathbb C-$starlike in the following sense
$$
f\in P_p^{-1}(0)\Rightarrow cf\in P_p^{-1}(0),\ \ \ \forall\,c\in \mathbb C.
$$
\item[$(3)$] $P_p$ is a linear operator if and only if $P_p^{-1}(0)$ is a linear subspace of $L^p(\Omega).$
\end{enumerate}
\end{proposition}

\begin{proof}
$(1)$ Since 
$$
\|f-(P_p(f)+h)\|_p\ge \|f-P_p(f)\|_p,\  \ \ \forall\,h\in A^p(\Omega),
$$
it follows that $Q_p(f)\in P_p^{-1}(0)$.  It is  also easy to verify that $A^p(\Omega)\cap P_p^{-1}(0)=\{0\}$.     Thus $L^p(\Omega)=A^p(\Omega)\oplus P_p^{-1}(0)$.   

$(2)$ By definition of $P_p$,  we have $P_p(cf)=cP_p(f)$,  $\forall\,c\in \mathbb C,\,f\in L^p(\Omega)$.   Thus 
$$
f\in P_p^{-1}(0)\Rightarrow cf\in P_p^{-1}(0),\ \ \ \forall\,c\in \mathbb C.
$$

$(3)$
 For any $f_1,f_2\in L^p(\Omega)$ and $c_1,c_2\in \mathbb C$,    we have
$$
P_p(c_1f_1+c_2f_2)+Q_p( c_1f_1+c_2f_2 ) = c_1f_1+c_2f_2 = c_1P_p(f_1)+c_2P_p(f_2)+ c_1 Q_p(f_1)+c_2Q_p(f_2).
$$
Thus
 $$
 P_p(c_1f_1+c_2f_2)=c_1P_p(f_1)+c_2P_p(f_2) \Rightarrow Q_p(c_1f_1+c_2f_2)=c_1 Q_p(f_1)+c_2Q_p(f_2);
 $$
 in particular,  if $P_p$ is linear,  then $P_p^{-1}(0)$ is a linear subspace of $L^p(\Omega)$.  On the other hand,  if $P^{-1}(0)$ is linear,  then $c_1 Q_p(f_1)+c_2Q_p(f_2)\in P_p^{-1}(0)$,  so that
 $$
 \| c_1 Q_p(f_1)+c_2Q_p(f_2)\|_p\le   \| c_1 Q_p(f_1)+c_2Q_p(f_2)-h\|_p,\ \ \ \forall\,h\in A^p(\Omega), 
 $$
 that is
 $$
 \| (c_1f_1+c_2f_2)- (c_1P_p(f_1)+c_2P_p(f_2))\|_p \le \| (c_1f_1+c_2f_2)- (c_1P_p(f_1)+c_2P_p(f_2))-h \|_p
 $$
 for all $h\in A^p(\Omega)$,  i.e.,  $P_p(c_1f_1+c_2f_2)=c_1P_p(f_1)+c_2P_p(f_2)$.

\end{proof}

\begin{proposition}
If\/ $\Omega$ is the unit disc $\Delta$,  then we have
\begin{enumerate}
\item[$(1)$] $P_p$ is nonlinear when $p\neq 2;$
\item[$(2)$] $P_p^{-1}(0)$ contains an infinite-dimensional linear subspace. 
\end{enumerate}
\end{proposition}

\begin{proof}
$(1)$ We first claim that $P_p({\bar{z}}^k)=0$ for each $k\in \mathbb Z^+$.  To see this,  simply use polar coordinates to verify that
$$
\int_\Delta |{\bar{z}}^k|^{p-2} z^k\cdot h=0, \ \ \ \forall\,h\in A^p(\Delta),  
$$
so that Proposition \ref{prop:p-projection}/(2) applies.

Set $f_t(z)=\bar{z}+t{\bar{z}}^2$,  $t\in \mathbb C$.  For any $p\neq 2$,  we have
$$
|f_t(z)|^{p-2}=|z|^{p-2}\left(|1+t\bar{z}|\right)^{\frac{p-2}2}=|z|^{p-2}\left(1+(p-2)\mathrm{Re}(t\bar{z})+O(|t|^2)\right)
$$
as $t\rightarrow 0$.  This gives
\begin{eqnarray*}
\int_\Delta |f_t|^{p-2}\bar{f}_t & = & \int_\Delta |z|^{p-2}\left(1+(p-2)\mathrm{Re}(t\bar{z})+O(|t|^2)\right) \left(z+\bar{t} z^2\right)\\
& = & (p-2) \int_\Delta |z|^{p-2}z \mathrm{Re}(t\bar{z})+O(|t|^2)\\
& = & \frac{p-2}2 \int_\Delta |z|^{p-2}z \left(t\bar{z}+\bar{t}z\right)+O(|t|^2)\\
& = &  \frac{p-2}2 \int_\Delta |z|^{p}\cdot t +O(|t|^2),
 \end{eqnarray*}
 so that $\int_\Delta |f_t|^{p-2}\bar{f}_t\neq 0$ when $|t|\ll 1$.   By Proposition \ref{prop:p-projection}/(2),  we  conclude that 
 $$
 P_p(\bar{z}+t{\bar{z}}^2)\neq 0=P_p(\bar{z})+tP_p(\bar{z}^2).
 $$
 
 $(2)$ For any measurable function $\lambda$ on the interval  $[0,1]$,  we have 
 $$
 \lambda(|z|)\in L^p(\Delta) \iff \int_0^1 r |\lambda(r)|^p dr<\infty
 $$
and 
$$
\int_\Delta |\lambda(|z|)|^{p-2}\overline{\lambda(|z|)}\cdot z^k=0,\ \ \ \forall\,k\in \mathbb Z^+,
$$
by using polar coordinates.  This together with
 Proposition \ref{prop:p-projection}/(2) yield that $\lambda(|z|)\in P_p^{-1}(0)$ if and only if 
 $$
 \int_\Delta |\lambda(|z|)|^{p-2}\overline{\lambda(|z|)}=2\pi \int_0^1 r |\lambda(r)|^{p-2}{\lambda(r)}dr=0.
 $$
 Fix any positive continuous function $\eta$ on $[0,1]$ with $\int_0^1 r \eta(r)^{p-1} dr=\infty$, i.e.  $\int_\Delta \eta(|z|)^{p-1}=\infty$.  We may take a sequence of mutually disjoint annuli $\{\mathcal A_j\}$ in $\Delta$ such that 
 $$
 \int_{\mathcal A_{2j-1}} \eta(|z|)^{p-1}=\int_{\mathcal A_{2j}}\eta(|z|)^{p-1},\ \ \ \forall\,j\in \mathbb Z^+.
 $$
 Set $f_j=\left(\chi_{\mathcal A_{2j-1}}-\chi_{\mathcal A_{2j}}\right)\eta(|z|)$,  where $\chi_\cdot$ stands for the characteristic function.   Since
 $$
 \int_\Delta |f_j|^{p-2}\overline{f_j}= \int_{\mathcal A_{2j-1}} \eta(|z|)^{p-1}-\int_{\mathcal A_{2j}}\eta(|z|)^{p-1}=0,
 $$
 so $f_j\in P_p^{-1}(0)$.  Thus for any $f=\sum_{j=1}^N c_j f_j$,  we have
 $$
 \int_\Delta |f|^{p-2}\overline{f}=\sum_{j=1}^N |c_j|^{p-2}\bar{c}_j\int_\Delta |f_j|^{p-2}\overline{f_j}=0,
 $$
 so that $f\in P_p^{-1}(0)$,  i.e.,
 $
 \mathrm{span}\{f_1,f_2,\cdots\}\subset P_p^{-1}(0).
 $
 Clearly,   these functions $f_1,f_2,\cdots$ are linear independent.  
\end{proof}

\begin{problem}
When is $A^p(\Omega)$ complemented in $L^p(\Omega)$, i.e., $L^p(\Omega)=A^p(\Omega)\oplus \mathcal N$ where $\mathcal N$ is a closed linear subspace?
\end{problem}

Note that for every strongly pseudoconvex domain $\Omega\subset \mathbb C^n$ the (linear) Bergman projection $P_2$ maps $L^p(\Omega)$ continuously onto $A^p(\Omega)$ for $1<p<\infty$ (cf.  \cite{Lanzani}; see also  \cite{Zhu} for the special case of the unit ball);  thus $P_2^{-1}(0)\cap L^p(\Omega)$ is a closed linear subspace of $L^p(\Omega)$,  i.e.,    $A^p(\Omega)$ is complemented in $L^p(\Omega)$.  

\subsection{Bases in $A^p(\Omega)$}
A sequence $\{x_j\}$ is said to be a (Schauder)\/ {\it basis} for a normed linear space $X$ if,  for each $x \in X$,  there is a\/ {\it unique} sequence of scalars $\{c_j\}$ such that $x =\sum_{j=1}^\infty  c_j x_j$,  where the series converges in norm to $x$.  We say that $\{x_j\}$ a\/ {\it basic sequence} if it is a basis for its closed linear span $\overline{\mathrm{span}(\{x_j\})}$.  A basic question is

\begin{problem}
Does $A^p(\Omega)$ admit a  basis for all $1\le p<\infty$?
\end{problem}
 
 Set
$$
 \mathcal H^\alpha_{p,z}:=\left\{ f\in A^p(\Omega): \partial^{(\alpha)}f(z)=1\ \text{and}\ \partial^{(\beta)}f(z)=0,\forall\ \beta\prec \alpha  \right\}.
 $$ 
 Given two multi-indices   $\alpha=(\alpha_1,\cdots,\alpha_n)$ and $\beta=(\beta_1,\cdots,\beta_n)$,    define $\beta\prec \alpha$ $\iff$ $|\beta|<|\alpha|$ or $|\beta|=|\alpha|$ and $\beta_j=\alpha_j$ for $j<k$ while $\beta_k>\alpha_k$.

 \begin{proposition}\label{prop:basic}
There exist a subsequence $\alpha^{j_1}\prec \alpha^{j_2}\prec\cdots$ and functions $f_{j_k}\in \mathcal H^{\alpha^{j_k}}_{p,z}$,   $k=1,2,\cdots$,  such that $\{f_{j_k}\}$ forms a basic sequence.  
  \end{proposition}
 
The proof of this proposition is a mimic of Mazur's construction of basic sequences in an infinite-dimensional Banach space  (cf.  \cite{Carothers},  p.  34--35).  We first show the following
 
 \begin{lemma}\label{lm:basic}
 Let $Y$ be a finite-dimensional subspace of $A^p(\Omega)$,  $1\le p<\infty$.  Then,  given $\varepsilon>0$ and multi-index $\alpha$,  there exists a nonzero $f\in A^p(\Omega)$ with $\partial^{(\beta)}f(z)=0$ for all $\beta\prec \alpha$,  such that
 $$
 \|g\|_p\le (1+\varepsilon)\|g+cf\|_p,\ \ \ \forall\,g\in Y,\,c\in \mathbb C.
 $$  
 \end{lemma}
 
 \begin{proof}
Since $Y$ is finite-dimensional,    so the unit sphere $S_Y$ of $Y$  is compact.    Thus we can take a finite $\varepsilon/2-$net $g_1,\cdots,g_m$ for $S_Y$.  Thanks to the Hahn-Banach theorem,  there exists for every $1\le k\le m$ a continuous linear functional $F_k^\ast$ on $A^p(\Omega)$ such that $\|F^\ast_k\|=1$ and $F^\ast_k(g_k)=1$.   Note that the linear subspace 
 $$
 \mathcal L^\alpha_{p,z}:=\left\{f\in A^p(\Omega): \partial^{(\beta)}f(z)=0,\forall\,\beta\prec \alpha\right\}
 $$
 is of finite-codimension in $A^p(\Omega)$.  Since $A^p(\Omega)$ is infinite-dimensional,   so the finite-codimensional subspace
 $$
 \bigcap_{k=1}^m\mathrm{Ker}F^\ast_k \,\cap \mathcal L^\alpha_{p,z}
 $$
 must contain a nonzero element $f$.  
 
 Given $g\in Y$ with $\|g\|_p=1$,   take $k$ such that $\|g-g_k\|_p<\varepsilon/2$.  We have
 \begin{eqnarray*}
 \|g+cf\|_p & \ge & \|g_k+cf\|_p-\varepsilon/2\\
 & \ge & F^\ast_k(g_k+cf)-\varepsilon/2\\
 & = & 1-\varepsilon/2\ge \frac1{1+\varepsilon},
 \end{eqnarray*}
 i.e.,  $\|g\|_p\le (1+\varepsilon)\|g+cf\|_p$.  Since the inequality is homogeneous,  we are done. 
 \end{proof}
 
 We also need the following useful criterion due to Banach.
 
 \begin{theorem}[cf.  \cite{Carothers},  Theorem 3.2]
 A sequence $\{x_j\}$ of nonzero vectors in a Banach space $X$ is a basic sequence if and only if there exists a constant $C>0$ such that
 $$
 \|c_1x_1+\cdots+c_l x_l\|\le C\|c_1x_1+\cdots + c_m x_m\|
 $$
 for all scalars $c_j$ and all $l<m$.
 \end{theorem} 
 
 \begin{proof}[Proof of Proposition \ref{prop:basic}]
 Given $\varepsilon>0$,   take a sequence of positive numbers $\{\varepsilon_k\}$ such that 
 $\prod_{k=1}^\infty (1+\varepsilon_k) \le 1+\varepsilon$.   First,  choose $\alpha^{j_1}=0$ and
 $f_{j_1}=m_p(\cdot,z)$.  By Lemma \ref{lm:basic},  there exist a multi-index $\alpha^{j_2}$ with $\alpha^{j_1}\prec \alpha^{j_2}$ and a nonzero function $f_{j_2}\in \mathcal H^{\alpha^{j_2}}_{p,z}$ such that 
 $$
 \|g\|_p\le (1+\varepsilon_1)\|g+cf_{j_2}\|_p,\ \ \ \forall\,g\in \mathrm{span}\{f_{j_1}\},\,c\in \mathbb C.
 $$
 Next,  choose $\alpha^{j_3}$ with $\alpha^{j_2}\prec \alpha^{j_3}$ and a nonzero function $f_{j_3}\in \mathcal H^{\alpha^{j_3}}_{p,z}$ such that 
 $$
 \|g\|_p\le (1+\varepsilon_2)\|g+cf_{j_3}\|_p,\ \ \ \forall\,g\in \mathrm{span}\{f_{j_1},f_{j_2}\},\,c\in \mathbb C.
 $$
 Continue this procedure,  we get a sequence $\{f_{j_k}\}$ with 
 $$
 \|c_1f_{j_1}+\cdots+c_l f_{j_l} \|_p \le C \|c_1f_{j_1}+\cdots + c_m f_{j_m}\|_p
 $$
 for all scalars $c_j$ and all $l<m$,  where $C:= \prod_{k=1}^\infty (1+\varepsilon_k) \le 1+\varepsilon$.  Thus $\{f_{j_k}\}$ is a basic sequence in view of Banach's criterion.
 \end{proof}
 
 \begin{problem}
 Does there exist a sequence $\{f_\alpha\}$ with $f_\alpha\in \mathcal H^\alpha_{p,z}$ such that it gives a basic sequence,  or even a basis?
 \end{problem}
 
 \begin{remark}
 The following high-order minimizing problem
$$
m_p^{(\alpha)}(z):=\inf\left\{\|f\|_p: f\in \mathcal H^\alpha_{p,z} \right\}
$$
also admits a unique  minimizer $m_p^{(\alpha)}(\cdot,z)$ for $1\le p<\infty$.  An elegant observation due to Bergman states that 
$\{m_2^{(\alpha)}(\cdot,z)/\|m_2^{(\alpha)}(\cdot,z)\|_2\}_\alpha$ is a complete orthonormal basis of $A^2(\Omega)$. 
\end{remark}
 
 \section{Appendix}

 {\bf Proof of (\ref{eq:eleIneq_1}).} A straightforward calculation shows
\begin{eqnarray*}
 (|b|^{p-2}+|a|^{p-2}) |b-a|^2 & = & |b|^p + |a|^p + |b|^{p-2}|a|^2+|a|^{p-2}|b|^2\\
&& -2{\rm Re}( |b|^{p-2}\bar{b}a)-2{\rm Re}( |a|^{p-2}\bar{a}b)\\
 (|b|^{p-2}-|a|^{p-2})(|b|^2-|a|^2) & = & |b|^p + |a|^p - |b|^{p-2}|a|^2 - |a|^{p-2}|b|^2.
\end{eqnarray*}
Summing up, we obtain the following basic equality:
\begin{eqnarray}\label{eq:identity}
&& (|b|^{p-2}+|a|^{p-2}) |b-a|^2 + (|b|^{p-2}-|a|^{p-2})(|b|^2-|a|^2)\nonumber\\
& = & 2|b|^p+2|a|^p-2{\rm Re}( |b|^{p-2}\bar{b}a)-2{\rm Re} (|a|^{p-2}\bar{a}b)\nonumber\\
& = & 2 {\rm Re}\left\{(|b|^{p-2}\bar{b}-|a|^{p-2}\bar{a})(b-a)\right\}.
\end{eqnarray}

For every $p\ge 2$ we have
$$
|b-a|^{p-2}\le (|a|+|b|)^{p-2}\le 2^{p-2} \max\{|a|^{p-2},|b|^{p-2}\}\le 2^{p-2}(|a|^{p-2}+|b|^{p-2}).
$$
This combined with (\ref{eq:identity}) gives (\ref{eq:eleIneq_1}).

{\bf Proof of (\ref{eq:eleIneq_2}).}
Let $1<p\le 2$.   The Newton-Lebnitz formula yields
\begin{eqnarray*}
|b|^{p-2}\bar{b}-|a|^{p-2}\bar{a} & = & \int_0^1 \frac{d}{dt} \left\{|a+t(b-a)|^{p-2}\cdot\overline{a+t(b-a)}\right\} dt\\
& = &(\bar{b}-\bar{a}) \int_0^1 |a+t(b-a)|^{p-2} dt\\
&& + (p-2) \int_0^1  |a+t(b-a)|^{p-4} {\rm Re}\left\{t|b-a|^2+a(\bar{b}-\bar{a})\right\}\overline{a+t(b-a)}dt,
\end{eqnarray*}
so that
\begin{eqnarray}\label{eq:Basic_Eq}
&& (|b|^{p-2}\bar{b}-|a|^{p-2}\bar{a})(b-a)\nonumber\\
& = & |b-a|^2 \int_0^1 |a+t(b-a)|^{p-2}\nonumber\\
&& +  (p-2)  \int_0^1  |a+t(b-a)|^{p-4} {\rm Re}\left\{t|b-a|^2+a(\bar{b}-\bar{a})\right\}\overline{t|b-a|^2+a(\bar{b}-\bar{a})}dt.
\end{eqnarray}
It follows that
\begin{eqnarray}\label{eq:Basic_Ineq}
&& {\rm Re} \left\{ (|b|^{p-2}\bar{b}-|a|^{p-2}\bar{a})(b-a)\right\}\nonumber\\
& = & |b-a|^2 \int_0^1 |a+t(b-a)|^{p-2}\nonumber\\
&& +  (p-2)  \int_0^1  |a+t(b-a)|^{p-4} \left|{\rm Re}\left\{t|b-a|^2+a(\bar{b}-\bar{a})\right\}\right|^2dt\nonumber\\
& = & (p-1) |b-a|^2 \int_0^1 |a+t(b-a)|^{p-2} dt \nonumber\\
&& +(2-p) \left|{\rm Im}\left\{a \bar{b} \right\}\right|^2 \int_0^1  |a+t(b-a)|^{p-4} dt
\end{eqnarray}
since
\begin{equation}\label{eq:eleIneq_0}
 \left({\rm Re}\{\bar{a}(b-a)\}+t|b-a|^2\right)^2 +\left({\rm Im}\{\bar{a}b\}\right)^2 =  |b-a|^2|a+t(b-a)|^2.
\end{equation}
\eqref{eq:Basic_Ineq} implies \eqref{eq:eleIneq_2} since
$$
|a+t(b-a)|=|(1-t)a+tb|\le |a|+|b|,\ \ \ 0\le t\le 1.
$$

{\bf Proof of (\ref{eq:eleIneq_4}).}
Set
\begin{eqnarray*}
\eta(t) &:=& |a+t(b-a)|^2=|a|^2+2t{\rm Re}\{\bar{a}(b-a)\}+t^2|b-a|^2\\
\kappa(t) &:=& \eta(t)^{p/2}=|a+t(b-a)|^p.
\end{eqnarray*}
A straightforward calculation yields
$$
\kappa'(t)=\frac{p}2\eta(t)^{\frac{p}2-1}\eta'(t)=p |a+t(b-a)|^{p-2} \left({\rm Re}\{\bar{a}(b-a)\}+t|b-a|^2\right)
$$
and
\begin{eqnarray*}
\kappa''(t) & = & \frac{p}2 \eta(t)^{\frac{p}2-1}\eta''(t)+\frac{p}2\left(\frac{p}2-1\right)\eta(t)^{\frac{p}2-2}\eta'(t)^2\\
& = & p|b-a|^2 |a+t(b-a)|^{p-2} \\
&& +p(p-2)|a+t(b-a)|^{p-4}\left({\rm Re}\{\bar{a}(b-a)\}+t|b-a|^2\right)^2.
\end{eqnarray*}
This combined with \eqref{eq:eleIneq_0} gives
\begin{eqnarray}\label{eq:ident_5}
\kappa''(t) & = & p|a+t(b-a)|^{p-4}\left({\rm Im}\{\bar{a}b\}\right)^2\nonumber\\
&& + p(p-1) |a+t(b-a)|^{p-4}\left({\rm Re}\{\bar{a}(b-a)\}+t|b-a|^2\right)^2.
\end{eqnarray}
Thus we have
\begin{eqnarray}\label{eq:KeyIneq_5}
\kappa''(t) & \ge & p\min\{1,p-1\}  |a+t(b-a)|^{p-4}\nonumber\\
&& \cdot \left\{\left({\rm Im}\{\bar{a}b\}\right)^2+\left({\rm Re}\{\bar{a}(b-a)\}+t|b-a|^2\right)^2 \right\}\nonumber\\
& = &  p\min\{1,p-1\}|b-a|^2 |a+t(b-a)|^{p-2}.
\end{eqnarray}
On the other hand, itegral by parts gives
$$
\kappa(1)=\kappa(0)+\kappa'(0)+ \int_0^1 (1-t)\kappa''(t)dt,
$$
that is,
\begin{equation}\label{eq:KeyIdent}
|b|^p  = |a|^p + p{\rm Re}\{|a|^{p-2}\bar{a}(b-a)\} +  \int_0^1 (1-t)\kappa''(t)dt.
\end{equation}
This combined with (\ref{eq:KeyIneq_5}) gives
\begin{eqnarray}\label{eq:Key_Ineq5}
|b|^p & \ge & |a|^p + p{\rm Re}\{|a|^{p-2}\bar{a}(b-a)\}\nonumber\\
&& +  p\min\{1,p-1\}|b-a|^2\int_0^1 (1-t) |a+t(b-a)|^{p-2}dt.
\end{eqnarray}

Since $|a+t(b-a)|\le |a|+|b|$,  we see that (\ref{eq:eleIneq_4}) follows from (\ref{eq:Key_Ineq5}).

{\bf Proof of (\ref{eq:eleIneq_3}).}
Suppose $p>2$. Note that
$$
I(t):=\int_0^1 (1-t) |a+t(b-a)|^{p-2}\ge \int_0^1 (1-t)||a|-t|b-a||^{p-2}.
$$
If $|a|\ge |b-a|/2$, then
$$
I(t)\ge |b-a|^{p-2} \int_0^{1/4}(1-t)(1/2-t)^{p-2}\ge \frac7{4^{p+3}} |b-a|^{p-2}.
$$
If $|a|\le |b-a|/2$, then
$$
I(t)\ge |b-a|^{p-2} \int_{3/4}^1 (1-t)(t-1/2)^{p-2}dt \ge \frac1{4^{p+3}} |b-a|^{p-2}.
$$
These combined with (\ref{eq:Key_Ineq5}) gives (\ref{eq:eleIneq_3}).

{\bf Proof of (\ref{eq:eleIneq_5}).} This follows directly from (\ref{eq:ident_5}) and (\ref{eq:KeyIdent}).

\bigskip

{\bf Acknowledgements.} The authors would like to thank  Yuanpu Xiong for a number of valuable discussions and the referee for valuable comments.

\end{document}